\documentclass[12pt,twoside]{amsart}
\usepackage{amssymb}
\usepackage{amscd}
\usepackage[abbrev,alphabetic]{amsrefs}
\usepackage{hyperref}
\usepackage{comment}
\usepackage{array,multirow,tabularx,longtable}
\usepackage{tikz}
\usetikzlibrary{cd}
\usepackage{here}
\usepackage{multirow}
\usepackage[margin=1.25in]{geometry}
\usepackage{mathtools}

\usepackage{framed}
\usepackage{fancybox}
\usepackage{ascmac}

\title[105 families of Fano threefolds in positive characteristic]
{105 families of Fano threefolds in positive characteristic} 

\thanks{
The  author is supported by JSPS KAKENHI Grant Numbers JP22H01112 and JP23K03028.}

\author{Hiromu Tanaka} 

\DeclareMathOperator{\Isom}{Isom}
\DeclareMathOperator{\h}{\text{-}}
\DeclareMathOperator{\all}{all}

\DeclareMathOperator{\Tot}{Tot}
\DeclareMathOperator{\cosp}{cosp}
\DeclareMathOperator{\Nef}{Nef}

\DeclareMathOperator{\Sym}{Sym}

\renewcommand{\sp}[0]{{\operatorname{sp}}}

\newcommand{\ol}{\overline}

\newcommand{\pr}[0]{{\operatorname{pr}}}

\newcommand{\Bl}[0]{{\operatorname{Bl}}}

\newcommand{\Hilb}[0]{{\operatorname{Hilb}}}

\newcommand{\NE}[0]{{\operatorname{NE}}}
\newcommand{\red}[0]{{\operatorname{red}}}

\newcommand{\Proj}[0]{\operatorname{Proj}}
\newcommand{\Spec}[0]{\operatorname{Spec}}

\newcommand{\Hom}[0]{{\operatorname{Hom}}}

\newcommand{\Pic}[0]{\operatorname{Pic}}

\newcommand{\Ex}[0]{{\operatorname{Ex}}}

\newcommand{\Aut}[0]{{\operatorname{Aut}}}

\newcommand{\Gr}[0]{{\operatorname{Gr}}}

\newcommand{\PGL}[0]{{\operatorname{PGL}}}
\newcommand{\GL}[0]{{\operatorname{GL}}}

\newtheorem{thm}{Theorem}[section]
\newtheorem{lem}[thm]{Lemma}
\newtheorem*{lem*}{Lemma}

\newtheorem{cor}[thm]{Corollary}

\newtheorem{prop}[thm]{Proposition}

\newtheorem*{claim*}{Claim}         
\newtheorem{step}{Step}

\theoremstyle{definition}

\newtheorem{dfn}[thm]{Definition}

\newtheorem{rem}[thm]{Remark}

\makeatletter
  
  \@addtoreset{equation}{thm}
  \makeatother

\newcommand{\bB}{\mathbb{B}}

\newcommand{\bG}{\mathbb{G}}

\newcommand{\cA}{\mathcal{A}}

\newcommand{\cC}{\mathcal{C}}
\newcommand{\cD}{\mathcal{D}}
\newcommand{\cE}{\mathcal{E}}
\newcommand{\cF}{\mathcal{F}}

\newcommand{\cH}{\mathcal{H}}
\newcommand{\cI}{\mathcal{I}}

\newcommand{\cL}{\mathcal{L}}
\newcommand{\cM}{\mathcal{M}}

\newcommand{\cO}{\mathcal{O}}
\newcommand{\cP}{\mathcal{P}}

\newcommand{\cX}{\mathcal{X}}
\newcommand{\cY}{\mathcal{Y}}
\newcommand{\cZ}{\mathcal{Z}}

\newcommand{\MO}{\mathcal{O}}

\newcommand{\R}{\mathbb{R}}
\newcommand{\Q}{\mathbb{Q}}
\newcommand{\Z}{\mathbb{Z}}
\newcommand{\F}{\mathbb{F}}

\renewcommand{\P}{\mathbb{P}}

\newcommand{\m}{\mathfrak{m}}

\newcommand{\la}{\langle}
\newcommand{\ra}{\rangle}

\DeclareMathOperator{\Fl}{Fl}

\usepackage{listings}
\begin{document}

\maketitle

\begin{abstract}
Over an algebraically closed field of characteristic $p>0$, 
we prove that smooth Fano threefolds form 105 irreducible deformation families. 
\end{abstract}



\tableofcontents

\section{Introduction}

The classification problem of Fano threefolds goes back to the work of Gino Fano. 
In characteristic zero, Fano threefolds were classified by Mori and Mukai 
\cite{MM81}, \cite{MM83}, \cite{MM03}, 
following fundamental work of Iskovskikh and Shokurov \cite{Isk77}, \cite{Isk78}, \cite{Sho79a}, \cite{Sho79b}. 
Together, these results show that there are exactly $105$ deformation families of Fano threefolds \cite{MM85}.

In positive characteristic, a classification of Fano threefolds has been established in \cite{FanoIV}, 
where the classification list is identical to the one in characteristic zero. 
However, it remained to show that each number corresponds to a single deformation family.
The main result of this paper fills this gap.

\begin{thm}[Theorem \ref{t parameter final}]\label{intro main}
Let $k$ be an algebraically closed field of characteristic $p>0$. 
Then there exist exactly 105 deformation families of Fano threefolds over $k$. 
Specifically, given an arbitrary number x-yz listed in \cite[Section 7]{FanoIV}, 
there exists a smooth projective $k$-morphism  $\pi : \cX \to U$ 
of integral quasi-projective schemes over $k$ such that 
\begin{itemize}
\item the fibre of $\pi$ over every closed point is a Fano threefold of No.\ x-yz, and 
\item every Fano threefold $X$ of No. x-yz is isomorphic to  the fibre of $\pi$ over some closed point. 
\end{itemize}
\end{thm}

Along the way, we also establish the following base point free theorem 
for families of Fano threefolds: 

\begin{thm}[Corollary \ref{c relative semi-ample}]\label{intro bpf}
Let $U$ be a connected noetherian scheme and 
let $\pi : \cX \to U$ be a smooth projective morphism such that every geometric fibre is a Fano threefold. 
Take an invertible sheaf $\cL$ on $\cX$. 
Then $\cL$ is $\pi$-semi-ample if and only if $\cL_u$ is nef for some point $u \in U$. 
\end{thm}

\subsection{Proof of Theorem \ref{intro bpf}}

A key auxiliary result is a relative base point free theorem for families of Fano threefolds (Theorem \ref{intro bpf}). Once this is established, the nef cones and the corresponding extremal contractions can be controlled uniformly in families, which leads to the deformation invariance of the classification number.

In what follows, we overview some ideas of 
the proof of the base point free theorem (Theorem \ref{intro bpf}). 
By a rigidity argument, the relative statement is reduced to the absolute case, i.e., $U = \Spec k$.  
Thus it remains to prove that every nef Cartier divisor $L$ on a Fano threefold $X$ is semi-ample. 
The proof consists of the following three steps. 
First, using the liftability of Fano threefolds and Kodaira vanishing, we prove that $L$ is abundant. 
Second, when $k=\overline{\F}_p$, the semi-ampleness of $L$ follows from 
\cite[Theorem 0.5]{Kee99} ($\kappa(L)=3$) 
and the finite generation of section rings for the surface case. 
Finally, for an arbitrary algebraically closed field $k$, we spread out $(X,L)$ over a variety defined over $\overline{\F}_p$ and deduce the semi-ampleness of $L$ from the case when $k=\overline{\F}_p$.






\subsection{Proof of Theorem \ref{intro main}}

Sections 4--6 are devoted to the existence of Fano threefolds of all the numbers in the classification. In Section 4, we establish several general criteria which will be used repeatedly. In Section 5, we treat Fano threefolds obtained as blowups of explicit varieties, such as P
3
 and del Pezzo threefolds. The existence for the remaining cases is treated in Section 6. Based on the existence result, we prove Theorem \ref{intro main} in Sections 7--9.

The existence argument largely follows Mori--Mukai in characteristic zero \cite{MM85}. The main differences in positive characteristic are the use of the prime divisor criterion, the treatment of double covers in characteristic two, and a separate construction for No.\ 3-2. The irreducibility part also follows the general strategy of Mori--Mukai; we give the details needed to carry out their argument in positive characteristic.

\medskip
\noindent {\bf Acknowledgements.}
The author was supported by JSPS KAKENHI Grant number JP22H01112 and JP23K03028. 
ChatGPT (OpenAI) was used during the preparation of this article for assistance with mathematical exploration, computations, and language editing.
\section{Preliminaries}

\subsection{Notation}\label{ss-notation}

In this subsection, we summarise notation used in this paper. 

\begin{enumerate}
\item We will freely use the notation and terminology in \cite{Har77} and \cite{KM98}. 
In particular, $D_1 \sim D_2$ means linear equivalence of Weil divisors. 
\item 
Throughout this paper, 
we work over an algebraically closed field $k$ 
of characteristic $p>0$ unless otherwise specified. 
\item For an integral scheme $X$, 
we define the {\em function field} $K(X)$ of $X$ 
as the local ring $\MO_{X, \xi}$ at the generic point $\xi$ of $X$. 
For an integral domain $A$, $K(A)$ denotes the function field of $\Spec\,A$. 

\item 
For a scheme $X$, its {\em reduced structure} $X_{\red}$ 
is the reduced closed subscheme of $X$ such that the induced closed immersion 
$X_{\red} \to X$ is surjective. 
\item 
Our notation will not distinguish between invertible sheaves and 
Cartier divisors. For example, we will write $L+D$ for 
an invertible sheaf $L$ and a Cartier divisor $D$. 
\item We say that $X$ is a {\em variety} (over $k$) if 
$X$ is a separated integral scheme which is of finite type over $k$. 
We say that $X$ is a {\em curve} (resp. a {\em surface}, resp. a {\em threefold})  
if $X$ is a variety over $k$ of dimension one (resp. {\em two}, resp. {\em three}). 
\item 
Given a variety $Y$ and a closed subscheme $Z$ of $Y$, 
$\Bl_Z\,Y$ denotes the blowup of $Y$ along $Z$. 
In this case, $Z$ is called {\em the (blowup) centre} of the induced blowup $\Bl_Z\,Y \to Y$. 
Let  $\Ex(f)$ be the exceptional divisor 
equipped with reduced scheme structure. 
In particular, if $Y$ is a smooth threefold and $Z$ is a smooth curve on $Y$, 
then we have $K_X \sim f^*K_Y +\Ex(f)$ for $X := \Bl_Z\,Y$.
\item We say that $f: X \to Y$ is a {\em contraction} if $f$ is a projective morphism of schemes satisfying $f_*\MO_X = \MO_Y$. 
Here the equality $f_*\MO_X = \MO_Y$ means that the induced ring homomorphism $\MO_Y \to f_*\MO_X$ is an isomorphism. 

\item 
For the definition of types of extremal rays for smooth projective threefolds, 
we refer to \cite[Definition 3.3]{FanoIII}. 
\item $Q$ denotes a smooth quadric hypersurface on $\P^4$. 
For $1 \leq d \leq 5$ or $d =7$, let $V_d$ be a Fano threefold of index two such that $(-K_{V_d}/2)^3 =d$. 
Let $W$ be a smooth divisor on $\P^2 \times \P^2$ of bidegree $(1, 1)$. 
Note that such a threefold $W$ is unique up to isomorphisms \cite[Lemma 5.16]{FanoIII}. 
\item For a Fano threefold $Y$ with $\rho(Y)=1$, 
$\MO_Y(1)$ denotes an invertible sheaf which  generates $\Pic\,Y (\simeq \Z)$ and 
we set $\MO_Y(\ell) := \MO_Y(1)^{\otimes \ell}$. 
\item 
\begin{itemize}
\item We say that a divisor $D$ on $\P^a \times \P^b \times \P^c$ 
is of {\em tridegree} $(d_1, d_2, d_3)$ if 
$\MO_{\P^a \times \P^b \times \P^c}(D) \simeq \MO_{\P^a \times \P^b \times \P^c}(d_1, d_2, d_3)$. 
We define the {\em bidegree} of a divisor on $\P^a \times \P^b$ in a similar way. 
\item 
We say that a curve $B$ on $\P^a \times \P^b \times \P^c$ 
is of {\em tridegree} $(d_1, d_2, d_3)$ if 
$\pr^*_i\MO(1) \cdot B = d_i$ for every $i \in \{1, 2, 3\}$. 
When $(a, b) \neq (1, 1)$, 
we define the {\em bidegree} of a curve on $\P^a \times \P^b$ in a similar way.   
\end{itemize}
\item 
We say that x-yz is a {\em number} 
if x-yz is one of the numbers listed in \cite[Section 7]{FanoIV}. 
For the definition of  Fano threefolds $X$ {\em of No.\ x-yz}, 
we refer to \cite[Section 7]{FanoIV}. 
\item 
Fix a number x-yz. 
We say that $\pi : \cX \to U$ is a {\em parameter space} of No.\ x-yz if  $\pi$ is a smooth projective morphism of quasi-projective varieties over $k$ such that 
\begin{itemize}
\item the fibre of $\pi$ over every closed point is a Fano threefold, and 
\item every Fano threefold $X$ of No. x-yz is  isomorphic to  the fibre of $\pi$ over some closed point. 
\end{itemize}
In this case, we will show that the fibre of $\pi$ over every closed point is a Fano threefold of No.\ x-yz (Theorem \ref{t No inv}). 
We say that $U$ is a {\em parameter space} of No. x-yz 
if there exists a parameter space $\pi : \cX \to U$ of No. x-yz. 
\end{enumerate}

\begin{rem}\label{r curve bidegree}
Let $f : \P^1_1 \times \P^1_2 \times \P^1_3 \to \P^1_1 \times \P^1_2$ 
be the projection onto the first and second direct product factors. 
Take a curve $B$ on $\P^1_1 \times \P^1_2 \times \P^1_3$ such that 
the induced morphism $B \to B' := f(B)$ is an isomorphism. 
If $B$ is of tridegree $(d_1, d_2, d_3)$, then $B'$ is of bidegree $(d_2, d_1)$. 
\end{rem}

\subsection{Fano threefolds}

 We say that $X$ is a {\em Fano threefold}  
if $X$ is a three-dimensional smooth projective variety over $k$ such that $-K_X$ is ample. 
For a Fano threefold $X$, 
the following holds: 
\begin{enumerate}
\item 
$1 \leq r_X \leq 4$ for the {\em index} $r_X$ of $X$ 
is the largest positive integer $r$ 
that divides $-K_X$ in $\Pic\,X$ \cite[Theorem 2.18]{FanoI}. 
\item $\Pic X \simeq \Z^{\oplus \rho(X)}$ \cite[Theorem 2.4]{FanoI}. 
\item 
$h^0(X, -mK_X) = \chi(X, -mK_X) = \frac{1}{12} m(m+1)(2m+1)(-K_X)^3 + 2m + 1$ by \cite[Corollary 2.6]{FanoI} and Theorem \ref{t Fano3 ANV} below. 
\end{enumerate}


\begin{thm}\label{t Fano3 ANV}
Let $X$ be a Fano threefold and let $A$ be an ample Cartier divisor. 
Then $H^i(X, \Omega_X^j(A))= 0$ for every $i+j >3$. 
\end{thm}

\begin{proof}
The assertion follows from \cite{KTLift1} and \cite{KTLift2}. 
\end{proof}

\section{Families of Fano threefolds}

In \cite{MM85}, the numbers of Fano threefolds in characteristic zero are characterised by 
the triple $(\Pic X, K_X, \mu_X)$, where 
\[
\mu_X : \Pic X \times \Pic X \times \Pic X \to \Z
\]
denotes the induced $\Z$-module homomorphism. 
On the other hand, the numbers of Fano threefolds in positive characteristic is defined by the data of extremal rays \cite{FanoIV}. 
The purpose of this section is to prove that these two notions coincide. 
To this end, we extend the rigidity of Kleiman-Mori cones 
\cite{Wis09} to families of Fano threefolds in arbitrary characteristic.  
Furthermore, we prove the base point free theorem for Fano threefolds. 

\subsection{Families of Fano threefolds over CDVR}

\begin{prop}\label{p DVR rho=1 equiv}
Let $(R, \m)$ be a complete discrete valuation ring 
such that $R/\m$ is algebraically closed. 
Let
$\pi\colon
\cX \to \Spec R$ be a smooth projective morphism 
whose geometric fibres are Fano threefolds. 
Take an invertible sheaf $\cL$ on $\cX$. 
Set $s :=\m$. 
Then the following hold: 
\begin{enumerate}
\item $\cL_s$ is nef if and only if $\cL$ is $\pi$-nef. 
\item $\cL_s$ is semi-ample if and only if $\cL$ is $\pi$-semi-ample. 
\item If $\cL$ is $\pi$-semi-ample and $\cX \to \cY$ is the contraction 
induced by $\cL$, 
then the base change $\cX_t \to \cY_t$ is a contraction for every geometric point $t \to \Spec R$. 
\item 
If $\cL$ is $\pi$-semi-ample, 
$\cX \to \cY$ is the contraction induced by $\cL$, and $\rho(\cX_s/\cY_s)=1$, 
then $\rho(\cX_{\ol{\xi}}/\cY_{\ol{\xi}})=1$. 
\end{enumerate}
\end{prop}

\begin{proof}
The assertion (1) follows from \cite[Lemma 2.6]{CT20}. 

Let us show (2) and (3). 
Assume that $\cL_s$ is semi-ample. 
It is enough to prove that 
$H^0(\cX, \cL^m) \to H^0(\cX_s, \cL^m_s)$ 
is surjective for every integer $m \geq 0$, 
which follows from the Kodaira vanishing for $\cX_s$ (Theorem \ref{t Fano3 ANV}). 
Thus (2) and (3) hold.

Let us show (4). 
For a Cartier divisor $\cD$ on $\cX$ and a curve $C$ on $X_\xi$ contracted by $f_\xi$, 
if $\cD_{\xi} \cdot C=0$, then 
$\cD_s \equiv 0$, which implies that 
$\cD_{\xi} \equiv 0$. 
Therefore, $\rho(\cX_{\ol{\xi}}/\cY_{\ol{\xi}})\leq 1$. 

Thus it is enough to show that $X_\xi \to Y_\xi$ is not an isomorphism, i.e., 
$\cL_\xi$ is not ample. 
If $\cL_s$ is not big, then we have $\cL_{\xi}^3 = \cL_s^3 =0$, and hence 
$\cL_{\xi}$ is not ample. 
Assume that $\cL_s$ is big. 
In this case, $f_s$ is a birational contraction of type $E_1-E_5$.

We now show that $H^1(\cX_s, E)=0$ for $E := \Ex(f_s)$. 
If $E$ is of type $E_1$, then this follows from the exact sequence 
\[
0 \to H^1(\cY_s, (f_s)_*\MO_{\cX_s}(E)) \to 
H^1(\cX_s, \MO_{\cX_s}(E)) \to 
H^0(\cY_s, R^1(f_s)_*\MO_{\cX_s}(E)), 
\]
together with $R^1(f_s)_*\MO_{\cX_s}(E)=0$ and 
\[
H^1(\cY_s, (f_s)_*\MO_{\cX_s}(E)) =
H^1(\cY_s, \MO_{\cY_s}) \hookrightarrow H^1(\cX_s, \MO_{\cX_s})=0. 
\]
If $E$ is of type $E_2-E_5$, then $H^1(E, \MO_{\cX_s}(E)|_E)=0$ 
\cite[Theorem 1.1]{Kol91}. 
This, together with $H^1(\cX_s, \MO_{\cX_s})=0$, implies $H^1(\cX_s, E)=0$ by using an exact sequence 
\[
0 \to \MO_{\cX_s} \to \MO_{\cX_s}(E) \to \MO_{\cX_s}(E)|_E \to 0. 
\]
This completes the proof of $H^1(\cX_s, E)=0$.

For a lift $\cD$ of $E := \Ex(f_s)$ as a Cartier divisor, 
the Nakayama's lemma implies that $H^1(\cX, \cD)=0$. 
By using the exact sequence 
\[
0 \to \MO_{\cX}(\cD) \to \MO_{\cX}(\cD) \to \MO_{\cX_s}(E) \to 0,
\] 
we can find an effective Cartier divisor $\cE$ on $\cX$ satisfying 
$\cD \sim \cE$ and $\cE|_{\cX_s} = E$. 
Then we obtain 
\[
(\cL|_{\cX_\xi})^2 \cdot ( \cE_\xi) = 
(\cL|_{\cX_s})^2 \cdot ( \cE_s) =  (\cL|_{\cX_s})^2 \cdot E =0, 
\]
which implies that $\cL|_{\cX_\xi}$ is not ample. 
Thus (4) holds. 
\end{proof}

\begin{rem}\label{r sp map isom}
Let $(R, \m)$ be a complete discrete valuation ring 
such that $R/\m$ is algebraically closed. 
Let
$\pi\colon
\cX \to \Spec R$ be a smooth projective morphism 
whose geometric fibres are Fano threefolds. 
For the closed point $s$ and geometric generic point $\ol{\xi}$ of $\Spec R$, we have the specialisation maps $\sp^1: N^1(X_{\ol{\xi}}) \to N^1(X_s)$ and 
$\sp_1 : N_1(X_{\ol{\xi}}) \to N_1(X_s)$, where $\sp^1$ is an isomorphism. 
By compatibility with the intersection pairing,
the cospecialisation map 
\[
\cosp: N_1(X_s) \simeq N^1(X_s)^* \xrightarrow{(\sp^1)^*, \simeq} N^1(X_{\ol{\xi}})^* \simeq N_1(X_{\ol{\xi}})
\]
is the inverse of $\sp_1$, i.e., $\sp_1 = (\cosp)^{-1}$. 
In particular, $\sp_1 : N_1(X_{\ol{\xi}}) \to N_1(X_s)$ is an isomorphism. 
\end{rem}

\begin{lem}\label{l ray bpf}
Let $X$ be a Fano threefold and let 
$M$ be a nef Cartier divisor such that $\NE(X) \cap M^{\perp} =:R$ is an extremal ray. 
Then $M$ is semi-ample.
\end{lem} 

\begin{proof}
Let $f: X \to Y$ be the contraction of $R$. 
Fix a curve $\ell$ on $X$ such that $f(\ell)$ is a point. 
Then the sequence 
\[
0 \to \Pic Y \xrightarrow{f^*} \Pic X \xrightarrow{(-) \cdot \ell} \to \Z
\]
is exact \cite[Proposition 3.12]{FanoII}. 
Therefore, we can find a Cartier divisor $M_Y$ on $Y$ satisfying $M \sim f^*M_Y$. 
It is easy to see that $M_Y \cdot C >0$ for every curve $C$ on $Y$. 
By $\NE(Y) = f_*(\NE(X))$, 
$\NE(Y)$ is a rational polyhedral cone. 
It follows from Kleiman's criterion that $M_Y$ is ample. 
Therefore, $M \sim f^*M_Y$ is semi-ample. 
\end{proof}

\begin{prop}\label{p DVR NE Nef}
Let $(R, \m)$ be a complete discrete valuation ring 
such that $R/\m$ is algebraically closed. 
Let 
$\pi\colon
\cX \to \Spec R$ be a smooth projective morphism 
whose geometric fibres are Fano threefolds. 
Then the following hold: 
\begin{enumerate}
\item $\NE(X_{\ol{\xi}}) \xrightarrow{\simeq, \sp_1} \NE(X_s)$. 
\item $\Nef(X_{\ol{\xi}}) \xrightarrow{\simeq, \sp^1} \Nef(X_s)$. 
\end{enumerate}
\end{prop}

\begin{proof}
By Remark \ref{r sp map isom}, the specialisation maps are linear isomorphisms: 
\[
\sp^1\colon
N^1(X_{\ol{\xi}})
\xrightarrow{\simeq}
N^1(X_s), \qquad 
\sp_1\colon
N_1(X_{\ol{\xi}})
\xrightarrow{\simeq}
N_1(X_s). 
\]
As $\sp^1$ and $\sp_1$ are compatible with the intersection pairing, 
we get 
\[
\sp^1(D)\cdot\sp_1(\alpha)
=
D\cdot\alpha
\]
for every $D\in N^1(X_{\ol{\xi}})$ and $\alpha\in N_1(X_{\ol{\xi}})$. 
Since the specialisation of an effective $1$-cycle on $X_{\ol{\xi}}$ is an effective $1$-cycle on $X_s$, we get 
\[
\sp_1
\left(
\NE(X_{\ol{\xi}})
\right)
\subset
\NE(X_s).
\]

Let us prove the opposite inclusion. 
As $X_s$ is Fano, the cone $\NE(X_s)$ is rational polyhedral. 
Fix an extremal ray $R_s$ of $\NE(X_s)$. 
By Lemma \ref{l ray bpf}, we can find a semi-ample  Cartier divisor $M_s$ on $X_s$ such that
\[
\NE(X_s)\cap M_s^\perp
=
R_s.
\]
As $R$ is complete and $H^{>0}(\MO_{X_s})=0$, the invertible sheaf $\mathcal O_{X_s}(M_s)$ extends to an invertible sheaf $\cM$ on $\cX$. 
By Proposition \ref{p DVR rho=1 equiv}, 
$\cM$ is $\pi$-semi-ample. 
Let
$f\colon
\cX
\longrightarrow
\cY$ be the contraction induced by $\cM$.
The induced morphism
$f_s\colon
X_s
\longrightarrow
\cY_s$ is the contraction of $R_s$ (Proposition \ref{p DVR rho=1 equiv}). 
In particular, $\rho(X_s/\cY_s)=1.$ 
Again by Proposition \ref{p DVR rho=1 equiv}, $\rho(X_{\ol{\xi}}/\cY_{\ol{\xi}})=1$. 
Therefore
\[
R_{\ol{\xi}}
:=
\NE(\cX_{\ol{\xi}})
\cap
\cM_{\ol{\xi}}^\perp
\]
is an extremal ray of $\NE(X_{\ol{\xi}})$. 

By compatibility with the intersection pairing,
\[
\sp_1(R_{\ol{\xi}})
\subset
\NE(X_s)\cap M_s^\perp
=
R_s.
\]
Since $\sp_1$ is a linear isomorphism and $R_{\ol{\xi}}$ is nonzero, the image $\sp_1(R_{\ol{\xi}})$ is a nonzero ray. Hence
$\sp_1(R_{\ol{\xi}})
=
R_s.$ 
Thus every extremal ray of $\NE(X_s)$ is the image of an extremal ray of $\NE(X_{\ol{\xi}})$. Since $\NE(X_s)$ is generated by its extremal rays, we obtain the opposite inclusion 
\[
\NE(X_s)
\subset
\sp_1
\left(
\NE(X_{\ol{\xi}})
\right).
\]
Thus (1) holds. 
The assertion (2) follows from (1). 
\qedhere



\end{proof}

\begin{cor}\label{c relative nef}
Let $U$ be a connected noetherian scheme and 
let $\pi : \cX \to U$ be a smooth projective morphism such that every geometric fibre is a Fano threefold. 
Take an invertible sheaf $\cL$ on $\cX$. 
Then the following hold: 
\begin{enumerate}
\item $\cL$ is $\pi$-nef if and only if $\cL_u$ is nef for some point $u \in U$. 
\item $\cL$ is $\pi$-ample if and only if $\cL_u$ is ample for some point $u \in U$. 
\end{enumerate}
\end{cor}

\begin{proof}
By standard argument, we may assume $U = \Spec R$ for a complete discrete valuation ring $R$ whose residue field is algebraically closed. 
Then the assertion (1) follows from Proposition \ref{p DVR NE Nef}. 
The ample cone is the interior of the nef cone. Since $\sp^1$ is a linear isomorphism identifying the two nef cones, it also identifies their interiors. Hence $\cL_{\ol{\xi}}$ is ample if and only if $\cL_s$ is ample. 
Thus (2) holds. 
\end{proof}

\subsection{Base point free theorem}

\begin{lem}\label{l abundant nefbig}
Let $X$ be a projective normal variety. 
Let $L$ be a nef Cartier divisor. 
Assume that $L$ is abundant, i.e., $\kappa(X, L) = \nu(X, L)$. 
Then there exist morphisms 
$\mu : X' \to X$ and $\pi : X' \to T$ 
of projective normal varieties 
such that $\dim X = \dim X'$ and 
$\mu^*L \sim_{\Q} \pi^*M$ for some nef and big $\Q$-Cartier $\Q$-divisir with $\dim T = \kappa(X, L) = \nu(X, L)$.
\end{lem}

\begin{proof}
By using de Jong's alteration \cite{dJ96}, 
the same argument as in \cite[Proposition 2.1]{Kaw85} works. 
\end{proof}

\begin{thm}\label{t Fano3 bpf thm}
Let $X$ be a Fano threefold and let $L$ be a nef Cartier divisor. 
Then $L$ is semi-ample. 
\end{thm}

\begin{proof}

\setcounter{step}{0}
\begin{step}\label{s1 Fano3 bpf thm}
$L$ is abundant, i.e., $\kappa(X, L) = \nu(X, L)$. 
\end{step}

\begin{proof}[Proof of Step \ref{s1 Fano3 bpf thm}]
We can find a finite extension $R$ of $W(k)$ such that 
there exists an $R$-lift $(\cX, \cL)$ of $(X, L)$ \cite{KTLift1}, \cite{KTLift2}. 
Let $(X_0, L_0)$ be the geometric generic fibre of $(\cX, \cL) \to \Spec R$. 
Fix an ample Cartier divisor $\cH$ on $\cX$. 
For a geometric point $t$, we have 
\[
\nu(\cX_t, \cL_t) = \max \{ i \,|\, \cL_t^i \cdot \cH_t^{3-i} >0\}. 
\]
Therefore, $\nu(X, L) = \nu(X_0, L_0)$. 
Thus it suffices to show $\kappa(X, L) = \kappa(X_0, L_0)$, 
which follows from the equalities 
\[
h^0(X, mL) = \chi(X, mL) = \chi(X_0, mL_0) = h^0(X_0, mL_0)
\]
for $m \geq 0$, where $H^{>0}(X, mL)=0$ is insured by Theorem \ref{t Fano3 ANV}. 
This completes the proof of 
Step \ref{s1 Fano3 bpf thm}. 
\end{proof}

\begin{step}\label{s2 Fano3 bpf thm}
The assertion holds if $k =\overline{\F}_p$. 
\end{step}

\begin{proof}[Proof of Step \ref{s2 Fano3 bpf thm}]
If $L$ is big, then the assertion holds by \cite[Theorem 0.5]{Kee99}. 
In what follows, we assume $\kappa(X, L) =\nu(X, L) < 3$. 
Since $L$ is nef and abundant (Step \ref{s1 Fano3 bpf thm}), 
we can find morphisms $\mu : X' \to X$ and $\pi : X' \to T$ 
and a nef and big $\Q$-Cartier $\Q$-divisor $M$ on $T$ as in Lemma \ref{l abundant nefbig}. 
By $k =\overline{\F}_p$ and $\dim T =\kappa(X, L) \leq 2$, $M$ is semi-ample \cite{Kee99}. 
This, together with $\pi^*M \sim_{\Q} \mu^*L$, implies that $\mu^*L$ is semi-ample, and hence so is $L$. 
This completes the proof of 
Step \ref{s2 Fano3 bpf thm}. 
\qedhere



\end{proof}

\begin{step}\label{s3 Fano3 bpf thm}
The assertion holds without any additional assumptions. 
\end{step}

\begin{proof}[Proof of Step \ref{s3 Fano3 bpf thm}]
There exists a smooth projective morphism $\pi : \cX \to T$ 
and a Cartier divisor $\cL$ on $\cX$ such that 
$T = \Spec R$ is a smooth affine variety over $k =\overline{\F}_p$, every geometric fibre is a Fano threefold, and $(X, L)$ is a base change of the generic fibre of $(\cX, \cL) \to T$. 
It suffices to show that $\cL$ is $\pi$-semi-ample. 
Recall that $\cL$ is $\pi$-nef (Corollary \ref{c relative nef}). 
Fix a closed point $t \in T$. 
By Theorem \ref{t Fano3 ANV}. the restriction map 
\[
H^0(\cX, m \cL) \to H^0(\cX_t, m\cL_t)
\]
is surjective for every $m \geq 0$. 
Since $\cL_t$ is semi-ample by Step \ref{s2 Fano3 bpf thm}, 
$\cL$ is $\pi$-semi-ample around the fibre $\cX_t$. 
As $t$ is chosen to be an arbitrary closed point, 
$\cL$ is $\pi$-semi-ample. 
This completes the proof of 
Step \ref{s3 Fano3 bpf thm}. 
\end{proof}
\end{proof}

\begin{cor}\label{c rel pic vs abs pic}
Let $X$ be a Fano threefold and let $f:X \to Y$ be a contraction to a projective normal variety $Y$. 
Then $\rho(X/Y) = \rho(X) -\rho(Y)$. 
\end{cor}

\begin{proof}
It suffices to show that the induced  sequence 
\[
0 \to N^1(Y) \xrightarrow{f^*} N^1(X) \to  N^1(X/Y) \to 0
\]
is exact. 
Pick a Cartier divisor $L$ on $X$ such that $L \equiv_f 0$. 
It is enough to prove that $L \sim_\Q f^*L_Y$ for some $\Q$-Cartier $\Q$-divisor. 
Fix an ample Cartier divisor $A_Y$ on $Y$. 
For $n \gg0$, the cone theorem implies that $L + nf^*A_Y$ is nef. 
Then $L+ nf^*A_Y$ is semi-ample by Theorem \ref{t Fano3 bpf thm}. 
As $L+nf^*A_Y \equiv_f 0$ and $f_*\MO_X = \MO_Y$, 
we can find a $\Q$-Cartier $\Q$-divisor $M_Y$ on $Y$ satisfying 
$L+nf^*A_Y \sim_\Q f^* M_Y$. 
Then the assertion holds by setting $L_Y := M_Y - nA_Y$. 
\end{proof}

\begin{cor}\label{c relative semi-ample}
Let $U$ be a connected noetherian scheme and 
let $\pi : \cX \to U$ be a smooth projective morphism such that every geometric fibre is a Fano threefold. 
Take an invertible sheaf $\cL$ on $\cX$. 
Then $\cL$ is $\pi$-semi-ample if and only if $\cL_u$ is nef for some point $u \in U$. 
\end{cor}

\begin{proof}
If $\cL$ is $\pi$-semi-ample, then $\cL_u$ is nef for every point $u \in U$. 
Conversely, assume that $\cL_u$ is nef for some point $u \in U$. 
Then $\cL$ is $\pi$-nef by Corollary \ref{c relative nef}. 
In order to show that $\cL$ is $\pi$-semi-ample, 
we may assume that $U = \Spec R$ for a local ring $(R, \m)$. 
By Theorem \ref{t Fano3 bpf thm}, $\cL_\m$ is semi-ample. 
This, together with the surjectivity of 
\[
H^0(\cX, n \cL) \to H^0(\cX_\m, n\cL_\m), 
\]
implies that $\cL$ is $\pi$-semi-ample. 
\end{proof}

\subsection{Rigidity of Fano numbers}

\begin{prop}\label{p cont bc}
Let $U$ be a connected noetherian scheme and 
let $\pi : \cX \to U$ be a smooth projective morphism such that every geometric fibre is a Fano threefold. 
Take a $\pi$-semi-ample invertible sheaf $\cL$ on $\cX$ and 
let $f: \cX \to \cY$ be the contraction induced by $\cL$. 
Fix two geometric points $\ol{u}, \ol{v}$ of $U$. 
Then the following hold: 
\begin{enumerate}
\item 
For a morphism $U' \to U$ from a noetherian scheme $U'$ 
and the induced diagram 
\[
\begin{tikzcd}
\cX' \arrow[r, "f'"] \arrow[d, "\alpha"] & \cY' \arrow[r] \arrow[d, "\beta"] & U' \arrow[d, "\gamma"] \\
\cX \arrow[r, "f"] & \cY \arrow[r] & U
\end{tikzcd}
\]
in which each square is cartesian, 
$f'$ coincides with the contraction induced by the pullback $\cL'$ of $\cL$ on $\cX'$. 
\item $\cY \to U$ is flat. In particular, $\dim \cY_{\ol{u}} = \dim \cY_{\ol{v}}$. 
\item $\rho(X_{\ol{u}}/Y_{\ol{u}}) = \rho(X_{\ol{v}}/Y_{\ol{v}})$. 
\end{enumerate}
\end{prop}

\begin{proof}
Let us show (1). 
Replacing $\cL$ by a sufficiently divisible positive tensor power, 
we may assume that $\pi^*\pi_*\cL \to \cL$ is surjective. 
We have 
\[
\cY
 = 
\Proj_U(\cA) \qquad \text{for} \qquad 
\cA
:=
\bigoplus_{n\geq 0}
\pi_*\cL^{\otimes n}.
\]
For every $n \geq 0$, 
both $R\pi_*\cL^{\otimes n} =\pi_*\cL^{\otimes n}$ 
and  $R\pi'_*\cL'^{\otimes n} =\pi'_*\cL'^{\otimes n}$ are 
locally free sheaves. Hence we get 
\[
\gamma^*\pi_*\cL^{\otimes n} \simeq 
L\gamma^*\pi_*\cL^{\otimes n} \simeq 
L\gamma^*R\pi_*\cL^{\otimes n} \overset{(\star)}{\simeq}  
R\pi'_*L\alpha^*\cL^{\otimes n} \simeq R\pi'_*\cL'^{\otimes n} \simeq \pi'_*\cL'^{\otimes n} 
\]
where $(\star)$ follows from \cite[Tag 0B91]{SP}. 
Then we obtain 
\[
\cY\times_U U'
= \Proj_U\left(
\bigoplus_{n\geq 0}
\pi_*\cL^{\otimes n}
\right) \times_U U'
\overset{(\star\star)}{\simeq} \Proj_{U'}\left(
\bigoplus_{n\geq 0}
\gamma^*\pi_*\cL^{\otimes n}
\right)
\simeq
\Proj_{U'}
\left(
\bigoplus_{n\geq 0}
\pi'_*\cL'^{\otimes n}
\right)
\]
where $(\star\star)$ holds by  
\cite[Tag 01O3]{SP}. 
Thus (1) holds.

The assertion (2) follows from the fact that each graded piece $\pi_*\cL^{\otimes n}$ of $\cA$ is a locally free sheaf. 
The assertion (3) holds by 
\[
\rho(X_{\ol{u}}/Y_{\ol{u}})
\overset{{\rm (1)}}{=}
\dim \operatorname{Span}_{\mathbb R}
\left(
\NE(X_{\ol{u}})\cap \cL_{\ol{u}}^\perp
\right)
\overset{(\sharp)}{=}
\dim \operatorname{Span}_{\mathbb R}
\left(
\NE(X_{\ol{v}})\cap \cL_{\ol{v}}^\perp
\right)
\overset{{\rm (1)}}{=}
\rho(X_{\ol{v}}/Y_{\ol{v}}), 
\]
where $(\sharp)$ follows from  Proposition \ref{p DVR NE Nef} 
and the compatibility of the
specialisation maps with the intersection pairing. 
\end{proof}

\begin{prop}\label{p ext type family}
Let $U$ be a connected noetherian scheme and 
let $\pi : \cX \to U$ be a smooth projective morphism such that every geometric fibre is a Fano threefold. 
Take a $\pi$-semi-ample invertible sheaf $\cL$ on $\cX$ and 
let $f: \cX \to \cY$ be the contraction induced by $\cL$. 
Fix two geometric points $\ol{u}, \ol{v}$ of $U$. 
Assume that $\rho(X_{\ol{u}}/Y_{\ol{u}}) = 1$. 
Then the following hold: 
\begin{enumerate}
\item 
Let $T_n$ be one of $C_1, C_2, D_1, D_2, D_3, E_1, E_2, E_5$. 
Then $f_{\ol{u}}$ is of type $T_n$ if and only if 
$f_{\ol{v}}$ is of type $T_n$. 
\item $f_{\ol{u}}$ is of type $E_3$ or $E_4$ if and only if 
$f_{\ol{v}}$ is of type $E_3$ or $E_4$. 
\item $Y_{\ol{u}}$ is a Fano threefold if and only if 
$Y_{\ol{v}}$ is a Fano threefold. 
\end{enumerate}
\end{prop}

\begin{proof}
By standard argument, the problem is reduced to the case when 
$U = \Spec R$, 
where $(R, \m)$ is a complete discrete valuation ring  such that 
$R/\m$ is algebraically closed. 
Let $s$ (resp.\ $\ol{\xi}$) be the closed (resp.\ geometric generic) point of $\Spec R$. 
By Proposition \ref{p cont bc}, we have 
$\dim Y_{\ol{\xi}}= \dim Y_s$ and 
$\rho(X_{\ol{\xi}}/Y_{\ol{\xi}})=\rho(X_s/Y_s)$. 

Let us prove (1) for the case when $\dim Y_{\ol{u}}<3$. 
Let  $f_{\ol{\xi}}$ is of type $T_n$, where $T_n$ is one of $C_1, C_2, D_1, D_2, D_3$. 
In this case, $n$ is equal to the length of the extremal ray 
\cite[Remark 3.4]{FanoIII}. 
Hence we can find a rational curve $C_{\ol{\xi}}$ on $X_{\ol{\xi}}$ satisfying 
$-K_{X_{\ol{\xi}}} \cdot C_{\ol{\xi}} = n$. 
Let $C_s$ be the curve on $X_s$ contained in the specialisation of $C_{\ol{\xi}}$, which is a rational curve. 
Then we get $0<-K_{X_s} \cdot C_s \leq -K_{X_{\ol{\xi}}} \cdot C_{\ol{\xi}} = n$. 
In particular, we are done for the case when $n=1$. 

Assume $n>1$. 
In this case, we have $-K_{X_{\ol{\xi}}} \equiv_{f_{\ol{\xi}}} n D_{\ol{\xi}}$ 
for some Cartier divisor $D_{\ol{\xi}}$ on $X_{\ol{\xi}}$ 
\cite[Proposition 3.5]{FanoIII}. 
For the specialisation $D_s$ of $D_{\ol{\xi}}$ on $X_s$, 
we get $-K_{X_s} \equiv_{f_s} nD_s$, which implies that $n' \in n\Z_{>0}$ for the type $T_{n'}$ of $f_s$. 
Therefore, $n=n'$. 
This completes the proof of (1) for the case when $\dim Y_{\ol{u}}<3$.

Let us show (1) and (2) for the case when $\dim Y_{\ol{u}}=3$. 
By the same argument as in Proposition \ref{p DVR rho=1 equiv}(4), 
we can find a prime divisor $\cE$ flat over $\Spec R$ whose geometric fibre $\cE_t$ is the exceptional prime divisor. 
As for type $E_1$, 
we can use the fact that $f_{\ol{u}}$ is of type $E_1$ if and only if $\cL_{\ol{u}} \cdot \cE^2_{\ol{u}} <0$. 
Then the remaining cases follow from 
 \[
\cE_{\ol{u}}^3
=
\begin{cases}
1 & \text{for type }E_2,\\
2 & \text{for type }E_3\text{ or }E_4,\\
4 & \text{for type }E_5.
\end{cases}
\]
Thus (1) and (2) hold. 

Let us show (3). 
Recall that $-K_{\cY_{\ol{w}}}$ is automatically ample if $f_{\ol{w}}(\Ex(f_{\ol{w}}))$ is a point. 
Hence we may assume that $f_{\ol{w}}$ is of type $E_1$ for every geometric point $\ol{w}$ of $U$. 
In this case, $-K_{\cY_{\ol{w}}}$ is ample if and only if 
$-K_{\cY_{\ol{w}}} \cdot B_{\ol{w}}  >0$ for 
$B_{\ol{w}} := f_{\ol{w}}(\Ex(f_{\ol{w}}))$. 
We have 
\[
-K_{\cY_{\ol{w}}} \cdot B_{\ol{w}} = K_{\cX_{\ol{w}}} \cdot \cE^2_{\ol{w}} -\cE^3_{\ol{w}}. 
\]
Then (3) follows from the fact that 
the right hand side does not depend on $\ol{w}$. 
\end{proof}


\begin{prop}\label{p Pic lift alteration}
Let $S$ be an excellent scheme with $\dim S \leq 2$ and 
let $U$ be an integral separated scheme of finite type over $S$. 
Let $\pi : \cX \to U$ be a smooth projective morphism such that 
every geometric fibre is a Fano threefold. 
Then there exists a generically finite proper surjective morphism $U' \to U$ such that the following hold for every geometric fibre $X$ of 
$\pi' : \cX'  \to U'$: 
\begin{enumerate}
\item $\cX' := \cX \times_U U'$ and $U'$ are regular integral schemes, 
\item the restriction map $\Pic(\cX') \to \Pic X$ is surjective. 
\item 
$N^1(\cX'/U') \xrightarrow{\simeq} N^1(X)$ and $\Nef(\cX'/U') \xrightarrow{\simeq} \Nef(X)$. 
\item 
$N_1(X) \xrightarrow{\simeq} N_1(\cX'/U')$ and 
$\NE(X) \xrightarrow{\simeq} \NE(\cX'/U')$. 
\end{enumerate}

\end{prop}

\begin{proof}
There exists a finite surjective morphism $U_1 \to U$ 
from an integral scheme $U_1$ 
such that every line bundle on the geometric generic fibre of $\pi : \cX \to U$ descends to the generic fibre of the base change $\pi_1 : \cX_1 \to U_1$ of $\pi$. 
Let $U' \to U_1$ be a regular alteration, 
whose existence is guaranteed by 
\cite[Corollary 5.15]{dJ97}.  
Then the resulting base change $\pi' : \cX' \to U'$ 
satisfies (1). 
As $\cX'$ is regular, every Cartier divisor on the generic fibre of $\pi'$ extends to a Cartier divisor on $\cX'$. 
Therefore, (2) holds for the case when $X$ is the geometric generic fibre, which implies (2) for the general case by Remark \ref{r sp map isom}. 
Then (3) holds by  (2) and Corollary \ref{c relative nef}. 
Note that (4) follows from (3) by duality.  
\end{proof}

\begin{thm}\label{t No inv}
Let $U$ be a connected noetherian scheme and 
let $\pi : \cX \to U$ be a smooth projective morphism such that every geometric fibre is a Fano threefold. 
Fix two geometric points $\ol{u}, \ol{v}$ of $U$. 
Then $\cX_{\ol{u}}$ is of No.\ x-yz if and only if 
$\cX_{\ol{v}}$ is of No.\ x-yz. 
\end{thm}

\begin{proof}
By standard argument, we may assume that $U = \Spec R$, 
where $(R, \m)$ is a complete discrete valuation ring, $R/\m$ is algebraically closed, 
and $\Pic \cX \to \Pic X_s$ is surjective 
for the closed point $s$ of $\Spec R$. 
Let $\ol{\xi}$ be the geometric generic point. 
Note that $x = \rho(X_s) = \rho(X_{\ol{\xi}})$. 
There is nothing to show if $x \geq 6$ \cite[Subsection 7.6]{FanoIV}. 
In what follows, we prove the assertion  by induction on $x$. 

If $x \neq 3$, then the assertion follows from 
the induction hypothesis, Proposition \ref{p ext type family}, 
and 
the definition of No.\ x-yz 
\cite[Section 7]{FanoIV}. 
Assume $x=3$. In this case, the assertion holds by 
the induction hypothesis, Proposition \ref{p ext type family}, 
and the classification list \cite[Subsection 7.3]{FanoIV}. 
For example, 3-11 and 3-12 can be distinguished by whether there exists an extremal ray of type $E_1$ whose contraction is a Fano threefold of No.\ 25. 
\end{proof}

\section{Existence criteria}

\subsection{Complete intersection}

\begin{lem}\label{l CI exist}
Let $W$ be a smooth projective variety with $\dim W =c+ 3$ and $c>0$. 
Let $H_1, ..., H_c$ be very ample Cartier divisors on $W$. 
Assume that $-(K_W +H_1+ \cdots +H_c)$ is ample. 
Set $X :=  H'_1 \cap \cdots \cap H'_c$, 
where each $H'_i$ is a general member of $|H_i|$. 
Then the following hold: 
\begin{enumerate}
\item $X$ is a Fano threefold. 
\item $\rho(X) = \rho(W)$. 
\end{enumerate}
\end{lem}

\begin{proof}
By the Bertini theorem and the adjunction formula, 
$X$ is a Fano threefold. Thus (1) holds. 
Set $W_0 := W$ and $W_i := H'_1 \cap \cdots \cap H'_i$ 
for every $1 \leq i \leq c$. We have $X = W_c$. 
By the Grothendieck--Lefschetz theorem for N\'eron--Severi groups 
\cite[Section 4.1]{Amb23}, the restriction homomorphism
\[
\operatorname{NS}(W_{i-1})_{\mathbb Q}
\longrightarrow
\operatorname{NS}(W_i)_{\mathbb Q}
\]
is an isomorphism for every $i\in\{1,\ldots,c\}$, because $\dim W_{i-1}\geq4$. Consequently, (2) holds. 
\end{proof}

\subsection{Double covers}

\begin{prop}\label{p smooth double cover exist}
Let $Y$ be a smooth projective variety with $\dim Y >0$ 
and let $L$ be a very ample line bundle on $Y$. 
Take a general element $(a, b) \in H^0(Y,L) \oplus H^0(Y,L^{\otimes 2})$. 
Let $f\colon X\longrightarrow Y$ be the finite morphism 
locally defined by 
\[
t^2 + at + b=0. 
\]
Then $f$ is a finite flat morphism such that 
$X$ is a smooth projective variety, 
$K(X)/K(Y)$ is a separable extension with $[K(X):K(Y)]=2$, 
and 
$\omega_X
\simeq
f^*(\omega_Y\otimes L)$. 
\end{prop}

\begin{proof}
If $p \neq 2$, then the assertion is well known. 
In what follows, we assume $p=2$.

For the total space of $L$ 
\[
\pi: \operatorname{Tot}(L)
:=\operatorname{Spec}_Y \bigoplus_{d \geq 0} L^{-d} \to Y, 
\]
let $z\in
H^0\left(
\operatorname{Tot}(L),
\pi^*L
\right)$ be the tautological section.
For $a\in H^0(Y,L)$ and $b\in H^0(Y,L^{\otimes 2})$, 
define
\[
X_{a,b}
:=
Z\left(
z^2+\pi^*a\cdot z+\pi^*b
\right)
\subset
\operatorname{Tot}(L).
\]
Since the defining equation is a monic polynomial of degree two, the restriction
$\pi|_{X_{a,b}}
\colon
X_{a,b}
\to 
Y$ is a finite flat morphism of degree two.

\setcounter{step}{0}

\begin{step}\label{s1 smooth double cover exist}
$X_{a, b}$ is an integral scheme for a general pair $(a, b)\in H^0(Y,L) \oplus H^0(Y,L^{\otimes 2})$. 
\end{step}

\begin{proof}[Proof of Step \ref{s1 smooth double cover exist}]
Being geometrically integral is an open condition. 
Hence it is enough to show that $X_{0, b}$ is an integral scheme 
for a general element $b \in H^0(Y, L^{\otimes 2})$. 
Let $D \in |L^{\otimes 2}|$ be the member corresponding to $b$. 
By the Bertini theorem, $D$ is a smooth effective divisor on $Y$. 
Then it is easy to see that $t^2 +b \in K(Y)[t]$ is irreducible over $K(Y)$. 
Thus $K(Y)[t]/(t^2 +b)$ is an integral domain. 
For a point $y \in Y$, 
it is enough to show that 
$\MO_{Y, y}[t]/(t^2+b)$ is an integral domain, which follows from 
the fact that the induced ring homomorphism 
\[
\MO_{Y, y}[t]/(t^2+b) \to K(Y)[t]/(t^2+b)
\]
is injective. 
This completes the proof of Step \ref{s1 smooth double cover exist}. 
\end{proof}

\begin{step}\label{s2 smooth double cover exist}
$X_{a, b}$ is smooth for a general pair $(a,b)\in H^0(Y,L) \oplus H^0(Y,L^{\otimes 2})$. 
\end{step}

\begin{proof}[Proof of Step \ref{s2 smooth double cover exist}]
Set
\[
\mathcal P
:=
H^0(Y,L)
\oplus
H^0(Y,L^{\otimes 2}).
\]
Fix a closed point
$z_0\in\operatorname{Tot}(L)$ and set $y_0 := \pi(z_0) \in Y$. 
Take an affine open neighbourhood $y_0 \in U \subset Y$ with a trivialisation $\theta: L|_U \xrightarrow{\simeq} \MO_U$. 
Set $A := \MO_Y(U)$, $\theta(a|_U) =: \alpha \in A$, 
and $\theta^{\otimes 2}(b|_U) =:\beta \in A$. 
Under the identification $\pi^{-1}(U) = \Spec A[t] = U \times \mathbb A^1_t$, 
we have $z_0 = (y_0, t_0)$ for some $t_0 \in k$. 

The local defining equation of $X_{a,b}$ is given by 
\[
F:=t^2+\alpha t+\beta \in A[t].
\]
Since $p=2$, we obtain $dF = td\alpha + \alpha dt + d\beta$ in 
$\Omega^1_{A[t]/k} = (\Omega^1_{A/k}\otimes_A A[t]) \oplus A[t] dt$, which implies 
$dF|_{z_0} = t_0 (d\alpha|_{y_0}) +  \alpha(y_0) dt + d\beta|_{y_0}$ 
in $\Omega^1_{A[t]/k} \otimes_{A[t]} k_{z_0} 
\simeq (\Omega^1_{A/k}\otimes_A k_{z_0}) \oplus k_{z_0} dt$. 
Therefore, we get the equivalence (1) $\Leftrightarrow$ (2). 
\begin{enumerate}
\item $z_0$ is a singular point of $X_{a, b}$. 
\item 
\begin{itemize}
\item $\alpha(y_0)=0$, 
\item $\beta(y_0)=t_0^2$, and $d\beta|_{y_0}
=
t_0d\alpha|_{y_0}$.
\end{itemize}
\end{enumerate}

Since $L$ and $L^{\otimes 2}$ are very ample,  the restriction maps  
\[
\varphi : H^0(Y, L) \to L
\otimes_{\cO_{Y}}
\cO_{Y,y_0}/\mathfrak m_{y_0},\qquad    
\psi: H^0(Y,L^{\otimes 2})
\to 
L^{\otimes 2}
\otimes_{\cO_{Y}}
\cO_{Y,y_0}/\mathfrak m_{y_0}^2
\]
are surjective. 
We have $\varphi(a) = \alpha(y_0)$. 
For the isomorphism 
\[
\Theta: L^{\otimes 2}
\otimes_{\cO_{Y}}
\cO_{Y,y_0}/\mathfrak m_{y_0}^2 \xrightarrow{\theta^{\otimes 2}, \simeq}  
\cO_{Y,y_0}/\mathfrak m_{y_0}^2 
\xrightarrow{\simeq} 
\cO_{Y,y_0}/\mathfrak m_{y_0} \oplus  
\mathfrak m_{y_0}/\mathfrak m_{y_0}^2 
\xrightarrow{\simeq} 
k \oplus (\Omega^1_{Y/k}|_{y_0})
 \]
induced by $\theta$, we get $\Theta \circ \psi(b) = (\beta(y_0), d\beta|_{y_0})$. 

Set 
\[
\mathcal I
:= 
\left\{
(a,b,z_0)
\in
\mathcal P\times\operatorname{Tot}(L)
\ \middle|\
z_0\text{ is a singular point of }X_{a,b}
\right\} 
\]
which is a closed subscheme of $\mathcal P\times\operatorname{Tot}(L)$. 
By the equivalence (1) $\Leftrightarrow$ (2) and the previous paragraph, 
every fibre $\cI_{z_0}$ of $\cI \to \Tot(L)$ 
over a closed point $z_0 \in \Tot(L)$ 
is an affine space satisfying 
\[
\dim \cI_{z_0} = (h^0(Y, L) -1) + (h^0(Y, L^{\otimes 2}) -(\dim Y+1)) 
= \dim \cP -\dim Y -2. 
\]
Therefore, 
\[
\dim \cI \leq \dim \Tot(L) + ( \dim \cP -\dim Y -2)  = \dim \cP-1. 
\]
Then the closure of the image of $\cI \to \cP$ is a proper closed subset of $\cP$. 
Hence $X_{a, b}$ is smooth for a general element $(a, b) \in H^0(Y,L) \oplus H^0(Y,L^{\otimes 2})$. 
This completes the proof of Step \ref{s2 smooth double cover exist}. 
\end{proof}

By Step \ref{s1 smooth double cover exist} and Step \ref{s2 smooth double cover exist}, $X_{a, b}$ is a smooth projective variety. 
It is clear that $K(X)/K(Y)$ is separable 
if and only if $a \neq 0$. 
Finally, the isomorp
$\omega_X
\simeq
f^*(\omega_Y\otimes L)$ follows from \cite[Proposition 0.2.20]{CDL25}
\qedhere 

\end{proof}

\subsection{Blowup along curves}

\begin{dfn}\label{d-intersection-members}
Let $Y$ be a smooth projective variety, 
let $C\subset Y$ be a closed subscheme, and let $L$ be a line bundle on $Y$.
We say that $C$ is an {\em intersection of members of} $|L|$ if
$\mathcal I_{C/Y} \otimes L$ is globally generated.
\end{dfn}

\begin{prop}\label{p MM 2.12}
Let $Y$ be a smooth projective threefold and 
let $C \subset Y$ be a  non-empty smooth closed subscheme of pure
codimension $r$ with $r\in\{2,3\}$. 
Assume that there exists a line bundle $L$ on $Y$ such that 
$-K_Y -(r-1)L$ is ample and $C$ is an intersection of members of $|L|$. 
Then the blowup $\operatorname{Bl}_C(Y)$ is a Fano threefold.
\end{prop}

\begin{proof}
The same argument as in \cite[Proposition 2.12]{MM85} works. 
\end{proof}

\begin{prop}\label{p MM 2.13}
Let $S$ be a smooth projective surface and let $g: Y \to S$ be a $\P^1$-bundle (in particular, $Y$ is a  smooth projective threefold). 
Let 
$C
\subset
Y$ be a non-empty disjoint union of smooth irreducible curves such that the
restriction
$g|_C\colon
C
\to 
S$
is a closed immersion.
Assume that there exists a very ample Cartier divisor $N$ on $S$ such that
$C$ is an intersection of members of
\[
\left|
-K_Y-g^*N
\right|.
\]
Then the blowup $\Bl_C(Y)$ is a Fano threefold. 
\end{prop}

\begin{proof}
The same argument as in \cite[Proposition 2.13]{MM85} works. 
\end{proof}

\begin{prop}\label{p prime div criterion}
Let $Y$ be a smooth projective threefold. 
Let $\sigma : X \to Y$ be the blowup along a non-empty disjoint union $C$ of smooth curves on $Y$. 
Assume that there exists a prime divisor $S$ on $Y$ such that 
\begin{enumerate}
\item $-K_Y -S$ is ample, 
\item $-K_Y|_S -C$ is ample, 
\item $C \subset S$, and $S$ is smooth along $C$. 
\end{enumerate}
Then $-K_X$ is ample. 
\end{prop}

\begin{proof}
Set $E:=\Ex(\sigma)$ and $A := -K_Y-S$. 
Note that $A$ is ample  by (1). 
Let $S_X$ be the proper transform of $S$ on $X$. 
By (3), we have $\sigma^*S = S_X +E$.  
We can find a small positive rational number $0 < \epsilon < 1$ 
such that the $\Q$-divisor 
\[
A' := \sigma^*( (1-\epsilon) A -\epsilon K_Y)-\epsilon E
\]
is ample. It holds that 
\begin{align*}
-K_X &\sim \sigma^*(-K_Y) -E\\
 &= \epsilon \sigma^*(-K_Y) + (1-\epsilon) \sigma^*(A+S) -E\\
&= \epsilon \sigma^*(-K_Y) + (1-\epsilon) \sigma^*A + (1-\epsilon)S_X -\epsilon E\\
&= A' +  (1-\epsilon)S_X. 
\end{align*}
In particular, we have $\bB(-K_X) \subset S_X$ 
for the stable base locus $\bB(-K_X)$ of $-K_X$. 
Via the isomorphism $S_X \xrightarrow{\simeq} S$, we have $-K_X|_{S_X} \sim (-\sigma^*K_Y -E)|_{S_X} = -K_Y|_S -C$. 
Hence $-K_X|_{S_X}$ is ample by (2). 
Therefore, $-K_X$ is semi-ample by \cite[Theorem 1.10]{Fuj83semipositive}. 

Fix a curve $\Gamma$ on $X$. It suffices to show $-K_X \cdot \Gamma >0$. 
If $\Gamma \subset S_X$, then $-K_X \cdot \Gamma = (-K_X|_{S_X}) \cdot \Gamma >0$ by ampleness of $-K_X|_{S_X}$. 
If $\Gamma \not\subset S_X$, then 
$-K_X \cdot \Gamma = (A' + (1-\epsilon)S_X) \cdot \Gamma \geq A' \cdot \Gamma >0$. 
\end{proof}

\begin{cor}\label{c prime div criterion to P1-bdl}
Let $T$ be a smooth projective surface and let $L$ be a nef line bundle  on $T$.  
Let $S$ be the section of the $\P^1$-bundle $\pi: Y:=\P_T(\MO_T \oplus L) \to T$ corresponding to the projection $\MO_T \oplus L \to L$. 
Take a smooth curve $C$ on $S$ and set $C_T := \pi(C)$. 
Assume that $-K_T-L$ and $L-K_T-C_T$ are ample. 
Then $\Bl_C\,Y$ is a Fano threefold. 
\end{cor}

\begin{proof}
Let $\xi := \MO_Y(1)$ be the tautological line bundle of the $\P^1$-bundle 
$\pi$. 
We have $K_Y \sim -2\xi+\pi^*(K_T+L)$, 
$S \sim \xi$, and $\xi|_S \simeq L$. 
As $-(K_T+L)$ is ample and $\xi$ is nef and $\pi$-ample, 
the divisor $-K_Y-S \sim \xi-\pi^*(K_T+L)$ is ample. 
Furthermore, via the isomorphism $S \xrightarrow{\simeq} T$, we obtain
\[
-K_Y|_S-C
\sim
(
2\xi+\pi^*(-K_T-L)
)|_S
-C
\sim 
2L -K_T -L-C_T = L-K_T-C_T,  
\]
which is ample. 
Therefore, $X := \Bl_C\,Y$ is a Fano threefold  by Proposition \ref{p prime div criterion}. 
\end{proof}
\section{Existence via blowups}

For a Fano threefold $Z$, 
we call a Fano threefold $X$ a {\em maximal blowup} of $Z$ if there exists a birational morphism $X \to Z$ which is a composition of blowups along smooth curve centres and 
there is no smooth curve $C$ on $X$ such that $\Bl_C X$ is Fano. 
The main purpose of this section is to establish the existence of maximal blowups of 
$\P^3, Q, \P^2 \times \P^1$, and del Pezzo threefolds. 
We start by treating some easy cases in the first subsection.

\subsection{Direct constructions and cases $\rho=1$ and $\rho \geq 5$}

\begin{prop}\label{p exist as toric}
Let \[
\text{x-yz} \in 
\{ 
\text{1-17, 2-33, 2-34, 2-35, 2-36,}
\]
\[
\text{
3-25, 3-26, 3-27, 3-28, 3-29, 3-30, 3-31, 4-9, 4-10, 4-11, 4-12, 5-2, 5-3}\}. 
\]
Then there exists a Fano threefold of No.\ x-yz. 
\end{prop}

\begin{proof}
These cases are realised as smooth toric Fano threefolds 
\cite[Theorem in \S 1]{WW82}. 
\end{proof}

\begin{prop}\label{p exist as CI}
Let 
\[
\text{x-yz} \in \{\text{1-2, 1-3, 1-4, 1-5, 1-6, 1-7, 1-8, 1-9, 1-11, 1-12, 1-13, 1-14, 1-15, 1-16,}
\]
\[
\text{2-4, 
2-6, 2-7, 2-9, 
2-12, 2-13, 
2-17, 
2-24, 2-25, 2-27, 2-29, 2-32, 2-33, 3-3, 
3-17, 4-1}\}.
\]
Then exists a Fano threefold of No.\ x-yz. 
\end{prop}

\begin{proof}
For each of 1-5, 1-6, 1-7, 1-8, 1-9, 
the existence is guaranteed by \cite[Example 5.2.2(i)-(v)]{IP99}. 
For the remaining cases, we use complete intersection description as in \cite[Table 1 in Section 5]{BFT22}. 
Here we only give the proof for the case 2-17, as the other cases are simpler.

We have the following diagram: 
\[
\begin{tikzcd}
& F:=\Fl(1,2,4) \arrow[dl,"\pi_1"] \arrow[dr,"\pi_2"', 
"{\P^1\text{-bundle}}"] & \\
\mathbb P^3 =\Gr(1,4) && \Gr(2,4)
\end{tikzcd}
\]
For $a, b \in \Z$, we set 
$\cO_F(a,b):=
\pi_1^*\cO_{\P^3}(a)
\otimes
\pi_2^*\cO_{\Gr(2,4)}(b)$. 
Take a general member $H\in |\MO_{\Gr(2, 4)}(1)|$. 
As $|\MO_{\Gr(2, 4)}(1)|$ is very ample, 
$H$ is a smooth prime divisor by the Bertini theorem. 
For $D_{(0, 1)} := \pi_2^*H$, 
$D_{(0, 1)}$ is a smooth prime divisor on $F$, because 
the projection 
\[
D_{(0, 1)} = \pi_2^{-1}(H) = H \times_{\Gr(2, 4)} F \to H
\]
is a $\P^1$-bundle. 
Take a general member $D_{(1, 1)}$ of $|\MO_F(1, 1)|$. 
Since $|\MO_F(1, 1)|$  is very ample, 
the scheme-theoretic intersection 
$X :=D_{(0, 1)} \cap D_{(1, 1)}$ is a smooth prime divisor, and hence a Fano threefold of No. 2-17. 
\qedhere

\qedhere

\end{proof}

\begin{prop}\label{p double cover exist}
Let $\text{x-yz} \in \{\text{1-1, 2-2, 2-8, 2-18, 3-1}\}$. 
Then there exists a Fano threefold of No.\ x-yz. 
\end{prop}

\begin{proof}
By the double-cover descriptions in \cite[Section 7]{FanoIV}, it suffices to consider a smooth double cover
$f\colon X\longrightarrow Y$ associated 
with a line bundle $\cL$ as follows:
\begin{itemize}
\item 1-1: $Y=\P^3$ and $\cL=\MO_{\P^3}(3)$.
\item 2-2: $Y=\P^2\times\P^1$ and $\cL=\MO_{\P^2\times\P^1}(2,1)$.
\item 2-8: $Y=V_7=\P_{\P^2}(\MO_{\P^2}\oplus\MO_{\P^2}(1))$ and $\cL=\MO_{V_7}(1) \otimes \pi^*\MO_{\P^2}(1)$, 
where $\MO_{V_7}(1)$ denotes the tautological line bundle 
of the induced $\P^1$-bundle $\pi\colon V_7\to\P^2$. 
\item 2-18: $Y=\P^2\times\P^1$ and $\cL=\MO_{\P^2\times\P^1}(1,1)$.
\item 3-1: $Y=\P^1\times\P^1\times\P^1$ and $\cL=\MO_{\P^1\times\P^1\times\P^1}(1,1,1)$.
\end{itemize}
In each case, $-K_Y-\cL$ is ample and $|\cL|$ is very ample. 
Then 
such a smooth double cover $X$ exists and 
$X$ is a Fano threefold by  Proposition \ref{p smooth double cover exist}. 
\end{proof}

\begin{prop}\label{p exist rho 1}
Let x-yz be a number with $x = 1$, i.e., 
\[
\text{x-yz} \in 
\{\text{1-1, 1-2, ..., 1-17}\}. 
\]
Then there exists a Fano threefold of No.\ x-yz.
\end{prop}

\begin{proof}
Except for 1-10, 
the existence is guaranteed by the results listed below. 
\begin{itemize}
\item 1-1: Proposition \ref{p double cover exist}. 
\item 1-17: Proposition \ref{p exist as toric}. 
\item 1-2, 1-3, 1-4, 1-5, 1-6, 1-7, 1-8, 1-9, 1-11, 1-12, 1-13, 1-14, 1-15, 1-16: Proposition \ref{p exist as CI}. 
\end{itemize}
The existence for No.\ $1$-$10$ follows from \cite[Theorem 1.1]{IKTTg12}. 
\end{proof}

\begin{prop}\label{p exist rho geq 5}
Let x-yz be a number with $x \geq 5$, i.e., 
\[
\text{x-yz} \in 
\{\text{5-1, 5-2, 5-3, 6-1, 7-1, 8-1, 9-1, 10-1}\}. 
\]
Then there exists a Fano threefold of No.\ x-yz.
\end{prop}

\begin{proof}
The case $x \geq 6$ is realised as a direct product of 
$\P^1$ and a smooth del Pezzo surface. 
Hence we may assume $x=5$, i.e., 
$x\text{-}yz \in \{\text{5-1, 5-2, 5-3}\}$. 
If $x\text{-}yz = \text{5-1}$ (resp.\ 
$x\text{-}yz \in \{\text{5-2, 5-3}\}$), 
then the assertion follows from 
\cite[Proposition 5.23]{FanoIV} (resp.\ Proposition \ref{p exist as toric}). 
\end{proof}

\subsection{Del Pezzo threefolds}

In this subsection, we establish the existence of maximal blowups of del Pezzo threefolds.

\begin{lem}\label{l dP3 elliptic exist}
Let 
\[
\text{x-yz} \in \{\text{2-1, 2-3, 2-5, 2-10, 2-14, 3-7, 3-11, 4-1}\}. 
\]
Then there exists a Fano threefold of No.\ x-yz. 
\end{lem}

\begin{proof}
Let $Y$ be a del Pezzo threefold. 
Take an ample Cartier divisor $H$ satisfying $-K_Y \sim 2H$. 
Set $d := H^3$. 
Let $S$ and $S'$ be general members of $|H|$. 
Then $C := S \cap S'$ is a smooth curve by 
\cite[Theorem A(1)]{KT25}. 
In particular, $S$ is a prime divisor smooth along $C$. 
Then both $-K_Y -S \sim 2H -H = H$ and $-K_Y|_S -C \sim  (2H -S')|_S = H|_S$ are ample.  
Therefore, $-K_X$ is ample for $X := \Bl_C\,Y$ (Proposition \ref{p prime div criterion}). 
\end{proof}

The maximal blowups of $\P^1 \times \P^1 \times \P^1$ are 4-1, 4-3, 4-6, 4-8, 4-13, and 10-1. 
Since the cases 4-1 and 10-1 have already been settled, 
it remains to establish the existence of 4-3, 4-6, 4-8, and 4-13.




\begin{prop}
\label{p existence 4-13}
Let $\text{x-yz} \in \{\text{4-3, 4-6, 4-8, 4-13}\}$. 
Then there exists a Fano threefold of No.\ x-yz. 
\end{prop}

\begin{proof}
Set
$Y
:=
\P^1_1\times\P^1_2\times\P^1_3$ and $S := \Delta \times \P^1_3 \subset Y$, 
where $\Delta$ denotes the diagonal on $\P^1_1 \times \P^1_2$. 
We have $S \sim H_1 + H_2$ for the $i$-th projection $\pr_i : Y \to \P^1_i$ and $H_i := \pr_i^*\MO_{\P^1_i}(1)$. 
Set $M := H_1|_S = H_2|_S$ and $N := H_3|_S$. 
Fix $d \in \{0, 1, 2, 3\}$ and 
take a smooth curve $C \subset S$ satisfying $C \in |dM + N|$. 
Then $C$ is of tridegree $(1, 1, d)$. 
Let $\sigma : X := \Bl_C\,Y \to Y$ be the blowup along $C$. 
By \cite[Subsection 7.4]{FanoIV}, it suffices to prove that $-K_X$ is ample. 
As $S$ is a smooth prime divisor with $C \subset S$, 
it is enough to show (1) and (2) below (Proposition \ref{p prime div criterion}). 
\begin{enumerate}
\item $-K_Y -S$ is ample. 
\item $-K_Y|_S -C$ is ample. 
\end{enumerate}
The condition (1) holds by 
$-K_Y -S \sim (2H_1 + 2H_2 + 2H_3) -(H_1+H_2) = H_1 + H_2+2H_3$. 
The condition (2) follows from $4-d>0$ and 
\[
-K_Y|_S -C \sim 
(2H_1 + 2H_2 + 2H_3)|_S -(dM+N) = (4-d)M + N. 
\]
\qedhere

\end{proof}

We now treat the remaining maximal blowups of del Pezzo threefolds $V_d$ with $1 \leq d \leq 5$.


\begin{prop}
\label{p-existence-blowups-of-Vd}
Let 
\[
\text{x-yz} \in \{\text{2\text{-}11,\ 
2\text{-}16,\ 
2\text{-}19,\ 
2\text{-}20,\ 
2\text{-}22,\ 
2\text{-}26.}\}. 
\]
Then there exists a Fano threefold of No.\ x-yz. 
\end{prop}

\begin{proof}
For $d\in\{3,4,5\}$, 
let $V_d$ be a del Pezzo threefold. 
Take an ample divisor $H$ on $V_d$ satisfying $-K_{V_d} \sim 2H$. 
In particular, $|H|$ is very ample, $H^3 =d$, and $\Pic V_d=\mathbb Z H$. 
Let $S$ be a general member of $|H|$, which is a smooth del Pezzo surface with $K_S^2 = d$. 
Fix a conic bundle structure $\pi : S \to \P^1$ and a birational contraction $\tau : S \to \P^2$. 
Choose  a smooth curve $C$ on $S$ as follows:
\[
\begin{array}{c|c|c}
\text{No.}
&
V_d
&
C
\\
\hline
2\text{-}11
&
V_3
&
\text{a }(-1)\text{-curve on }S
\\
2\text{-}16
&
V_4
&
\text{a smooth fibre of } \pi:S \to \P^1
\\
2\text{-}19
&
V_4
&
\text{a }(-1)\text{-curve on }S
\\
2\text{-}20
&
V_5
&
\text{the pullback of a general line under } \tau : S\to\P^2
\\
2\text{-}22
&
V_5
&\text{a smooth fibre of } \pi:S \to \P^1
\\
2\text{-}26
&
V_5
&
\text{a }(-1)\text{-curve on }S
\end{array}
\]
For $X := \Bl_C\, V_d$, it suffices to prove that $-K_X$ is ample  \cite[Subsection 7.2]{FanoIV}. 
Since $S$ is a smooth prime divisor with $C \subset S$ and 
the divisor $-K_{V_d} -S \sim H$ is ample, 
it suffices to show  $-2K_S -C$ is ample (Proposition \ref{p prime div criterion}). 
Fix a $(-1)$-curve $\Gamma$ on $S$. 
By Kleiman's criterion for ampleness, 
it suffices to show that $(-2K_S-C) \cdot \Gamma >0$. 
To summarise, the problem is reduced to proving $(\star)$ below. 
\begin{enumerate}
\item[($\star$)] $C \cdot \Gamma \leq 1$. 
\end{enumerate}
For $D := C +\Gamma$, 
we have $D^2 = C^2 + 2C \cdot \Gamma -1$.   
Then it follows from the Hodge index theorem that 
\[
d( C^2 + 2C \cdot \Gamma -1)\leq 
D^2 \cdot (-K_S)^2 \leq ( D \cdot (- K_S))^2 =(1 -K_S \cdot C)^2. 
\]

\underline{2-11, 2-19, 2-26}: 
Assume that $C$ is a $(-1)$-curve. 
In this case, we get 
\[
3( -1 +2 C \cdot \Gamma -1) \leq d( C^2 + 2C \cdot \Gamma -1) \leq (1 -K_S \cdot C)^2 = (1- (-1))^2 =4, 
\]
which implies $2C \cdot \Gamma \leq 2 + \frac{4}{3}$. 
Hence $C \cdot \Gamma \leq 1$, as required.

\underline{2-16, 2-22}: 
Assume that $C$ is a smooth fibre of $\pi : S \to \P^1$. 
In this case, we get 
\[
4( 0 +2 C \cdot \Gamma -1) \leq d( C^2 + 2C \cdot \Gamma -1) \leq (1 -K_S \cdot C)^2 = (1- (-2))^2 =9, 
\]
which implies $2C \cdot \Gamma \leq 1 + \frac{9}{4}$. 
Hence $C \cdot \Gamma \leq 1$, as required.

\underline{2-20}: 
Assume that $C$ is the pullback of a general line under $\tau : S \to \P^2$. 
In this case, we get 
\[
10 C \cdot \Gamma =5(1 + 2C \cdot \Gamma -1) =d( C^2 + 2C \cdot \Gamma -1)\leq (1 -K_S \cdot C)^2 = (1+3)^2 = 16, 
\]
which implies $C \cdot \Gamma \leq 1$, as required. 
\end{proof}

The maximal blowups of $W$ are 3-7, 3-13, 3-16, 3-20, and 4-7.
The cases 3-16 and 3-20 will be treated later via $V_7$ and $Q$, respectively. 
Since the case 3-7 has already been settled in 
Lemma \ref{l dP3 elliptic exist}, 
we treat the remaining two cases  
3-13 and 4-7.


\begin{prop}
\label{p-existence-blowups-of-W}
Let x-yz be 3-13 or 4-7. 
Then there exists a smooth Fano threefold of No.\ x-yz. 
\end{prop}

\begin{proof}
Let $W\subset\P^2\times\P^2$ be a smooth divisor of bidegree
$(1,1)$. 
For the $i$-th projection $\pi_i\colon W\longrightarrow\P^2$ and 
$H_i := \pi_i^*\MO_{\P^2}(1)$, 
we have $-K_W \sim 2H_1 + 2H_2$ by adjunction. 

\underline{4-7}: 
Fix fibres $C_1$ and $C_2$  of $\pi_1$ and $\pi_2$, respectively, such that $C_1 \cap C_2 = \emptyset$. 
For $L := H_1 + H_2$, $C:= C_1 \amalg C_2$ is an intersection of members of $|L|$, 
and hence $X := \Bl_C W$ is Fano by Proposition \ref{p MM 2.12}. 
Then $X$ is of No.\ 4-7.

\underline{3-13}: 
Fix  a general smooth member $S\in|H_1+H_2|$. 
Then $S$ is a smooth del Pezzo surface with $K_S^2 = 6$. 
In particular, 
there is the blowup 
$\tau : S \simeq \Bl_{P_1, P_2, P_3}\,\P^2 \to \P^2$, 
where $P_1, P_2, P_3$ are  non-collinear points. 
We can find a smooth conic $C_{\P^2}$ such that $P_1 \in C_{\P^2}, P_2 \in C_{\P^2}, P_3 \not\in C_{\P^2}$. 
For the proper transform $C$ of $C_{\P^2}$ on $S$, 
let $\sigma: X \to W$ be the blowup along $C$. 
As $S$ is a smooth prime divisor with $C \subset S$ and $-K_W -S$ is ample, 
it suffices to show that $-K_W|_S -C$ is ample 
by Proposition \ref{p prime div criterion} and \cite[Subsection 7.3]{FanoIV}. 
For $M := \tau^*\MO_{\P^2}(1)$, we have
\[
\begin{aligned}
-K_W|_S-C
&\sim
-2K_S-C
\\
&\sim
2(3M-E_1-E_2-E_3)-(2M-E_1-E_2)
\\
&\sim
4M-E_1-E_2-2E_3.
\end{aligned}
\]
The Kleiman-Mori cone of $\NE(S)$ is generated by the six $(-1)$-curves
\[
E_1,\qquad
E_2,\qquad
E_3,
\]
\[
M-E_1-E_2,\qquad
M-E_1-E_3,\qquad
M-E_2-E_3.
\]
We have $(4M-E_1-E_2-2E_3)\cdot E_i>0$ for every $i \in\{1, 2, 3\}$, and 
\[
(4M-E_1-E_2-2E_3)\cdot(M-E_1-E_2)=2,
\]
\[
(4M-E_1-E_2-2E_3)\cdot(M-E_1-E_3)=1,
\]
\[
(4M-E_1-E_2-2E_3)\cdot(M-E_2-E_3)=1.
\]
By Kleimann's criterion, $-K_W|_S -C$ is ample, as required. 
\end{proof}

The maximal blowups of $V_7$ are 3-11, 3-14, 3-16, 3-23, 3-29, 5-1, and 5-2.
The cases 3-11, 3-29, 5-1, and 5-2 have already been settled. 
Thus it remains to establish the existence of 3-14, 3-16, and 3-23.




\begin{prop}
\label{p-existence-blowups-of-V7}
Let 
$\text{x-yz} \in 
\{\text{3\text{-}14,\ 
3\text{-}16,\ 
3\text{-}23}\}.$ 
Then there exists a Fano threefold of No.\ x-yz. 
\end{prop}

\begin{proof}
Let $\tau : Y := V_7 \to \P^3$ be the blowup at a closed point $P$. 
We have the $\P^1$-bundle structure: 
\[
\pi\colon
Y
=
V_7
=
\P_{\P^2}
\left(
\MO_{\P^2}\oplus\MO_{\P^2}(1)
\right)
\longrightarrow
\P^2. 
\]
Set $H:=\tau^*\MO_{\P^3}(1)$, $L:=\pi^*\MO_{\P^2}(1)$, and 
$D:=\Ex(\tau)$, which is a section of $\pi$. 
Then $-K_Y\sim 2H+2L$ and $D\sim H-L$. 
Note that a divisor $aH+bL$ is ample 
if and only if $a>0$ and $b>0$. 

\medskip

\underline{3-14}: 
Fix a plane $S_{\P^3} \subset\P^3$ with $P \not\in S_{\P^3}$. 
For the proper transform $S \subset Y$ of $S_{\P^3}$ on $Y$, we have 
$S \sim H$, $S\xrightarrow{\simeq} S_{\P^3} =\P^2$, and 
$H|_S\simeq L|_S\simeq\MO_{\P^2}(1)$. 
Take a smooth plane cubic curve $C\subset S$. 
Let 
\[
\sigma\colon X:=\Bl_C\,Y\longrightarrow Y
\]
be the blowup along $C$.
We have $-K_Y-S\sim H+2L$, which is ample. Moreover, 
\[
-K_Y|_S-C
\sim
\MO_{\P^2}(4)-\MO_{\P^2}(3)
\sim
\MO_{\P^2}(1),
\]
which is ample.
Proposition \ref{p prime div criterion} implies that $X$ is Fano.

\medskip

\underline{3-16 and 3-23}:
Fix $d\in\{1,2\}$ and a smooth quadric surface
$T_{\P^3} \simeq\P^1\times\P^1$ on $\P^3$ with $P \in T_{\P^3}$. 
Take a smooth curve $B_{\P^3} \subset T_{\P^3}$ of bidegree
$(1,d)$ such that $P \in B_{\P^3}$. 
Note that $B_{\P^3}$ is a conic (resp.\ rational cubic curve) if $d=1$ (resp.\ $d=2$). 
Let $T\subset Y$ and $B\subset Y$ be the proper transforms of $T_{\P^3}$ and $B_{\P^3}$, respectively.
Let $\sigma\colon X:=\Bl_B\,Y\longrightarrow Y$ be the blowup along $B$.
Then $T = \tau^*T_{\P^3} -D \sim2H-D\sim H+L$ 
and \[ 
T\simeq\Bl_P(\P^1\times\P^1). 
\]
In particular, $-K_Y - T \sim  (2H+2L) - (H+L) \sim H+L$, which is ample.

By  Proposition \ref{p prime div criterion} and \cite[Subsection 7.3]{FanoIV}, 
it suffices to show that $-K_Y|_T -B$ is ample. 
Let $M$ and $N$ be the pullbacks of the two ruling classes on $\P^1\times\P^1$ and let $E \subset T$ be the exceptional $(-1)$-curve. 
$\NE(T)$ is generated by
\[ 
E, \qquad M-E, \qquad N-E. 
\]
After switching $M$ and $N$ if necessary, 
we have $-K_T\sim2M+2N-E$ and $B = (\tau|_T)^*B_{\P^3} -E \sim M+dN-E$. 
Moreover, 
\[ 
-K_Y|_T-B \sim -2K_T-B \sim 
(4M+4N-2E) - (M+dN-E) = 
3M+(4-d)N-E. 
\]  
Then $-K_Y|_T -B$ is ample by Kleiman's criterion and the following computations: 
\[ 
\left( 3M+(4-d)N-E \right)\cdot E = 1, 
\] 
\[ 
\left( 3M+(4-d)N-E \right)\cdot(M-E) = 3-d > 0, 
\]
\[ 
\left( 3M+(4-d)N-E \right)\cdot(N-E) = 2, 
\]
as required. 
\end{proof}



\subsection{ $\P^3$, $Q$, $\P^2 \times \P^1$}

In this subsection, we establish the existence of maximal blowups of 
$\P^3, Q$, and $\P^2 \times \P^1$. 
The remaining maximal blowups of $\P^1 \times \P^2$ are 3-5, 3-22, 4-5, which are settled as follows. 

\begin{prop}
\label{p-existence-3-5-3-22}
Let $\text{x-yz} \in 
\{\text{3\text{-}5, 3\text{-}22,  4\text{-}5}\}.$ 
There exists a Fano threefold of No.\ x-yz. 
\end{prop}

\begin{proof}
Set $Y:=\P^1\times\P^2$ and let $\pi_i\colon Y\longrightarrow\P^i$ be the $i$-th projection. 
For $H_1:=\pi_1^*\MO_{\P^1}(1)$ and $H_2:=\pi_2^*\MO_{\P^2}(1)$, we have 
$-K_Y\sim2H_1+3H_2$.

\medskip

\underline{3-22}: 
Fix a closed point $t\in\P^1$ and a smooth conic $B\subset\P^2$. 
Set
$S:=\pi_1^{-1}(t)=\{t\}\times\P^2$ and $C:=\{t\}\times B$. 
Then $C \subset S$, $S \in |H_1|$, and $C\in|2H_2|_S|$. 
We see that the divisor $-K_Y-S
\sim
H_1+3H_2$ is ample. 
Via the isomorphism
$S\simeq\P^2$, we obtain
\[
-K_Y|_S-C
\sim
\MO_{\P^2}(3)-\MO_{\P^2}(2)
\sim
\MO_{\P^2}(1),
\]
which is ample. 
By Proposition \ref{p prime div criterion}, 
$X := \Bl_C\,Y$ is a Fano threefold, and 
it is of  No.\ 3-22 \cite[Subsection 7.3]{FanoIV}.




\medskip

\underline{3-5}: 
Fix a smooth conic
$B\subset\P^2$ and set $S:=\pi_2^{-1}(B)=\P^1\times B \subset \P^1 \times \P^2 =Y$. 
Then $S\in|2H_2|$ is a smooth prime divisor on $Y$. 
Fix an isomorphism $B\simeq\P^1$ and identify
$S\simeq\P^1\times\P^1$. 
For the divisors $A_1:=\pr_1^*\MO_{\P^1}(1)$ and $A_2:=\pr_2^*\MO_{\P^1}(1)$ on $S$, we have $H_1|_S\sim A_1$ and $H_2|_S\sim2A_2$. 

Take a general member
$C\in|A_1+5A_2|$, which is a smooth curve on $S$. 
Then $C \subset \P^1 \times \P^2$ is of bidegree $(5, 2)$. 
Since both 
\[
-K_Y-S
\sim (2H_1 + 3H_2) - 2H_2 =2H_1+H_2 
\]
and 
\[
-K_Y|_S-C
\sim
(2H_1 + 3H_2)|_S -C \sim 
(2A_1+6A_2)-(A_1+5A_2)= A_1+A_2,
\]
are ample, 
$X := \Bl_C\,Y$ is Fano by Proposition \ref{p prime div criterion}. 
Then $X$ is of No.\ 3-5 by \cite[Subsection 7.3]{FanoIV}.

\medskip

\underline{4-5}: 
Set $T := \P^1 \times \P^1$ and $L:=\MO_{\P^1\times\P^1}(1,1)$. 
Let $\xi := \MO_Y(1)$ be the tautological line bundle of the $\P^1$-bundle $\pi : Y := \P_T(\MO_T \oplus L) \to T$ and let $S\subset Y$ be the section of $\pi$ corresponding to the projection $\MO_T\oplus L \to L$. 
Via the isomorphism $S\xrightarrow{\simeq}T$, take a general member $C\in|\MO_S(1,2)|$, which is a smooth rational curve by the adjunction formula. 
Since
\[
-K_T-L\sim\MO_{\P^1\times\P^1}(1,1)
\qquad\text{and}\qquad
L-K_T-\pi(C)\sim\MO_{\P^1\times\P^1}(2,1)
\]
are ample, $X := \Bl_C\,Y$ is a Fano threefold by Corollary \ref{c prime div criterion to P1-bdl}. 
It is easy to see that $(-K_X)^3 =32$. 
Since $Y$ is a Fano threefold of No.\ 3-31, $X$ is of No.\ $4$-$5$ 
by \cite[Subsection 7.4]{FanoIV}. 
\qedhere

\end{proof}

Let us treat the remaining maximal blowups of $\P^3$, 
namely 2-15, 2-28, 3-6, and 3-12.




\begin{prop}\label{p P^3 blowup}
Let $\text{x-yz} \in 
\{\text{2-15, 2-28, 3-6, 3-12}\}$. 
Then there exists a  Fano threefold of No.\ x-yz. 
\end{prop}

\begin{proof}
\underline{2-15, 2-28}: 
Let $D$ and $D'$ be smooth general hypersurfaces on $\P^3$ such that $\deg D =3$ and $\deg D' \in \{1, 2\}$. 
For $C := D \cap D'$, $X := \Bl_C\,\P^3$ is a Fano threefold by Proposition \ref{p MM 2.12}. 
If $\deg D' =2$ (resp.\ $\deg D'=1$), then $X$ is of No.\ 2-15 (resp.\ 2-28) \cite[Subsection 7.2]{FanoIV}.

\underline{3-6, 3-12}: 
Let $C \subset \P^3$ be either a smooth cubic rational curve or a 
quartic elliptic curve. 
In any case, $C \subset \P^3$ is an intersection of quadrics 
\cite[Proposition 4.12]{EH24}, 
and hence $\cI_{C/\P^3}(2)$ is globally generated. 
In particular, $\cI_{(L\amalg C)/\P^3}(3)$ is globally generated for a line $L \subset \P^3$ disjoint from $C$. 
Then $X := \Bl_{L\amalg C}\,\P^3$ is Fano by Proposition \ref{p MM 2.12}. 
If $C$ is a quartic elliptic curve (resp.\ a smooth cubic rational curve), then $X$ is of No.\ 3-6 (resp.\ 3-12) \cite[Subsection 7.3]{FanoIV}. 
\end{proof}

Concerning the smooth quadric threefold $Q$, 
the remaining maximal blowups are
2-21, 2-23, 3-10, 3-15, and 3-20, which are settled below. 



\begin{prop}
\label{p-existence-blowups-of-Q}
Let $\text{x-yz} \in 
\{\text{3-10, 3-15, 3-20}\}$. 
Then there exists a Fano threefold of 
No.\ x-yz. 
\end{prop}

\begin{proof}
Set $Q
:=
V\left(
x_0^2+x_1x_2+x_3x_4
\right)
\subset
\P^4$, which is a smooth quadric threefold. 
For $H:=\cO_Q(1)$, we have $-K_Q \sim 3H$. 
Take mutually disjoint curves $C_1$ and $C_2$ on $Q$ as follows: 
\begin{itemize}
\item 3-10: $C_1 := Q \cap V_1$ and $C_2 := Q \cap V_2$, where each $V_i$ is a general plane on $\P^4$. 
\item 3-15: 
Pick a line $C_1$ on $Q$. Set $C_2 := Q \cap V$, where $V$ is a general plane on $\P^4$. 
\item 3-20: 
Let $C_1$ and $C_2$ be mutually disjoint lines on $Q$ 
(e.g., $C_1 := \{ x_0 = x_1 =x_3=0\}$ and $C_2:= \{ x_0 =x_2=x_4=0\}$). 
\end{itemize}
For each $i \in \{1, 2\}$ and the smallest linear subvariety $\la C_i\ra $ in $\P^4$ containing $C_i$, we have $\la C_i \ra \cap Q = C_i$. 
Therefore, $I_{C_1 \amalg C_2/Q}(2)$ is globally generated. 
Then  $X := \Bl_{C_1 \amalg C_2}\,Q$ is Fano (Proposition \ref{p MM 2.12}), 
which is of No.\ 3-10, 3-15, or 3-20 depending on $(C_1, C_2)$ 
\cite[Subsection 7.3]{FanoIV}. 
\qedhere

\end{proof}

\begin{lem}\label{l 2-23 exist}
There exists a Fano threefold of No.\ 2-23. 
\end{lem}

\begin{proof}
Fix a smooth quadric threefold $Q \subset \P^4$ and set $H := \MO_Q(1)$. 
Take smooth general members $S \in |H|$ and $C \in |2H|_S|$. 
By $-K_Q \sim 3H$, 
both $-K_Q-S
\sim
2H$ and 
\[
-K_Q|_S-C \sim
3H|_S-C
\sim 3H|_S -2H|_S =H|_S
\]
are ample. 
Then $X:=\Bl_C\,Q$ is a Fano threefold (Proposition \ref{p prime div criterion}), 
which is of No.\ $2$-$23$ by \cite[Subsection 7.2]{FanoIV}.
\end{proof}

\begin{lem}\label{l 2-21 exist}
There exists a Fano threefold of No.\ 2-21. 
\end{lem}

\begin{proof}
Set 
\[
C
:=
\left\{
[s^4:s^3t:s^2t^2:st^3:t^4]
\ \middle|\
[s:t]\in\P^1
\right\}
\subset\P^4, 
\]
which is the rational normal quartic. 
In particular, $C$ is an intersection of quadrics, i.e., 
$\mathcal I_{C/\P^4}(2)$ is globally generated.
By substituting $[x_0:x_1:x_2:x_3:x_4]
=
[s^4:s^3t:s^2t^2:st^3:t^4]$, 
we see that $C$ is contained in the following smooth quadric threefold: 
\[
Q
:=
V\left(
x_0x_2-x_1^2
+x_0x_3-x_1x_2
+x_2x_4-x_3^2
\right)
\subset
\P^4. 
\]
Here the smoothness of $Q$ can be directly checked by Jacobian criterion. 
By the natural surjection $\mathcal I_{C/\P^4}(2)|_Q \twoheadrightarrow \mathcal I_{C/Q}(2)$, $\mathcal I_{C/Q}(2)$ is globally generated. 
Then Proposition \ref{p MM 2.12} implies that $X:=\operatorname{Bl}_C(Q)$ is a Fano threefold, 
which is of No.\ 2-21 \cite[Subsection 7.2]{FanoIV}. 
\qedhere
\end{proof}

\section{Proof of existence}

\subsection{$\rho = 4$}



\begin{prop}\label{p exist rho 4}
Let x-yz be a number with $x =4$, i.e., 
\[
\text{x-yz} \in 
\{\text{4-1, 4-2, 4-3, 4-4, 4-5, 4-6, 
4-7, 4-8, 4-9, 4-10, 4-11, 4-12, 4-13}\}. 
\]
Then there exists a Fano threefold of No.\ x-yz.
\end{prop}

\begin{proof}
As the case $\rho=5$ has been settled already (Proposition \ref{p exist rho geq 5}), 
it is enough to treat the cases 4-1, 4-2, 4-3, 4-5, 
4-6, 4-7, 4-8, 4-13 
(namely, 
the entries marked with \lq\lq none" in the \lq\lq blowups" column  \cite[Subsection 7.4]{FanoIV}). 
If $x{\h}yz \neq \text{4-2}$, then the assertion holds as follows: 
\begin{itemize}
\item 4-1: Proposition \ref{p exist as CI}. 
\item 4-3, 4-6, 4-8, 4-13: Proposition \ref{p existence 4-13}. 
\item 4-7: Proposition \ref{p-existence-blowups-of-W}. 
\item 4-5: Proposition \ref{p-existence-3-5-3-22}. 
\end{itemize}

In what follows, we treat the case $x{\h}yz = \text{4-2}$. 
Set $T := \P^1 \times \P^1$ and  $L:=\MO_{\P^1\times\P^1}(1,1)$. 
Let $\xi := \MO_Y(1)$ be the tautological line bundle of the $\P^1$-bundle $\pi : Y := \P_T(\MO_T \oplus L) \to T$ and let $S\subset Y$ be the section of $\pi$ corresponding 
to the projection $\MO_{\P^1\times\P^1}\oplus L \to L$. 
Take a general member $C\in|2\xi|_S| = |2L|$, 
which is an elliptic curve by 
the adjunction formula. 
Since $-K_T -L$ and $L-K_T-\pi(C)$ are ample, 
$X := \Bl_C\,Y$ is a Fano threefold (Corollary \ref{c prime div criterion to P1-bdl}), 
which is of No.\ 4-2 by  \cite[Proposition 5.29]{FanoIV}. 
\end{proof}

\subsection{$\rho=3$}

\begin{prop}\label{p existence 3-9}
There exists a smooth Fano threefold of No.\ $3$-$9$.
\end{prop}

\begin{proof}
Set $T := \P^2$ and $L:=\MO_{\P^2}(2)$. 
Let $\xi := \MO_Y(1)$ be the tautological line bundle of the $\P^1$-bundle $\pi : Y := \P_T(\MO_T \oplus L) \to T$ and let $S\subset Y$ be the section of $\pi$ corresponding to the projection $\MO_T\oplus L \to L$. 
Take a general member $C\in|2\xi|_S|=|2L|$, which is a smooth curve of genus three by the adjunction formula. 
Since
\[
-K_T-L\sim\MO_{\P^2}(1)
\qquad\text{and}\qquad
L-K_T-\pi(C)\sim\MO_{\P^2}(1)
\]
are ample, $X := \Bl_C\,Y$ is a Fano threefold by Corollary \ref{c prime div criterion to P1-bdl}. 
By \cite[Proposition 4.44]{FanoIV}, $X$ is of No.\ $3$-$9$.
\end{proof}

\begin{prop}\label{p 3-2 exist}
There exists a Fano threefold of No.\ $3$-$2$.
\end{prop}

\setcounter{step}{0}
\begin{proof}
Set $T:=\P^1_1\times\P^1_2$ and let 
\[
\pi_V\colon
V:=
\P_T
\left(
\MO_T(1,1)\oplus\MO_T^{\oplus2}
\right)
\longrightarrow
T
\]
be the $\P^2$-bundle. 
For $\MO_{\P^1_1}(1)_1 := \MO_{\P^1_1}(1)$ and $\MO_{\P^1_1}(1)_2 := \MO_{\P^1_1}(1)$, 
take 
\begin{itemize}
\item $\cF := \MO_{\P^1_1}(1)_1  \oplus \MO_{\P^1_1}(1)_2 \oplus\MO_{\P^1_1}^{\oplus2}$, and  
\item 
the $\P^3$-budnle 
$\pi_W : W := \P_{\P^1_1}(\cF) \to \P^1_1$. 
\end{itemize}
Let $\xi_V := \MO_V(1)$ and $\xi_W := \MO_W(1)$ be the tautological line bundles of $\pi_V$ and $\pi_W$, respectively. 
Set $E := \P_T(\MO_T^{\oplus 2})$ and 
$F := \P_{\P^1_1}(\MO_{\P^1_1}^{\oplus 2}) =\P^1_1 \times \P^1_3$, which are closed subschemes of $V$ and $W$, respectively. 
For each $i \in \{1, 2\}$, 
set $D_i := \P_{\P^1_1}( \cF/\MO_{\P^1_1}(1)_i) \simeq
\P_{\P^1_1}
\left(
\MO_{\P^1_1}(1)\oplus\MO_{\P^1_1}^{\oplus2}
\right)$.

\begin{step}\label{s1 3-2 exist}
There exists a commutative diagram 
\[
\begin{tikzcd}
E=\P^1_1\times\P^1_2\times\P^1_3
\arrow[r,"\psi"]
\arrow[d,hook]
&
F=\P^1_1\times\P^1_3
\arrow[d,hook]
\\
V
\arrow[r,"\varphi"]
\arrow[d,"\pi_V"']
&
W
\arrow[d,"\pi_W"]
\\
T=\P^1_1\times\P^1_2
\arrow[r,"\pr_1"]
&
\P^1_1
\end{tikzcd} 
\]
such that 
$\psi$ is the projection and 
$\varphi$ is the blowup along $F$ satisfying $\xi_V \sim \varphi^*\xi_W$ and $E = \Ex(\varphi)$. 
Moreover, $D_1 \cap D_2 = F$ and $D_i \sim \xi_W -\pi^*_W \MO_{\P^1_1}(1)$. 
\end{step}

\begin{proof}[Proof of Step \ref{s1 3-2 exist}]
By the composite surjective homomorphism
\[
\pi_V^*\pr_1^*\cF
=
\pi_V^*
(
\pr_1^*(\pr_1)_*\MO_T(1,1)\oplus\MO_T^{\oplus2}
)
\twoheadrightarrow
\pi_V^*
(
\MO_T(1,1)\oplus\MO_T^{\oplus2}
)
=
\pi_V^*(\pi_V)_*\xi_V
\twoheadrightarrow
\xi_V,
\]
the universal property of projective space bundles induces a morphism
$\varphi\colon V\to W$ satisfying
$\xi_V\sim\varphi^*\xi_W$ and
$\pi_W\circ\varphi=\pr_1\circ\pi_V$.
Thus the lower square is commutative.

Let us show that the upper square is commutative and 
$\varphi$ is the blowup along $F$ with $\Ex(\varphi)=E$.
Take an affine open subset $U\subset\P^1_1$ on which $\MO_{\P^1_1}(1)$ is trivial.
Then the base change of the above diagram by 
$U \to \P^1_1$ is identified with
\[
\begin{tikzcd} 
E_U :=U\times\P^1_2\times\P^1_3 \arrow[r, "\psi_U"] \arrow[d,hook] & U\times\P^1_3 =:F_U \arrow[d,hook] \\ 
V_U := U\times\P_{\P^1_2}(\MO_{\P^1_2}(1)\oplus\MO_{\P^1_2}^{\oplus2}) \arrow[r,"\varphi_U"] \arrow[d] & U\times\P^3 =:W_U \arrow[d] \\ T_U :=U\times\P^1_2 \arrow[r] & U. \end{tikzcd}
\]
Fix homogeneous coordinates $[s:t]$ on $\P^1_2$ and let
$[\lambda:\mu:\nu]$ be the fibre coordinate of
$\P_{\P^1_2}(\MO_{\P^1_2}(1)\oplus\MO_{\P^1_2}^{\oplus2})$.
Under the above trivialisations, the evaluation homomorphism
$\MO_{\P^1_2}^{\oplus2}\twoheadrightarrow\MO_{\P^1_2}(1)$
is given by $(a,b)\mapsto as+bt$, and hence
\[
\varphi_U
\colon 
U\times\P_{\P^1_2}(\MO_{\P^1_2}(1)\oplus\MO_{\P^1_2}^{\oplus2})  
\to U \times \P^3, \qquad 
([s:t],[\lambda:\mu:\nu])
\mapsto
[\lambda s:\lambda t:\mu:\nu].
\]
Moreover, 
$E_U =E\times_{\P^1_1}U$ is defined by $\lambda=0$, while
$F_U$ is defined by $x_0=x_1=0$.
From this local description, 
it is easy to see that the upper square is commutative and 
$\varphi$ is the blowup of $W$ along $F$ with $E=\Ex(\varphi)$.



Let us prove the \lq\lq moreover" part.
We have
\[
D_1\cap D_2
=
\P_{\P^1_1}
(
\cF/
(
\MO_{\P^1_1}(1)_1\oplus\MO_{\P^1_1}(1)_2
)
)
=
\P_{\P^1_1}
(
\MO_{\P^1_1}^{\oplus2}
)
=
F.
\]
As $D_i$ is the effective divisor corresponding to the element
$s_i\in H^0(W,\xi_W-\pi_W^*\MO_{\P^1_1}(1)_i)$
given by
$\pi_W^*\MO_{\P^1_1}(1)_i\to\pi_W^*\cF\twoheadrightarrow\xi_W$,
we get
$D_i\sim\xi_W-\pi_W^*\MO_{\P^1_1}(1)$.
This completes the proof of Step \ref{s1 3-2 exist}.
\end{proof}

Set
\[
A:=2\xi_W+\pi_W^*\MO_{\P^1_1}(1). 
\]
As $\xi_W$ is nef and $\pi_W$-ample, $A$ is ample. 
Since $W$ is a smooth projective toric variety, $|A|$ is very ample. 
Take a general member
$Y\in|A|$. 
By the Bertini theorem, 
all of $Y, S := Y \cap D_1$, and $C := Y \cap F$ are smooth projective  varieties. 
Therefore, 
\[
X:=\varphi^{-1}(Y)
\simeq
\Bl_C\,Y, 
\]
which is a smooth projective threefold. 
For each $i \in \{1, 3\}$, let $\pi_i : F = \P^1_1 \times \P^1_3 \to \P^1_i$ be the projection and 
we set $\MO_F(a, b) := \pi_1^*\MO_{\P^1_1}(a) \otimes \pi_3^*\MO_{\P^1_3}(b)$. 
It holds that $\varphi^*A
=
2\xi_V+\pi_V^*\MO_T(1,0)$.

\begin{step}\label{s2 3-2 exist}
The following hold: 
\begin{enumerate}
\item $Y \sim 2\xi_W + \pi_W^*\MO_{\P^1_1}(1)$. 
\item $X \sim 2\xi_V + \pi_V^*\MO_T(1, 0)$. 
\item $\pi_W^*\MO_{\P^1_1}(1)|_F \simeq\MO_F(1,0)$. 
\item $\xi_W|_F \simeq \MO_F(0, 1)$. 
\item $C \sim \MO_F(1, 2)$. 
\end{enumerate}
\end{step}

\begin{proof}[Proof of Step \ref{s2 3-2 exist}]
The assertion (1) follows from $Y \sim A = 2\xi_W+\pi_W^*\MO_{\P^1_1}(1)$. 
Then (2) holds by 
$X \sim \varphi^*Y \sim \varphi^*(2\xi_W+\pi_W^*\MO_{\P^1_1}(1)) 
\sim 2\xi_V +\pi^*_V\MO_T(1, 0)$. 
We have 
\[
\pi^*_W\MO_{\P^1_1}(1)|_F \simeq \pi_1^*\MO_{\P^1_1}(1) \simeq \MO_F(1, 0). 
\]
Thus (3) holds. 
As $\xi_F := \MO_F(0, 1)$ coincides with the tautological line bundle of 
the $\P^1$-budnle $F = \P_{\P^1_1}(\MO_{\P^1_1}^{\oplus 2}) \to \P^1_1$, we get 
$\xi_W|_F \simeq \xi_F = \MO_F(0, 1)$, and hence (4) holds. 
Finally, (5) follows from 
\[
C = Y \cap F \sim 
(2\xi_W + \pi_W^*\MO_{\P^1_1}(1))|_F \sim \MO_F(1, 2). 
\]
This completes the proof of Step \ref{s2 3-2 exist}. 
\end{proof}

\begin{step}\label{s3 3-2 exist}
$\xi_W|_Y$ is ample. 
\end{step}

\begin{proof}[Proof of Step \ref{s3 3-2 exist}]
Since $\cF
=
\MO_{\P^1_1}(1)^{\oplus2}\oplus\MO_{\P^1_1}^{\oplus2}$ is globally generated, 
$\xi_W=\MO_W(1)$ is globally generated. 
Suppose that there exists a curve $\Gamma$ on $Y$ such that $\xi_W \cdot \Gamma =0$. 
It suffices to derive a contradiction. 
As $\xi_W$ is $\pi_W$-ample, 
$\Gamma$ is not contained in any fibre of $\pi_W$, i.e., $\pi_W(\Gamma) = \P^1_1$. 
For the normalisation $\Gamma^N$ of $\Gamma$, 
we have the induced surjection 
\[
(\MO_{\P^1_1}(1)|_{\Gamma^N})^{\oplus 2}  \oplus 
(\MO_{\P^1_1}|_{\Gamma^N})^{\oplus2}
= \pi_W^*\cF|_{\Gamma^N} \twoheadrightarrow \xi_W|_{\Gamma^N}. 
\]
As $\deg (\xi_W|_{\Gamma^N})=0$ and 
$\MO_{\P^1_1}(1)|_{\Gamma^N}$ is ample, 
each ample direct summand $\MO_{\P^1_1}(1)|_{\Gamma^N}$ is mapped to zero on $\xi_W|_{\Gamma^N}$. 
Therefore, $\Gamma \subset D_1 \cap D_2 = F$. 
This, together with $\xi_W|_F \simeq \MO_F(0, 1)$, 
implies that $\Gamma \sim \MO_F(0, 1)$. 
On the other hand, we have 
\[
\Gamma \subset Y\cap F=C\in|\MO_F(1,2)|. 
\]
Hence we obtain  $\Gamma = C \sim \MO_F(1, 2)$, contradicting $\Gamma \sim \MO_F(0, 1)$. 
This completes the proof of Step \ref{s3 3-2 exist}. 
\end{proof}

\begin{step}\label{s4 3-2 exist}
$-K_X$ is ample. 
\end{step}

\begin{proof}[Proof of Step \ref{s4 3-2 exist}]
As $\pi_W : W =\P_{\P^1_1}(\cF) \to \P^1_1$ is a $\P^3$-bundle, 
$K_W
\sim
-4\xi_W + \pi_W^*(K_{\P^1_1} +\det (\cF)) \sim -4\xi_W$. 
Since $Y\sim 2\xi_W+\pi_W^*\MO_{\P^1_1}(1)$, 
the adjunction formula gives 
\[
-K_Y
\sim -(K_W+Y)|_Y 
\sim 4\xi_W|_Y -  (2\xi_W+\pi_W^*\MO_{\P^1_1}(1))|_Y 
= (2\xi_W-\pi_W^*\MO_{\P^1_1}(1))|_Y. 
\]
This, together with $S =D_1|_Y 
\sim
(
\xi_W-\pi_W^*\MO_{\P^1_1}(1))|_Y$, implies 
\[
-K_Y-S  \sim 
(2\xi_W-\pi_W^*\MO_{\P^1_1}(1))|_Y
-(
\xi_W-\pi_W^*\MO_{\P^1_1}(1))|_Y
=\xi_W|_Y. 
\]
Since $F=D_1\cap D_2$, we have
$C = F \cap Y = D_1 \cap D_2 \cap Y = D_2 \cap S = D_2|_S 
\sim
(
\xi_W-\pi_W^*\MO_{\P^1_1}(1)
)|_S$. 
Consequently, 
\[
-K_Y|_S-C \sim (2\xi_W-\pi_W^*\MO_{\P^1_1}(1))|_S - (
\xi_W-\pi_W^*\MO_{\P^1_1}(1))|_S
\sim 
\xi_W|_S. 
\]
As $\xi_W|_Y$ is ample, both $-K_Y-S$ and $-K_Y|_S -C$ are ample. 
Therefore, $X =\Bl_C\,Y$ is Fano by Proposition \ref{p prime div criterion}.
This completes the proof of Step \ref{s4 3-2 exist}. 
\end{proof}

\begin{step}\label{s5 3-2 exist}
$X$ is a Fano threefold of No.\ 3-2. 
\end{step}

\begin{proof}[Proof of Step \ref{s5 3-2 exist}]
Twisting by $\MO_T(-1,-1)$, 
we have an 
isomorphism of $T$-schemes
\[
V = \P_T\left(
\MO_T(1,1) \oplus \MO_T^{\oplus2}
\right)
\simeq
\P_T
\left(
\MO_T\oplus\MO_T(-1,-1)^{\oplus2}
\right)=:P.
\]
For the induced $\P^2$-bundle $\pi_P : P \to T$ and 
its tautological line bundle $\xi_P$, 
we have 
$
\xi_V\otimes\pi_V^*\MO_T(-1,-1)
\simeq
\xi_P$ via the above isomorphism. 
Hence $X$ is identified with a smooth member of
\[
\left|
2\xi_P+\pi_P^*\MO_T(3,2)
\right|.
\]
Then $(-K_X)^3 = 14$ by the same argument as in 
in \cite[Proposition 6.16]{FanoIII}. 
Hence $X$ is a Fano threefold of No.\ $3$-$2$ \cite[Subsection 7.3]{FanoIV}. 
This completes the proof of Step \ref{s5 3-2 exist}. 
\end{proof}
Step \ref{s5 3-2 exist} completes the proof of 
Proposition \ref{p 3-2 exist}.  
\end{proof}

\begin{prop}\label{p exist rho 3} 
Let x-yz be a number with $x=3$, i.e.,
\[
\text{x-yz}
\in
\{
\text{3-1, 3-2, ..., 3-30, 3-31}
\}.
\]
Then there exists a Fano threefold of No.\ x-yz.
\end{prop}

\begin{proof}
As the case $\rho=4$ has been settled already 
(Proposition \ref{p exist rho 4}), it is enough to treat the 
following cases \cite[Subsection 7.3]{FanoIV}:
\[
\begin{gathered}
3\text{-}1,\quad
3\text{-}2,\quad
3\text{-}3,\quad
3\text{-}4,\quad
3\text{-}5,\quad
3\text{-}6,\quad
3\text{-}7,\quad
3\text{-}8,\quad
3\text{-}9,\quad
3\text{-}10,\\
3\text{-}11,\quad
3\text{-}12,\quad
3\text{-}13,\quad
3\text{-}14,\quad
3\text{-}15,\quad
3\text{-}16,\quad
3\text{-}20,\quad
3\text{-}22,\quad
3\text{-}23,\quad
3\text{-}29,
\end{gathered}
\]
namely, the entries marked with \lq\lq none" in the \lq\lq blowups" column of \cite[Subsection 7.3]{FanoIV}.
Except for $3\text{-}4$ and $3\text{-}8$, the assertion holds as follows: 
\begin{itemize}
\item 3-1: Proposition \ref{p double cover exist}. 
\item 3-2: Proposition \ref{p 3-2 exist}. 
\item 3-3: Proposition \ref{p exist as CI}.
\item 3-5, 3-22: Proposition \ref{p-existence-3-5-3-22}.
\item 3-6, 3-12: Proposition \ref{p P^3 blowup}.
\item 3-7, 3-11: Lemma \ref{l dP3 elliptic exist}.
\item 3-9: Proposition \ref{p existence 3-9}. 
\item 3-10, 3-15, 3-20: Proposition \ref{p-existence-blowups-of-Q}.
\item 3-13: Proposition \ref{p-existence-blowups-of-W}.
\item 3-14, 3-16, 3-23: Proposition \ref{p-existence-blowups-of-V7}.
\item 3-29: Proposition \ref{p exist as toric}.
\end{itemize}


In what follows, we treat the cases 3-4 and 3-8. 
Let $Y$ be a Fano threefold of No.\ 2-18 or 2-24, 
whose existence is guaranteed by Proposition \ref{p double cover exist} 
and 
Proposition \ref{p exist as CI}, respectively. 
Let $g\colon Y\to\P^2$ be the unique conic bundle of type $C_1$ 
\cite[Subsection 7.2]{FanoIV}. 
Set $H := g^*\MO_{\P^2}(1)$. 
Let $h : Y \to Z$  be the contraction of the extremal ray not corresponding to $g$. 
For the ample generator $H_Z$ of $\Pic Z(\simeq \Z)$ and 
$H' := h^*H_Z$, 
the following holds \cite[Proposition 5.9(3)]{FanoIII}:
\[
-K_Y\sim2H+H'.
\]
Take a smooth fibre $C$ of $g$. 
Since $C$ is the scheme-theoretic intersection of two members of $|H|$, the ideal sheaf $\cI_{C/Y}(H)$ is globally generated. Moreover, $-K_Y-H\sim H+H'$ is ample. 
Proposition \ref{p MM 2.12} shows that $\Bl_C\,Y$ is Fano. If $Y$ is of No.\ 2-18 (resp.\ 2-24), then $\Bl_C\,Y$ is of No.\ 3-4 (resp.\ 3-8) by \cite[Subsection 7.3]{FanoIV}.
\end{proof}

\subsection{$\rho=2$}
\begin{prop}\label{p exist rho 2} 
Let x-yz be a number with $x=2$, i.e.,
\[
\begin{gathered}
\text{x-yz}
\in
\{
\text{2-1, 2-2, ..., 
2-36}
\}.
\end{gathered}
\]
Then there exists a Fano threefold of No.\ x-yz.
\end{prop}

\begin{proof}
As the case $\rho=3$ has been settled already (Proposition \ref{p exist rho 3}), it is enough to treat the cases
\[
\begin{gathered}
2\text{-}1,\quad
2\text{-}2,\quad
2\text{-}3,\quad
2\text{-}4,\quad
2\text{-}5,\quad
2\text{-}6,\quad
2\text{-}7,\quad
2\text{-}8,\quad
2\text{-}9,\quad
2\text{-}10,\quad
2\text{-}11,\quad
2\text{-}12,\quad
2\text{-}13,\\
2\text{-}14,\quad
2\text{-}15,\quad
2\text{-}16,\quad
2\text{-}17,\quad
2\text{-}19,\quad
2\text{-}20,\quad
2\text{-}21,\quad
2\text{-}22,\quad
2\text{-}23,\quad
2\text{-}26,\quad
2\text{-}28,
\end{gathered}
\]
namely, the entries marked with \lq\lq none" in the \lq\lq blowups" column of \cite[Subsection 7.2]{FanoIV}. The assertion holds by the results listed below. 
\begin{itemize}
\item 2-1, 2-3, 2-5, 2-10, 2-14: Lemma \ref{l dP3 elliptic exist}.
\item 2-2, 2-8: Proposition \ref{p double cover exist}.
\item 2-4, 2-6, 2-7, 2-9, 2-12, 2-13, 2-17: Proposition \ref{p exist as CI}.
\item 2-11, 2-16, 2-19, 2-20, 2-22, 2-26: Proposition \ref{p-existence-blowups-of-Vd}.
\item 2-15, 2-28: Proposition \ref{p P^3 blowup}.
\item 2-21: Lemma \ref{l 2-21 exist}.
\item 2-23: Lemma \ref{l 2-23 exist}.
\end{itemize}
\end{proof}

\begin{thm}
Let x-yz be an arbitrary number listed in \cite[Section 7]{FanoIV}. 
Then there exists a Fano threefold of No.\ x-yz.
\end{thm}

\begin{proof}
The assertion holds by the results listed below. 
\begin{itemize}
\item $\rho=1$: Proposition \ref{p exist rho 1}. 
\item $\rho=2$: Proposition \ref{p exist rho 2}. 
\item $\rho=3$: Proposition \ref{p exist rho 3}. 
\item $\rho=4$: Proposition \ref{p exist rho 4}. 
\item $\rho\geq 5$: Proposition \ref{p exist rho geq 5}. 
\end{itemize}
\end{proof}

\section{Irreducibility criteria and some isolated cases}

For a fixed number x-yz, we say that $\pi : \cX \to U$ is a {\em parameter space} of No.\ x-yz if  $\pi$ is a smooth projective morphism of quasi-projective varieties over $k$ (and hence irreducible) such that 
\begin{itemize}
\item the fibre of $\pi$ over every closed point is a Fano threefold, and 
\item every Fano threefold $X$ of No. x-yz is  isomorphic to  the fibre of $\pi$ over some closed point. 
\end{itemize}
The goal of the remaining part of this paper is to prove the existence of parameter spaces for all numbers. 
In this section, we collect some criteria and settle some isolated cases. 

\subsection{Complete intersections and double covers}

\begin{prop}\label{p irre CI}
Let 
\[
\text{x-yz}
\in
\{
\text{1-3, 1-4, 1-5, 1-6, 1-7, 1-8, 1-9, 1-10, 1-11, 1-13, 1-14,}
\]
\[
\text{2-24, 3-2, 3-3, 3-8, 3-17, 4-1}
\}.
\]
Then there exists a parameter space of No.\ x-yz. 
\end{prop}

\begin{proof}
Except for 1-10, 
there exist a projective variety $P$ and 
Cartier divisors $L_1, ..., L_r$ on $P$ such that  
every Fano threefold $X$ of No.\ x-yz is isomorphic to 
a complete intersection $D_1 \cap \cdots \cap D_r$ 
for some members $D_1 \in |L_1|, ..., D_r \in |L_r|$; 
see \cite[Theorem 5.4]{KTLift1} (1-5), 
\cite{KTprime} (1-6, 1-7, 1-8, 1-9), and \cite[Section 7]{FanoIV}. 
In this case, there exists an open subset 
\[
U \subset \P(H^0(P, L_1)) \times \cdots \times \P(H^0(P, L_r))
\]
which is a parameter space of No. x-yz. 
Concerning 1-10, we can apply the same argument by using a suitable vector bundle instead of line bundles \cite{KTprime}. 
\end{proof}

\begin{prop}\label{p irre double cover}
Let 
$\text{x-yz} \in \{\text{1-1, 1-12, 2-2, 2-8, 2-18, 3-1}\}$. 
Then there exists a parameter space of No.\ x-yz. 
\end{prop}

\begin{proof}
In any case, 
there exists a smooth projective threefold $Y$ and a line bundle $\cL$ on $Y$ such that 
every Fano threefold $X$ of No.\ x-yz admits a finite double cover 
$f: X \to Y$ satisfying $\cL^{-1} \simeq f_*\MO_X/\MO_Y$. 
In this case, we can find $a \in H^0(Y, \cL)$ and $b \in H^0(Y, \cL^{\otimes 2})$ such that 
\[
X \simeq \{ z^2 + az + b=0\}. 
\]
Therefore, there exists an open subset 
$U \subset H^0(Y, \cL) \times H^0(Y, \cL^{\otimes 2})$ which is a parameter space of No. x-yz. 
\end{proof}

\subsection{Relative blowup criteria}

\begin{lem}\label{l relative blowup parameter space}
Fix a number $\text{x-yz}$. 
Let $U$ be a quasi-projective variety, 
let $\pi\colon
\cY
\to 
U$ be a smooth projective morphism, and let $\cC
\subset
\cY$ be a closed subscheme such that the induced morphism 
$\cC
\to 
U$ is smooth. 
Assume that, 
given a Fano threefold $X$ of No.\ x-yz, 
there exists a closed point $u \in U$ such that 
$\cC_u$ is of pure one-dimensional and 
$X \simeq \Bl_{\cC_u}\, \cY_u$. 
Then there exists a parameter space of No.\ $\text{x-yz}$.
\end{lem}

\begin{proof}
After shrinking $U$, 
we may assume that 
$\cY_u$ is a smooth projective threefold and 
$\cC_u$ is a non-empty disjoint union of smooth curves for every closed point $u \in U$. 
Set $\cX
:=
\Bl_{\cC}\,\cY$. 
Since blowups of smooth schemes along smooth centres commute with 
(not necessarily flat) base changes \cite[Tag 0H1G]{SP}, we have
$\cX_u
\simeq
\Bl_{\cC_u}\cY_u$ for every closed point
$u\in U$. 
Moreover, $\cX\to U$ is a smooth projective morphism.
Then the assertion holds by replacing $U$ by a suitable open subset.
\qedhere 




\end{proof}

\begin{prop}\label{p relative CI centre parameter space}
Fix a number $\text{x-yz}$. 
Let $U$ be a quasi-projective variety and let
$\pi\colon
\cY
\longrightarrow
U$ be a smooth projective morphism such that 
$\cY_u$ is a Fano threefold for every closed point $u \in U$. 
Let $\cL_1$ and $\cL_2$ be $\pi$-nef line bundles on $\cY$. 
Set
\[
V
:=
\P_U(\cE_1^\vee)
\times_U \P_U(\cE_2^\vee), 
\]
where $\cE_i
:=
\pi_*\cL_i$ for every $i$, which is a vector bundle. 
For each $i \in \{1, 2\}$, let $\cD_i
\subset
\cY\times_U V$ be the universal divisor associated with $\cL_i$ 
and set
$\cC
:=
\cD_1\cap \cD_2
\subset
\cY\times_U V$. 
Assume that, given a Fano threefold $X$ of No.\ x-yz, 
there exists a closed point $v \in V$ such that 
$X \simeq \Bl_{\cC_v}(\cY_v)$ and 
$\cC_v$ is a non-empty disjoint union of smooth curves. 
Then there exists a parameter space of No.\ $\text{x-yz}$.
\end{prop}

\begin{proof}
Note that $V$ is a quasi-projective variety. 
Then the assertion follows from
Lemma \ref{l relative blowup parameter space} applied to
$\cY \times_U V
\longrightarrow
V$ 
and $\cC
\subset
\cY \times_U V$ after applying the base change $(-) \times_V V'$ 
for the largest open subset $V' \subset V$ such that $\cC \times_V V' \to V'$ is smooth. 
\end{proof}

\begin{prop}\label{p moving CI containing moving centre}
Fix a number $\text{x-yz}$. 
Let $U$ be a quasi-projective variety and let
$\rho\colon
\cP
\longrightarrow
U$ be a smooth projective morphism. 
Let
$\cC
\subset
\cP$ be a closed subscheme such that
$\cC
\longrightarrow
U$
is smooth. 
Let $\cL_1,\ldots,\cL_r$ be line bundles on $\cP$. 
Assume that, for every $i\in\{1,\ldots,r\}$, the function
\[
U
\longrightarrow
\Z,
\qquad
u
\longmapsto
h^0\left(
\cP_u, 
\cI_{\cC_u/\cP_u}
\otimes
\cL_i|_{\cP_u}
\right)
\]
is constant. 
Set $\cE_i
:=
\rho_*
\left(
\cI_{\cC/\cP}
\otimes
\cL_i
\right)$, which is a vector bundle on $U$. 
Set 
\begin{itemize}
\item $V
:=
\P_U(\cE_1^\vee)
\times_U
\cdots
\times_U
\P_U(\cE_r^\vee)$, 
\item $\cP_V
:=
\cP\times_U V$, and 
$\cC_V
:=
\cC\times_U V
\subset
\cP_V$. 
\end{itemize}
For every $i\in\{1,\ldots,r\}$, let
$\cD_i
\subset
\cP_V$ be the universal divisor containing $\cC_V$, and set
\[
\cY
:=
\cD_1\cap\cdots\cap\cD_r
\subset
\cP_V. 
\]
Assume that, given a Fano threefold $X$ of No.\ x-yz, 
there exists a closed point $v \in V$ 
such that 
$\cY_v$ is a smooth projective threefold, 
$X \simeq \Bl_{(\cC_V)_v}(\cY_v)$, and 
$(\cC_V)_v$ is a non-empty disjoint union of smooth curves. 
Then there exists a parameter space of No.\ $\text{x-yz}$.
\end{prop}

\begin{proof}
Note that $V$ is a quasi-projective variety. 
As $u \mapsto h^0\left(
\cP_u,
\cI_{\cC_u/\cP_u}
\otimes
\cL_i|_{\cP_u}
\right)$ is constant, 
$\cE_i
:=
\rho_*
\left(
\cI_{\cC/\cP}
\otimes
\cL_i
\right)$ is a vector bundle such that 
the induced linear map 
\[
\cE_i\otimes k(u)
\longrightarrow
H^0\left(
\cP_u,
\cI_{\cC_u/\cP_u}
\otimes
\cL_i|_{\cP_u}
\right)
\]
is an isomorphism for every closed point $u\in U$. 
Then  the assertion follows from
Lemma \ref{l relative blowup parameter space}.
\end{proof}




\subsection{1-2 and 2-6}

\begin{prop}\label{p parameter 2-6}
There exists a parameter space of No.\ 2-6. 
\end{prop}

\begin{proof}
Set $P:=\P^2 \times \P^2$, $H:=\MO_{\P^2 \times \P^2}(1, 1)$, and 
\[
V^{\all}
:=
H^0(P,H)
\oplus
H^0\left(
P,H^{\otimes2}
\right)
\oplus
H^0(P,H)
\oplus
\mathbb A^1, 
\]
which is an affine space. 
Let
\[
\rho \colon
\operatorname{Tot}_P(H)
:=
\operatorname{Spec}_P 
\left(
\bigoplus_{d \geq 0} H^{-d}
\right)
\longrightarrow
P
\]
be the total space of $H$ and let
$z
\in
H^0\left(
\operatorname{Tot}_P(H),
\rho^*H
\right)$ 
be the tautological section. For
a closed point $(a,b,c,r)\in V^{\all}$, let $X_{a,b,c,r}
\subset
\operatorname{Tot}_P(H)$ be the closed subscheme defined by
\[
X_{a,b,c,r} := \{ z^2+az+b=rz+c=0\} \subset \Tot_P(H). 
\]
Then there exists the closed subscheme $\cX^{\all} \subset \operatorname{Tot}_P(H) \times V^{\all}$ such that 
the fibre of the induced composition 
\[
\pi^{\all}: \cX^{\all} \hookrightarrow \operatorname{Tot}_P(H) \times V^{\all} \xrightarrow{\pr_2} V^{\all}
\]
over a closed point $(a,b,c,r)\in V^{\all}$ is equal to $X_{a, b, c, r}$. 
For the largest open subset $V$ 
over which $\pi^{\all}$ is flat, 
we have the induced flat projective morphism 
\[
\pi : \cX \to V, 
\]
where $\cX := \cX^{\all} \times_{V^{\all}} V$. 
Since $\dim \Tot_P(H) = \dim P +1=5$, 
a closed point $(a, b, c, r) \in V^{\all}$ is contained in $V$ if and only if $\dim X_{a, b, c, r}=3$. 
By \cite[Subsection 7.2]{FanoIV}, it suffices to show (i) and (ii) below. 
\begin{enumerate}
\item[(i)] 
If $X$ is a smooth divisor on $\P^2 \times \P^2$ of bidegree $(2, 2)$, 
then $X \simeq X_{a, b, c, r}$ for some $(a, b, c, r) \in V^{\all}$. 
\item[(ii)] 
If $f: X \to W$ is a split double cover with $f_*\MO_X/\MO_W \simeq H^{-1}|_W$, 
then  $X \simeq X_{a, b, c, r}$ for some $(a, b, c, r) \in V^{\all}$. 
\end{enumerate}

Let us show (i). 
In this case, we have $X = V(b) \in |H^{\otimes 2}|$ for some $b \in H^0\left(
P,H^{\otimes2}
\right)$. 
Then 
\[
X_{0, b, 0, 1} = (\{ z^2 + b= z=0 \} \subset \Tot_P(H)) \xrightarrow{\simeq} (\{ b=0\} \subset P)  = V(b) =X,
\]
where the isomorphism $\xrightarrow{\simeq}$ is induced by the projection $\Tot_P(H) \to P$. 
Thus (i) holds.

Let us show (ii).
We have
$W=V(c)\in|H|$ for some $c\in H^0(P,H)$. 
Fix a split double cover
$f\colon
X
\to 
W$. 
Then $X$ is defined by
\[
z^2+sz+t=0
\]
for some $s\in H^0\left(
W,H|_W
\right)$ and $t\in H^0\left(
W,H^{\otimes2}|_W
\right)$. 
By $H^1(P, H^m)=0$ and the exact sequence $0 \to H^{\otimes m} \to H^{\otimes (m+1)} \to H^{\otimes (m+1)}|_W \to 0$, 
there exist
$a\in H^0(P,H)$ and $b\in H^0\left(
P,H^{\otimes2}
\right)$ 
such that
$a|_W=s$ and $b|_W=t$. 
Then
\[
X_{a,b,c,0}
=
\left(
\left\{
z^2+az+b=c=0
\right\}
\subset
\Tot_P(H)
\right)
\xrightarrow{\simeq}
X,
\]
where the isomorphism is induced by the identification
$\rho^{-1}(W)
\simeq
\Tot_W(H|_W)$. Thus (ii) holds.
\end{proof}

\begin{prop}\label{p parameter 1-2}
There exists a parameter space of No.\ 1-2. 
\end{prop}
\begin{proof}
The same argument as in Proposition \ref{p parameter 2-6} works by setting 
$P := \P^4$ and $H := \MO_{\P^4}(2)$. 
\end{proof}

\subsection{Easy cases and $\rho=1$}

\begin{prop}\label{p unique easy}
Let 
\[
\text{x-yz}
\in
\{
\text{1-15, 1-16, 1-17, 2-27, 2-30, 2-31, 2-32, 2-33, 2-34, 2-35, 2-36,}
\]
\[
\text{3-16, 3-22, 3-23, 3-24, 3-25, 3-26, 3-27, 3-28, 3-29, 3-30, 3-31,}
\]
\[
\text{4-6, 4-7, 4-8, 4-9, 4-10, 4-11, 4-12, 5-2, 5-3, 6-1}
\}.
\]
Then there exists a unique Fano threefold $X_{x{\text -}yz}$ of No.\ x-yz up to isomorphisms. 
In particular, there exists a parameter space of No.\ x-yz. 
\end{prop}

\begin{proof}
We treat the following five cases separately: 
\begin{enumerate}
\item $\text{1-16, 1-17, 2-34, 2-35, 2-36, 3-27, 3-28, 3-31, 4-10, 5-3, 6-1}$. 
\item
$\text{2-27, 2-30, 2-33, 3-16, 3-22, 3-23, 3-25, 3-26, 3-29, 3-30, 
4-9, 4-11, 4-12, 5-2}$. 
\item 1-15,  2-31, 2-32.  
\item 4-6, 4-8. 
\item 3-24, 4-7. 
\end{enumerate}
The case (1) is obvious by \cite[Section 7]{FanoIV}. 

As for (2), the uniqueness holds by using suitable  projective
equivalence of the corresponding blowup centres. 
Here we only treat 3-16 and 5-2 as the other cases are simpler. 
The uniqueness for No.\ 3-16 follows from the fact that 
$\Aut(\P^3)$ acts on the set 
\[
\{ (C, P) \,|\, P \in C, C \text{ is a smooth cubic rational curve on }\P^3\}
\]
transitively. 
The uniqueness of No.\ 5-2 follows from the fact that $\Aut(\P^3)$ acts transitively on ordered pairs of disjoint lines $(L,L')$, and the stabiliser of such a pair induces the full automorphism group of $L (\simeq \P^1)$. Hence it acts transitively on pairs of distinct points on $L$, and therefore on pairs of distinct fibres of the exceptional divisor lying over $L$.

Concerning (3), the uniqueness holds as follows: 
\begin{itemize}
\item 1-15: \cite[Theorem 1.3]{IKTTV5}.
\item 2-31:  \cite[Theorem B.3]{Lan83}.
\item 2-32: \cite[Lemma 5.16(1)]{FanoIII}.
\end{itemize}

Let us treat the case (4). 
For a Fano threefold $X$ of No.\ 4-6, 
we have $X \simeq \Bl_C(\P^1_1 \times \P^1_2 \times \P^1_3)$ for some smooth curve $C$ of tridegree $(1, 1, 1)$. 
Note that each projection $\pi_i : C \to \P^1_i$ is an isomorphism. 
Then the image of the composite closed immersion 
\[
C \hookrightarrow \P^1_1 \times \P^1_2 \times \P^1_3 
\xrightarrow{\simeq, \pi_1^{-1} \times \pi_2^{-1} \times \pi_3^{-1}} C \times C \times C 
\]
is the diagonal $\{(t, t, t) \in C \times C \times C \,|\, t \in C\}$. 
Therefore, $X \simeq  \Bl_C(\P^1_1 \times \P^1_2 \times \P^1_3) \simeq \Bl_{\Delta}( C \times C \times C)$. 
This completes the proof for the case of No.\ 4-6. 
A similar argument is applicable for the case of No.\ 4-8.

Let us show the uniqueness for (5). 
In what follows, we treat only No.\ 4-7, as the other case is easier. 
We use the following incidence description:  
\[
W
=
\left\{
(p,\ell)
\in
\P^2\times(\P^2)^\vee
\ \middle|\
p\in\ell
\right\}. 
\]
Let $\pi_1\colon W \to \P^2$ 
and $\pi_2\colon W\to (\P^2)^\vee$ 
be the induced contractions. 
By the description in
\cite[Subsection 7.4]{FanoIV}, every Fano threefold of No.\ 4-7 is
isomorphic to
\[
\Bl_{C_1\amalg C_2}W,
\]
where $C_1=\pi_1^{-1}(p)$ and $C_2=\pi_2^{-1}([\ell])$ 
for a point $p\in\P^2$ and a line $\ell\subset\P^2$ satisfying
$p\notin\ell$. 


Let $V$ be a three-dimensional vector space. 
Then a point $p$ (resp.\ a line $\ell$) on $\P(V) (\simeq \P^2)$ corresponds to a 
one-dimensional (resp.\ two-dimensional) vector subspace 
$P$ (resp.\ $L$) of $V$. 
Under this correspondence, 
$p \not\in \ell$ if and only if $P \cap L = \{0\}$. 
Therefore,  it is enough to show that $\GL(V)$ acts transitively on the following set: 
\[
\Sigma := 
\{ (P, L) \,|\, 
P\text{ and } L\text{ are vector subspaces of }V, 
\dim P = 1, \dim L =2, P \cap L = 0\}.  
\]
Pick two elements $(P, L), (P', L') \in \Sigma$. 
Then the required transitivity can be checked by fixing linear bases of $P, L, P', L'$. 
\qedhere


\end{proof}

\begin{thm}\label{t parameter rho 1}
Let x-yz be a number with $x = 1$, i.e., 
\[
\text{x-yz} \in 
\{\text{1-1, 1-2, ..., 1-17}\}. 
\]
Then there exists a parameter space of No.\ x-yz.
\end{thm}

\begin{proof}
The assertion holds as follows: 
\begin{itemize}
\item 1-3, 1-4, 1-5, 1-6, 1-7, 1-8, 1-9, 1-10, 1-11, 1-13, 1-14: 
Proposition \ref{p irre CI}. 
\item 1-1, 1-12: Proposition \ref{p irre double cover}. 
\item 1-15, 1-16, 1-17: Proposition \ref{p unique easy}. 
\item 1-2: Proposition \ref{p parameter 1-2}.
\end{itemize}
\end{proof}

\subsection{$\rho\geq 5$}

\begin{prop}\label{p irre 5-1}
Let $\text{x-yz}
\in
\{
\text{3-18, 4-4, 5-1}\}$. 
Then there exists a parameter space  of No.\ x-yz. 
\end{prop}

\begin{proof}
In what follows, we only treat the case when x-yz is 5-1, as the other cases are similar. 
Fix a smooth quadric threefold $Q \in \P^4$. 
Since smooth conics on $Q$ are obtained as the scheme-theoretic intersection $Q \cap \P^2$ for a plane $\P^2 \subset \P^4$, 
we can find a non-empty open subset $U \subset \Gr(3, 5)$ and a closed subscheme $\cC \subset Q \times U$ such that 
every smooth conic $C \subset Q$ is obtained as a fibre over some closed point of 
$\cC \hookrightarrow Q \times U \to U.$

Set 
\[
V := (\cC \times_U \cC  \times_U \cC)  \setminus \bigcup_{i \neq j} \Delta_{ij}, 
\]
which parametrises  $(C; P_1, P_2, P_3)$, 
where $C$ is a smooth conic and 
$(P_1, P_2, P_3)$ is a triple consisting of mutually distinct points. 
Hence we have the universal smooth conic $\cC_V := \cC \times_U V \subset Q \times V$ and 
the universal points $\cP_1, \cP_2, \cP_3 \subset \cC_V$. 
As $\cC \times_U \cC  \times_U \cC$  is irreducible, so is $V$. 
Let $\rho: \cY \to Q \times V$ be the blowup along $\cC \times_U V$. 
Then the assertion holds by Lemma \ref{l relative blowup parameter space} 
applied to $\cC := \cE_1 \amalg \cE_2 \amalg \cE_3$, 
where $\cE_i := \rho^{-1}(\cP_i)$. 
\end{proof}

\begin{thm}\label{t parameter rho 5}
Let x-yz be a number with $x \geq  5$. 
Then there exists a parameter space of No.\ x-yz.
\end{thm}

\begin{proof}
If $x \geq 6$, then the assertion follows from 
\cite[Subsection 7.6]{FanoIV} and the corresponding result for smooth del Pezzo surfaces. 
When $x=5$, the assertion holds by 
Proposition \ref{p irre 5-1} (No.\ 5-1) and 
Proposition \ref{p unique easy} (No.\ 5-2, No.\ 5-3). 
\end{proof}

\section{Irreducibility for $\rho=2$}

\subsection{2-9, 2-12, 2-13, 2-17}

\begin{dfn}\label{d Ugdn}
    Take non-negative integers $g, d, n$. 
    Let $U_{g, d, n} \subset \Hilb^{dt+1-g}(\P^n)$ 
be the open subset that parametrises all the smooth curves $C \subset \P^n$ such that $\deg C = d$, $g(C)=g$, and the restriction homomorphism
\[
H^0\left(
\P^n,
\MO_{\P^n}(1)
\right)
\longrightarrow
H^0\left(
C,
\MO_{\P^n}(1)|_C
\right)
\]
is an isomorphism. 
\end{dfn}

\begin{lem}\label{l Mg H^1=0-case}
Take non-negative integers $g, d, n$ satisfying $g+n =d$. 
Then $U_{g, d, n}$ is either empty or irreducible (cf.\ Definition \ref{d Ugdn}). 
\end{lem}

\begin{proof}
Let $\cM_g$ be the moduli stack of smooth curves of genus $g$, 
which is known to be a smooth irreducible Artin stack. 
and let $\cC_g \to \cM_g$ be the universal curve. 
Let
\[
\mathcal{P}ic_g^d
\longrightarrow
\mathcal M_g
\]
be the universal (unrigidified) Picard stack, 
where 
$\mathcal{P}ic_g^d$ parametrises pairs $(C,L)$ consisting of 
a smooth curve $C$ of genus $g$ and a line bundle $L$ on $C$ of degree $d$. 
Since the relative Picard stack is an Artin stack smooth over $\cM_g$
\cite[Tags 0D04 and 0DPJ]{SP}, 
$\mathcal{P}ic_g^d$ is a smooth Artin stack. 
For a geometric point $[C] \to \cM_g$ represented by a smooth projective curve $C$ of genus $C$, 
the base change $\cP ic_{C}^d := \mathcal{P}ic_g^d \times_{\cM_g} [C]$ 
is the Picard stack parametrising the line bundles on $C$ of degree $d$. 
We have the induced morphism $\pi : \cP ic_{C}^d \to \Pic_C^d$ to the Picard variety $\Pic_C^d$, which is a smooth projective variety. 
Since $\pi$ is a $\bG_m$-gerbe \cite[Tag 0DME]{SP}, 
$\cP ic_{C}^d$ is irreducible \cite[Tag 06R9]{SP}. 
Therefore, $\mathcal{P}ic_g^d$ is a smooth irreducible Artin stack. 
The universal curve 
$\cC_g \times_{\cM_g} \mathcal{P}ic_g^d$ 
over $\mathcal{P}ic_g^d$ has the universal line bundle of degree $d$.





Let
\[
\mathcal P_{g,d}
\subset
\mathcal{P}ic_g^d
\]
be the open Artin substack parametrising the pairs $(C,L)$ such that 
\begin{enumerate}
\item[($\star$)] $H^1(C, L)=0$ and $|L|$ is very ample. 
\end{enumerate}
We have the universal curve $q\colon
\mathcal C_{g,d}
\longrightarrow
\mathcal P_{g,d}$ for $\cC_{g, d} := \cC_g \times_{\cM_g} \cP_{g, d}$ 
and the universal line bundle $\mathcal L$ on $\cC_{g, d}$ of degree $d$. 
Note that $q :\mathcal C_{g,d}
\to 
\mathcal P_{g,d}$ is representable, smooth, and proper of
relative dimension one, because so is $\cC_g \to \cM_g$. 
For every geometric point $(C,L)$ of $\mathcal P_{g,d}$, we have 
$H^1(C, L)=0$ by ($\star$), and hence 
the Riemann--Roch theorem implies 
\[
h^0(C,L)
= \chi(C, L) = d-g+1 =n+1. 
\]
Then it follows from cohomology and base change that $\mathcal E
:=
q_*\mathcal L$ is a vector bundle of rank $n+1$ on $\mathcal P_{g,d}$. 

Define the projective frame bundle by
\[
\mathcal F_{g,d,n}
:=
\operatorname{Isom}_{\mathcal P_{g,d}}
\left(
\P(\mathcal E^\vee),
\P^n_{\mathcal P_{g,d}}
\right) \to \cP_{g, d}.
\]
Over a geometric point $(C,L) \to \cP_{g, d}$ of 
$\cP_{g, d}$, its geometric fibre 
$\mathcal F_{g,d,n} \times_{\cP_{g, d}} (C, L)$ 
parametrises isomorphisms
\[
\P\left(
H^0(C,L)^\vee
\right)
\xrightarrow{\simeq}
\P^n
\]
between projective spaces. 
The morphism $\mathcal F_{g,d,n}
\to 
\mathcal P_{g,d}$ 
is a representable $\operatorname{PGL}_{n+1}$-torsor. 
In particular, $\cF_{g, d, n}$ is an Artin stack \cite[Tag 05UM]{SP}. 
By ($\star$), the complete linear system $|L|$ induces a closed immersion
$\iota_L : C
\hookrightarrow
\P\left(
H^0(C,L)^\vee
\right)$, 
and a projective frame induces a projective isomorphism
$\theta: \P\left(
H^0(C,L)^\vee
\right)
\xrightarrow{\simeq}
\P^n.$
Thus we obtain the composite  closed immersion
\[
\theta \circ \iota_L:C 
\overset{\iota_L}{\hookrightarrow} \P\left(
H^0(C,L)^\vee
\right)
\xrightarrow{\theta, \simeq}
\P^n.  
\]
Therefore, we get a morphism of Artin stacks
\[
\epsilon : \mathcal F_{g,d,n}
\to 
U_{g,d,n}, \qquad (C, L, \theta) \mapsto 
[\theta\circ\iota_L(C) \subset \P^n]. 
\]

Let us show that $\epsilon$ is surjective. 
Take a geometric point $[C\subset\P^n]
\in
U_{g,d,n}$.  
For $L :=
\MO_{\P^n}(1)|_C$, 
$|L|$ is very ample and we have $h^0(C, L) = n+1$ by Definition \ref{d Ugdn}. 
As  $g+n=d$,   the Riemann--Roch theorem implies $h^1(C,L)=0$. 
Therefore, $[C \subset \P^n]$ is contained in the image of $\epsilon$. 
Therefore, $\epsilon$ is surjective.

We have 
$\epsilon\colon
\mathcal F_{g,d,n}
\twoheadrightarrow 
U_{g,d,n}$ and 
\[
\cF_{g, d, n} \xrightarrow{\PGL_{n+1}\text{-torsor}} \cP_{g, d} 
\overset{{\rm open}}{\subset} \mathcal{P}ic_g^d.
\]
Assume that $U_{g, d, n} \neq \emptyset$. 
Then $\cF_{g, d, n} \neq \emptyset$. 
Hence $\cP_{g, d} \neq \emptyset$. 
As $\mathcal{P}ic_g^d$ is irreducible, so is $\cP_{g, d}$. 
Since 
$\cF_{g, d, n} \to \cP_{g, d}$ is a $\PGL_{n+1}$-torsor, 
$\cF_{g, d, n}$ is irreducible. 
Therefore, its image $U_{g, d, n}$ is irreducible, as required. 
\end{proof}

\begin{lem}\label{l Mg 573}
$U_{5,7,3}$ is either empty or irreducible  (cf.\ Definition \ref{d Ugdn}). 
\end{lem}

\begin{proof}
Let $\cM_{5,1}$ be the moduli stack of pointed smooth curves $(C,p)$ of
genus $5$. Recall that $\cM_{5,1}$ is a smooth irreducible Artin stack. 
We have the universal curve $q\colon
\cC_{5,1}
\longrightarrow
\cM_{5,1}$ and the universal section $\sigma\colon
\cM_{5,1}
\longrightarrow
\cC_{5,1}$. 
By abuse of notation, the image of $\sigma$ is also denoted by $\sigma$.

Let
\[
\cP_{5,1}
\subset
\cM_{5,1}
\]
be the open Artin substack parametrising pointed smooth curves $(C,p)$
such that $|K_C-p|$ is very ample. 
We denote the restrictions of $q$ and $\sigma$ to $\cP_{5,1}$ by the
same symbols. Set 
\[
\cL
:=
\omega_q(-\sigma).
\]
For every geometric point $(C,p)$ of $\cP_{5,1}$, we have
$h^0\left(
C,
\MO_C(p)
\right)
=
1$ and the Riemann--Roch theorem implies 
\[
h^0\left(
C,
K_C-p
\right) -1 =\chi(C, K_C-p) 
=\chi(C, \MO_C) + 
\deg(K_C-p) =
(1-5) + (8-1)
=
3.
\]
It follows from cohomology and base change that
$\cE
:=
q_*\cL$ is a vector bundle of rank $4$ on $\cP_{5,1}$.

Define the projective frame bundle by
\[
\cF_{5,7,3}
:=
\operatorname{Isom}_{\cP_{5,1}}
\left(
\P(\cE^\vee),
\P^3_{\cP_{5,1}}
\right)
\longrightarrow
\cP_{5,1}.
\]
Over a geometric point $(C,p) \to \cP_{5,1}$ of $\cP_{5,1}$, its
geometric fibre $\cF_{5,7,3}
\times_{\cP_{5,1}}
(C,p)$ 
parametrises isomorphisms
$\P\left(
H^0\left(
C,
K_C-p
\right)^\vee
\right)
\xrightarrow{\simeq}
\P^3$ between projective spaces. The morphism
$\cF_{5,7,3}
\to 
\cP_{5,1}$ 
is a representable $\PGL_4$-torsor.
In particular, $\cF_{5,7,3}$ is an Artin stack. 

The complete linear system $|K_C-p|$ induces a closed immersion
\[
\iota_{C,p}\colon
C
\hookrightarrow
\P\left(
H^0\left(
C,
K_C-p
\right)^\vee
\right),
\]
and a projective frame induces an isomorphism
$\theta\colon
\P\left(
H^0\left(
C,
K_C-p
\right)^\vee
\right)
\xrightarrow{\simeq}
\P^3$ between projective spaces. Thus we obtain the composite closed
immersion
\[
\theta\circ\iota_{C,p}\colon
C
\hookrightarrow
\P\left(
H^0\left(
C,
K_C-p
\right)^\vee
\right)
\xrightarrow{\simeq}
\P^3.
\]
Therefore, we get a morphism of Artin stacks
\[
\epsilon\colon
\cF_{5,7,3}
\longrightarrow
U_{5,7,3},
\qquad
(C,p,\theta)
\longmapsto
\left[
\theta\circ\iota_{C,p}(C)
\subset
\P^3
\right].
\]

Let us show that $\epsilon$ is surjective. Take a geometric point
$[C\subset\P^3]
\in
U_{5,7,3}$. 
Set $L
:=
\MO_{\P^3}(1)|_C$.  
By assumption, $h^0(C,L)=4$. 
Since $\deg L=7$ and $g(C)=5$, the Riemann--Roch theorem gives
$h^1(C,L)
=
1$, which implies 
$h^0(C,K_C-L)
=
1$ by Serre duality. 
Since $\deg(K_C-L)
=
2g(C)-2-\deg L
=
1$, there exists a unique closed point $p\in C$ satisfying 
$K_C-L
\sim p$. 
Hence $|K_C-p| =|L|$ is very ample. 
Then $(C,p)$ is a geometric point of $\cP_{5,1}$, and hence 
$[C\subset\P^3]$ is contained in the image of $\epsilon$. Therefore, $\epsilon$ is surjective.


We have 
$\epsilon\colon
\cF_{5,7,3}
\twoheadrightarrow 
U_{5,7,3}$ and 
\[
\cF_{5,7,3}
\xrightarrow{\PGL_4\text{-torsor}}
\cP_{5,1}
\overset{{\rm open}}{\subset}
\cM_{5,1}.
\]
Assume that $U_{5,7,3}\neq\emptyset$. Then
$\cF_{5,7,3}\neq\emptyset$. 
Hence $\cP_{5,1}\neq\emptyset$. As $\cM_{5,1}$ is irreducible, so is
$\cP_{5,1}$. Since
$\cF_{5,7,3}
\longrightarrow
\cP_{5,1}$ 
is a $\PGL_4$-torsor, $\cF_{5,7,3}$ is irreducible. Therefore, its
image $U_{5,7,3}$ is irreducible, as required.
\end{proof}

\begin{prop}\label{p parameter spaces rho 2 Mg}
Let
$\text{x-yz}
\in
\{
\text{2-9, 2-12, 2-13, 2-17}
\}.$ 
Then there exists a parameter space of No.\ $\text{x-yz}$.
\end{prop}

\begin{proof}
By the blowup descriptions in \cite[Subsection 7.2]{FanoIV}, 
every Fano threefold of No.\ x-yz is  isomorphic to $\Bl_C Y$, 
where the pair $(Y, C)$ is given as follows. 
\begin{table}[htbp]
\centering
\[
\begin{array}{c|c|c|c}
\text{No.}
&
g(C)
&
\deg C
&
Y
\\
\hline
\text{2-9}
&
5
&
7
&
\P^3
\\
\text{2-12}
&
3
&
6
&
\P^3
\\
\text{2-13}
&
2
&
6
&
Q
\\
\text{2-17}
&
1
&
5
&
Q
\end{array}
\]
\caption{Blowup descriptions for No.\ 2-9, 2-12, 2-13, and 2-17.}
\label{table rho 2 Mg}
\end{table}
Here $Q$ denotes a smooth quadric hypersurface.

For
$(g,d,n)
\in
\{
(5,7,3),
(3,6,3),
(2,6,4),
(1,5,4)
\}$, recall that 
\[
U_{g,d,n}
\subset
\Hilb^{dt+1-g}(\P^n)
\]
is the open subset introduced in Definition \ref{d Ugdn}. 
By Lemma \ref{l for 8.4} below, 
every smooth curve $C$ appearing in Table \ref{table rho 2 Mg}
defines a closed point of the corresponding $U_{g,d,n}$. 
Note that $U_{g,d,n}$ is either empty or irreducible 
(Lemma \ref{l Mg H^1=0-case}, Lemma \ref{l Mg 573}). 
Let $\cC_{g,d,n}
\subset
\P^n\times U_{g,d,n}$ be the universal curve. 
Concerning  the cases No.\ $2$-$9$ and No.\ $2$-$12$, 
the assertion holds by applying Lemma \ref{l relative blowup parameter space} to $\cY := \P^3 \times U_{g, d, 3}$ and $\cC_{g,d,3}
\subset \P^3\times U_{g,d,3} = \cY$.



In what follows, we treat the cases No.\ $2$-$13$ and No.\ $2$-$17$. 
Let $\cC_{g,d,4}
\subset
\P^4\times U_{g,d,4}$ be the universal curve. Set
\[
\rho\colon
\cP
:=
\P^4\times U_{g,d,4}
\longrightarrow
U_{g,d,4}
\]
to be the second projection and set
$\cL
:=
\pr_1^*\MO_{\P^4}(2)$. 

Fix a closed point $u\in U_{g,d,4}$ and set
\[
C
:=
\left(
\cC_{g,d,4}
\right)_u
\qquad\text{and}\qquad
L
:=
\MO_{\P^4}(1)|_C.
\]
Since
$d
\geq
2g+1$, the embedding $C \subset \P^4$ is projectively normal 
\cite[Corollary in page 55]{Mum70}. 
In particular, $\Sym^2 H^0(C, L) \to H^0(C, L^{\otimes 2})$ is surjective. By Definition \ref{d Ugdn}, 
we have the induced isomorphism: 
$\rho_1: H^0\left(
\P^4,
\MO_{\P^4}(1)
\right)
\xrightarrow{\simeq}
H^0(C,L)$. 
Then the restriction map $\rho_2 : H^0\left(
\P^4,
\MO_{\P^4}(2)
\right) \to  H^0(C,L^{\otimes 2})$ is surjective by the following commutative diagram: 
\[
\begin{tikzcd}
\Sym^2 H^0\left(
\P^4,
\MO_{\P^4}(1)
\right) \arrow[r, "\Sym^2 \rho_1", "\simeq"'] \arrow[d] & \Sym^2 H^0(C,L) \arrow[d, twoheadrightarrow] \\
H^0\left(
\P^4,
\MO_{\P^4}(2)
\right) \arrow[r, "\rho_2"] &  H^0(C,L^{\otimes 2}). 
\end{tikzcd}
\]
As $\deg L^{\otimes2}
=
2d
>
2g-2$, the Riemann--Roch theorem gives
\[
h^0\left(
C,
L^{\otimes2}
\right)
=\chi \left(
C,
L^{\otimes2}
\right)
= \deg  L^{\otimes2} + 1-g = 2d-g+1.
\]
Then it holds that
\[
h^0\left(
\P^4,
\cI_{C/\P^4}(2)
\right)
=
h^0\left(
\P^4,
\MO_{\P^4}(2)
\right) - h^0\left(
C,
L^{\otimes2}
\right)
= 15 -(2d-g+1). 
\]
Thus the function 
\[
U_{g,d,4}
\to \Z,\qquad u
\mapsto 
h^0\left(
\P^4,
\cI_{\left(\cC_{g,d,4}\right)_u/\P^4}(2)
\right)
\]
is constant. 
Then the assertion holds by Proposition
\ref{p moving CI containing moving centre}  applied with
$r=1$ and $\cL_1=\cL$. 
\qedhere

\end{proof}

\begin{lem}\label{l for 8.4}
Let $C \subset \P^n$ be a smooth curve listed as in 
Table \ref{table rho 2 Mg}, where $n:=3$ for 2-9 and 2-12 (resp.\ $n:=4$ for 2-13 and 2-17). 
Then the restriction map 
\[
\rho: H^0(\P^n, \MO_{\P^n}(1)) \to H^0(C, \MO_{\P^n}(1)|_C)
\]
is an 
isomorphism. 
\end{lem}

\begin{proof}Set $L:=\MO_{\P^n}(1)|_C$ and let x-yz be the number. 
We now finish the proof by assuming $(\star)$ below. 
\begin{enumerate}
\item[($\star$)] $\rho$ is injective. 
\end{enumerate}
It is enough to show that $h^0(C, L)=n+1$. 
If x-yz is not 2-9, then we have $H^1(C, L) =0$ by Serre duality. 
In this case, the Riemann-Roch theorem implies 
\[
h^0(C, L) = \chi(C, L) = \deg C +1-g(C) = n+1.
\]
as required.

Assume that x-yz is 2-9. 
By ($\star$), we have $h^0(C, L) \geq h^0(\P^3, \MO_{\P^3}(1)) =4$. 
It suffices to prove $h^0(C, L) \leq 4$. 
The Riemann-Roch theorem implies 
\[
h^0(C, L) -h^1(C, L) = \chi(C, L) = 1-g(C) +\deg C =3. 
\]
Hence it is enough to show $(h^1(C, L)=) h^0(C, K_C-L) \leq 1$, 
which follows from 
$g(C) >0$ and $\deg (K_C-L) = 2g(C) -2 - \deg C = 10-2 -7 =1$. 
This completes the proof under assuming $(\star)$.

\medskip

It is enough to prove $(\star)$, i.e., $\rho$ is injective. 
Otherwise, we have $C \subset \P^{n-1}$ for some hyperplane $\P^{n-1}$ on $\P^n$. 
If $n=3$ (i.e., x-yz is 2-9 or 2-12), 
then 
$(g(C), \deg(C)) \in \{(5,7), (3, 6)\}$ contradicts the genus formula $2g(C)- 2 = \deg(C) (\deg (C)-3)$. 
Assume that $n=4$ (i.e., x-yz is 2-13 or 2-17). 
Then $C \subset Q \cap \P^3 =:S$. 
Here  $S$ is a prime divisor on $Q$, and hence a (possibly singular) quadric surface.

Assume that $S$ is smooth, i.e., $S\simeq \P^1\times\P^1$. 
Write $C\sim aF_1+bF_2$, where $F_1$ and $F_2$ are the fibres of the two projections. 
We have $C \sim \MO_{\P^1 \times \P^1}(a, b)$ for some $a \geq 0$ and $b \geq 0$. 
It is easy to derive a contradiction by using  $\deg C = a+b$ and 
\[
2g(C) -2 = (K_S + C) \cdot C = -2 \deg C + C^2. 
\]

In what follows, we assume that $S$ is a singular quadric surface. 
Let $\mu : S' \to S$ be the minimal resolution. 
It is well known that $S' \simeq \P_{\P^1}(\MO_{\P^1} \oplus \MO_{\P^1}(2))$ and 
$E := \Ex(\mu)$ is a smooth rational curve with $E^2 = -2$. 
For the proper transform $C'$ of $C$ on $S'$, 
$\mu|_{C'} : C' \xrightarrow{\simeq} C$. 
Moreover, either $C' \cap E = 0$ or 
$C' \cdot E=1$ (indeed, if $C' \cdot E \geq 2$, then 
the scheme-theoretic intersection $C' \cap E$ is of length $\geq 2$ and mapped to the closed point $\mu(E)$, contradicting  
$\mu|_{C'} : C' \xrightarrow{\simeq} C$). 
Hence $C'^2 = (\mu^*C -aE)^2 = C^2 + a^2 E^2 = C^2 - b$ 
for  $(a, b) \in \{ (0, 0), (1/2, 1/2)\}$. 
As $K_{S'} \sim \mu^*K_S$,  the adjunction formula implies 
\[
2g(C) -2  = 2g(C') -2 = (K_{S'} + C') \cdot C' = 
K_{S'} \cdot C' + C'^2  
\]
\[
= K_S \cdot C + C^2 -b = -2 \deg C + \frac{1}{2} (\deg C)^2 -b.  
\]
However, this is impossible by 
$(g(C), \deg C) \in \{(2, 6), (1, 5)\}$ and $ b\in \{0, 1/2\}$. 
\end{proof}

\subsection{Proof of irreducibility ($\rho=2$)}

\begin{prop}\label{p parameter spaces rho 2 CI centres}
Let
\[
\text{x-yz}
\in
\{
\text{2-1, 2-3, 2-4, 2-5, 2-7, 2-10, 2-14, 2-15, 2-23, 2-25, 2-28, 2-29}\}.
\]
Then there exists a parameter space of No.\ $\text{x-yz}$.
\end{prop}

\begin{proof}
By the blowup descriptions in \cite[Subsection 7.2]{FanoIV}, 
the assertion follows from  Proposition
\ref{p relative CI centre parameter space}. 
In what follows, we only treat the case of No.\ $2$-$1$ to illustrate how to apply Proposition
\ref{p relative CI centre parameter space}.

Let $X$ be a Fano threefold of
No.\ $2$-$1$. 
By \cite[Subsection 7.2]{FanoIV}, there exist a del Pezzo
threefold $Y_{\text{1-11}}$ of No. 1-11 and 
members $D_1, D_2 \in |H|$ such that 
$C :=D_1 \cap D_2$ is a smooth curve and 
$X \simeq \Bl_C\,Y_{\text{1-11}}$.


By Theorem \ref{t parameter rho 1}, there exists a parameter space 
$\alpha : \cY_{\text{1-11}} \to U_{\text{1-11}}$ of No.\ 1-11. 
We may assume that  
$\Pic \cY_{\text{1-11}} \to \Pic Y_{\text{1-11}}$ is surjective for the fibre of every closed point (Proposition \ref{p Pic lift alteration}). 
Then we can find an $\alpha$-ample Cartier divisor $\cH$ 
on $\cY_{\text{1-11}}$ 
such that $-K_{(\cY_{\text{1-11}})_u} \sim 2\cH|_{(\cY_{\text{1-11}})_u}$ for every closed point $u \in U_{\text{1-11}}$. 
Then the assertion holds by applying Proposition
\ref{p relative CI centre parameter space} 
for $\cL_1 := \cH$ and $\cL_2 := \cH$. 
\qedhere





\end{proof}

\begin{prop}\label{p parameter spaces rho 2 moving ambient}
Let
$\text{x-yz}
\in
\{
\text{2-11, 2-16, 2-19, 2-20, 2-21, 2-22, 2-26}
\}$. 
Then there exists a parameter space of No.\ $\text{x-yz}$.
\end{prop}

\begin{proof}
By the blowup descriptions in \cite[Subsection 7.2]{FanoIV}, 
every Fano
threefold of No.\ x-yz 
is isomorphic to the
blowup of a smooth complete-intersection $Y$ in $P$ along a smooth
curve $C$, where $(P, Y, C)$ is described as follows:
\[
\begin{array}{c|c|c|c}
\text{No.}
&
P
&
C\subset P
&
Y\subset P
\\
\hline
\text{2-11}
&
\P^4
&
\text{a line}
&
Y\in|\MO_{\P^4}(3)|
\\
\text{2-16}
&
\P^5
&
\text{a conic}
&
Y=D_1\cap D_2,\quad D_i\in|\MO_{\P^5}(2)|
\\
\text{2-19}
&
\P^5
&
\text{a line}
&
Y=D_1\cap D_2,\quad D_i\in|\MO_{\P^5}(2)|
\\
\text{2-20}
&
\Gr(2,5)
&
\text{a twisted cubic}
&
Y=D_1\cap D_2\cap D_3,\quad
D_i\in|\MO_{\Gr(2,5)}(1)|
\\
\text{2-21}
&
\P^4
&
\text{a rational normal quartic}
&
Y\in|\MO_{\P^4}(2)|
\\
\text{2-22}
&
\Gr(2,5)
&
\text{a conic}
&
Y=D_1\cap D_2\cap D_3,\quad
D_i\in|\MO_{\Gr(2,5)}(1)|
\\
\text{2-26}
&
\Gr(2,5)
&
\text{a line}
&
Y=D_1\cap D_2\cap D_3,\quad
D_i\in|\MO_{\Gr(2,5)}(1)|.
\end{array}
\]
In any case, we can apply Proposition \ref{p moving CI containing moving centre}. 

In what follows, we provide the details only for the case of No.\ $2$-$20$, as the other cases are easier. 
Take the Pl\"{u}cker embedding $\Gr(2, 5) \subset \P^9$. 
Let $U \subset \operatorname{Mor}_3\left(\P^1,\Gr(2,5)\right)$ be the open subset that parametrises all the closed immersions $\P^1 \hookrightarrow \Gr(2,5)$ whose images are twisted cubic curves. 
Then $U$ is irreducible by \cite[Theorem 2.1]{Str87}.
Let $\cC
\subset
\Gr(2,5)\times U$ be the universal curve. Let
\[
\rho\colon
\cP :=
\Gr(2,5)\times U
\longrightarrow
U
\]
be the second projection. For every $i\in\{1,2,3\}$, set
$\cL_i
:=
\pr_1^*\MO_{\Gr(2,5)}(1)$.

Fix a closed point $u\in U$ and set $C:=\cC_u$. 
In order to apply Proposition \ref{p moving CI containing moving centre}, 
it suffices to show that 
$h^0\left(
\Gr(2,5),
\cI_{C/\Gr(2,5)}
\otimes
\MO_{\Gr(2,5)}(1)
\right)$ does not depend on $u$. 
In what follows, we prove $h^0\left(
\Gr(2,5),
\cI_{C/\Gr(2,5)}
\otimes
\MO_{\Gr(2,5)}(1)
\right)=
6$. 
Note that $C$ is a smooth cubic rational curve on $\P^3 := \la C \ra$, 
where $\la C \ra$ denotes the smallest linear subvariety on $\P^9$ containing $C$. 
We have $\MO_{\Gr(2,5)}(1)|_C
\simeq
\MO_{\P^1}(3)$ and the restriction map 
\[
\alpha : H^0\left(
\Gr(2,5),
\MO_{\Gr(2,5)}(1)
\right)
\longrightarrow
H^0\left(
C,
\MO_{\Gr(2,5)}(1)|_C
\right)
\]
is surjective by the following commutative diagram: 
\[ 
\begin{tikzcd}
& H^0\left( \Gr(2,5), \MO_{\Gr(2,5)}(1) \right) \arrow[dr, "\alpha"] & \\ 
H^0\left( \P^9, \MO_{\P^9}(1) \right) \arrow[ur,"\simeq"] \arrow[dr,two heads] && H^0\left( C, \MO_{\Gr(2,5)}(1)|_C \right) \\ 
& H^0\left( \P^3, \MO_{\P^3}(1) \right). \arrow[ur,"\simeq"] & \end{tikzcd} 
\]
Then we get the required equality 
$h^0\left(
\Gr(2,5),
\cI_{C/\Gr(2,5)}
\otimes
\MO_{\Gr(2,5)}(1)
\right)
=
6$ by 
\begin{itemize}
\item $h^0\left(
C,
\MO_{\Gr(2,5)}(1)|_C
\right)
= h^0(\P^1, \MO_{\P^1}(3)) = 4$, 
\item 
$h^0\left(
\Gr(2,5),
\MO_{\Gr(2,5)}(1)
\right)
=
10$, and 
\item the exact sequence 
\[
0 \to H^0\left(
\Gr(2,5),
\cI_{C/\Gr(2,5)}
\otimes
\MO_{\Gr(2,5)}(1)
\right) \to 
H^0\left(
\Gr(2,5),
\MO_{\Gr(2,5)}(1)
\right)
\]
\[\xrightarrow{\alpha} H^00\left(
C,
\MO_{\Gr(2,5)}(1)|_C
\right) \to 0. 
\]
\end{itemize}




\end{proof}

\begin{thm}\label{t parameter rho 2}
Let x-yz be a number with $x = 2$, i.e., 
\[
\text{x-yz} \in 
\{\text{2-1, 2-2, ..., 2-36}\}. 
\]
Then there exists a parameter space of No.\ x-yz.
\end{thm}
\begin{proof}
The assertion holds as follows: 
\begin{itemize}
\item 2-2, 2-8, 2-18:
Proposition \ref{p irre double cover}. 
\item 2-24:
Proposition \ref{p irre CI}. 
\item 2-6: Proposition \ref{p parameter 2-6}. 
\item 2-27, 2-30, 2-31, 2-32, 2-33, 2-34, 2-35, 2-36:
Proposition \ref{p unique easy}. 
\item 2-9, 2-12, 2-13, 2-17:
Proposition \ref{p parameter spaces rho 2 Mg}. 
\item 2-1, 2-3, 2-4, 2-5, 2-7, 2-10, 2-14, 2-15, 2-23, 2-25, 2-28, 2-29:
Proposition \ref{p parameter spaces rho 2 CI centres}. 
\item 2-11, 2-16, 2-19, 2-20, 2-21, 2-22, 2-26:
Proposition \ref{p parameter spaces rho 2 moving ambient}. 
\end{itemize}
\end{proof}

\section{Irreducibility for $\rho=3$ and $\rho=4$}

\subsection{Blowup along a curve inside a fixed surface}

\begin{lem}\label{l fixed lin sys}
Fix a number x-yz. 
Let $Y$ be a smooth projective threefold. 
Let $S$ be a normal prime divisor on $Y$ and let $D$ be a Cartier divisor on $S$. 
Assume that every Fano threefold of No.\ x-yz is isomorphic to $\Bl_C Y$ for some smooth member $C \in |D|$. 
Then there exists a parameter space of No.\ x-yz. 
\end{lem}

\begin{proof}
The assertion holds by applying Lemma \ref{l relative blowup parameter space} to a suitable open subset $U$ of $\P(H^0(S, D))$. 
\end{proof}

\begin{prop}\label{p rho=3 fixed lin sys}
Let 
$\text{x-yz}
\in
\{
\text{3-5, 3-13, 3-21}
\}.$ 
Then there exists a parameter space of No.\ x-yz. 
\end{prop}

\begin{proof}
Fix a contraction $\pi : W \to \P^2$. 
By the blowup descriptions in
\cite[Subsection 7.3]{FanoIV}, 
every Fano threefold $X$ of No.\ x-yz is isomorphic to $\Bl_C\,Y$ 
for some smooth member $C \in |D|$, 
where 
the triple $(Y, S, |D|)$  is given in the following table. 
\[
\begin{array}{c|c|c|c}
\text{No.}
&
\text{ambient }Y
&
\text{surface }S
&
|D|
\\
\hline
\text{3-5}
&
\text{2-34}
=
\P^2\times\P^1
&
B\times\P^1
\simeq
\P^1\times\P^1,
B\subset\P^2
\text{ a smooth conic}
&
\left|
\MO_{\P^1\times\P^1}(1,5)
\right|
\\
\text{3-13}
&
\text{2-32}
=
W
&
\pi^{-1}(B)
\simeq
\P^1\times\P^1, 
B\subset\P^2
\text{ a smooth conic}
&
\left|
\MO_{\P^1\times\P^1}(1,1)
\right|
\\
\text{3-21}
&
\text{2-34}
=
\P^2\times\P^1
&
B\times\P^1
\simeq
\P^1\times\P^1,
B\subset\P^2
\text{ a line}
&
\left|
\MO_{\P^1\times\P^1}(1,2)
\right|
\end{array}
\]
In order to apply Lemma \ref{l fixed lin sys}, 
it suffices to show that $S$ can be fixed. 
To this end, it is enough to check that $B$ can be fixed. 
This is clear for No.\ 3-5 and No.\ 3-21. 
Concerning No.\ $3$-$13$, 
this follows from the fact that 
$W
=
\left\{
x^ty=0
\right\}
\subset
\P^2_x\times\P^2_y$ is stable under 
the automorphism
\[
(x,y)
\longmapsto
\left(
(A^t)^{-1}x,
Ay
\right)
\]
for every $A
\in
\GL_3$. 
\qedhere

\end{proof}

\begin{prop}\label{p rho=4 tridegree}
Let
$\text{x-yz}
\in
\{
\text{4-3, 4-6, 4-8, 4-13}
\}.$ 
Then there exists a parameter space of No.\ x-yz.
\end{prop}

\begin{proof}
Define $d$ by 
\[
d:=0\quad(\text{if $x$-$yz$ is 4-8}), \qquad
d:=1\quad(\text{if $x$-$yz$ is 4-6}),
\]
\[
d:=2\quad(\text{if $x$-$yz$ is 4-3}), \qquad
d:=3\quad(\text{if $x$-$yz$ is 4-13}).
\]
Set
$Y
:=
\P^1_1
\times
\P^1_2
\times
\P^1_3$ and $
H_i
:=
\pr_i^*\MO_{\P^1_i}(1)$ for the $i$-th projection 
$\pr_i\colon
Y
\longrightarrow
\P^1_i$. 
For the diagonal $\Delta
\subset
\P^1_1\times\P^1_2$, set 
\[
Y \supset S
:=
\Delta\times\P^1_3 \simeq \P^1 \times \P^1_3,
\]
$M
:=
H_1|_S
=
H_2|_S$, and $N:=
H_3|_S.$

Let $X$ be a Fano threefold of No.\ x-yz. By the blowup descriptions in
\cite[Subsection 7.4]{FanoIV}, there exists a smooth curve 
$C
\subset
Y$ of tridegree
$(1,1,d)$ 
such that $X
\simeq
\Bl_C Y$. 
After applying a suitable automorphism on $\P^1 \times \P^1 \times \P^1$, 
we may assume that $C \subset S$. 
As $C$ is of tridegree $(1, 1, d)$, we obtain $C \in |dM+N|$. 
Then the assertion holds by applying Lemma \ref{l fixed lin sys} 
to $D := dM+N$. 
\qedhere


%
\end{proof}

\subsection{Blowup along a disjoint union of curves}

\begin{lem}\label{l disjoint centre}
Let x-yz, x'-y'z', $x_1$-$y_1z_1$, and $x_2$-$y_2z_2$ be numbers. 
Assume that 
the following hold: 
\begin{enumerate}
\item 
$Z$ is a Fano threefold of No.\ x'-y'z', 
every Fano threefold of No.\ x'-y'z' is isomorphic to $Z$, and 
the automorphism scheme $\Aut(Z)$ is integral. 
\item 
For each $i \in \{1, 2\}$, 
there exists a parameter space of No.\ $x_i$-$y_iz_i$. 
\item 
For each $i \in \{1, 2\}$ and every Fano threefold $Y_i$ of No.\ $x_i$-$y_iz_i$, 
there exists  a unique extremal ray $F$ such that 
$Z_i \simeq Z$ for the contraction $Y_i \to Z_i$ of $F$. 
\item 
Every Fano threefold $X$ of No.\ x-yz is isomorphic to $Y_1 \times_Z Y_2$ for some contractions $Y_1 \to Z$ and $Y_2 \to Z$. 
\end{enumerate}
Then there exists a parameter space of No.\ x-yz. 
\end{lem}

\begin{proof}
By standard argument, we may assume that ${\rm trans.deg}(k/\F_p)=\infty$. 
For each $i\in\{1,2\}$, take a parameter space
$\alpha_i\colon
\cY_i
\longrightarrow
U_i$ of No.\ $x_i$-$y_iz_i$. 
We may assume that the restriction group homomorphism
$\Pic(\cY_i)
\to \Pic(Y_i)$ is surjective for every geometric fibre $Y_i$ of $\alpha_i$ (Proposition \ref{p Pic lift alteration}).
By (1) and (3), there exists a contraction 
\[
\pi_i\colon
\cY_i
\longrightarrow
\cZ_i
\]
over $U_i$ such that the geometric fibre
$(\cZ_i)_u$ over every geometric point $u\in U_i$
is isomorphic to $Z$ (Proposition \ref{p ext type family}, Theorem \ref{t No inv}).

Since $Z$ is a smooth Fano variety, 
$\Aut(Z)$ is an affine group scheme of finite type over $k$, which is integral by (1). 
Let 
\[
q_i: V_i
:=
\operatorname{Isom}_{U_i}
\left(
\cZ_i,
Z\times U_i
\right) \to U_i  
\]
be the open subscheme of the Hom scheme $\Hom_{U_i}(\cZ_i, Z \times U_i)$ parametrising $U_i$-isomorphisms. 
As $\omega_{\cZ_i/U_i}^{-1}$ and $\omega_{Z \times U_i/U_i}^{-1}$ are ample over $U_i$, 
$\operatorname{Isom}_{U_i}
\left(
\cZ_i,
Z\times U_i
\right)$ is quasi-projective over $U_i$, and hence quasi-projective over $k$.

Let us show that $V_i$ is irreducible. 
Since $U_i$ is irreducible, it is enough to show that 
$q_i$ is flat and every geometric fibre is irreducible. 
By \cite[Theorem 1.3]{Poc25}, the fibrewise trivial family $\cZ_i \to U_i$ becomes trivial after a suitable \'etale base change. 
Hence there exists an \'etale surjective morphism $U'_i \to U_i$ such that 
\[
V_i \times_{U_i} U'_i \simeq \Isom_{U_i}
\left(
\cZ_i,
Z\times U_i
\right) \times_{U_i} U'_i 
=  \Isom_{U'_i}(\cZ_i \times_{U_i} U'_i,
Z\times U'_i) 
\]
\[
\simeq 
\Isom_{U'_i}( Z \times U'_i,
Z\times U'_i) \simeq \Aut(Z) \times_k U'_i
\]
and $q_i \times_{U_i} U'_i:\Aut(Z) \times_k U'_i \to U'_i$ is the second projection. 
Therefore, $q_i$ is flat and every geometric fibre of $q_i$ is integral by (1). 
Therefore, $V_i$ is irreducible. 

Over $V_i =\Isom_{U_i}
\left(
\cZ_i,
Z\times U_i
\right)$, there exists
the universal isomorphism
$\theta_i : \cZ_i\times_{U_i}V_i
\xrightarrow{\simeq}
Z\times V_i$. 
Thus, after replacing $\cY_i$ by its pullback to $V_i$, we obtain a
morphism
\[
\beta_i\colon
\cY_i\times_{U_i}V_i
\to  
\cZ_i\times_{U_i}V_i \xrightarrow{\simeq, \theta_i}
Z\times V_i 
\]
whose restriction to every geometric fibre is the contraction as 
in {\rm (3)}.

Set $V
:=
V_1\times V_2.$ Then $V$ is irreducible. 
Taking the base change $(-) \times_{V_i} V$,  we
obtain the induced morphism 
$\beta_{i,V}\colon
\cY_{i,V}
\longrightarrow
Z\times V$ for each $i\in\{1,2\}$. 
Set 
\[
\cX
:=
\cY_{1,V}
\times_{Z\times V}
\cY_{2,V}.
\]
Then the assertion holds by taking a suitable open subset $V^{\circ} \subset V$. 
\qedhere
\end{proof}

\begin{rem}\label{r Aut integral}
We shall apply the above lemma for the case when 
$Z \in \{ \P^3, \P^2 \times \P^1, Q\}$. 
In this case, $\Aut(Z)$ is irreducible and reduced as follows: 
\begin{itemize}
\item $\P^3, \P^2 \times \P^1$: 
\cite[Theorem 1 and Theorem 2(ii)]{LL22}. 
\item $Q$: \cite[Th\'eor\`eme 1, Corollaire in page 182]{Dem77}. 
\end{itemize}
\end{rem}

\begin{prop}\label{p rho=3 disjoint centre}
Let
\[
\text{x-yz}
\in
\{
\text{3-6, 3-10, 3-12, 3-14, 3-15, 3-19, 3-20, 4-5}
\}.
\]
Then there exists a parameter space of No.\ x-yz.
\end{prop}

\begin{proof}
The blowup descriptions and the intermediate Fano threefolds are given
as follows:
\[
\begin{array}{c|c|c|c}
\text{No.}
&
Z
&
(x_1\text{-}y_1z_1,\ x_2\text{-}y_2z_2)
&
\text{blowup centres on }Z
\\
\hline
\text{3-6}
&
\P^3
&
(\text{2-25},\text{2-33})
&
\text{an elliptic quartic and a line}
\\
\text{3-10}
&
Q
&
(\text{2-29},\text{2-29})
&
\text{two conics}
\\
\text{3-12}
&
\P^3
&
(\text{2-27},\text{2-33})
&
\text{a rational cubic and a line}
\\
\text{3-14}
&
\P^3
&
(\text{2-28},\text{2-35})
&
\text{a plane cubic and a point}
\\
\text{3-15}
&
Q
&
(\text{2-29},\text{2-31})
&
\text{a conic and a line}
\\
\text{3-19}
&
Q
&
(\text{2-30},\text{2-30})
&
\text{two points}
\\
\text{3-20}
&
Q
&
(\text{2-31},\text{2-31})
&
\text{two lines}
\\
\text{4-5}
&
\P^2\times\P^1
&
(\text{3-28},\text{3-21})
&
\text{curves of bidegree }(0,1)\text{ and }(1,2).
\end{array}
\]
More precisely, by the blowup descriptions in
\cite[Subsections 7.2, 7.3, and 7.4]{FanoIV}, every Fano threefold $X$ of No.\ x-yz 
is isomorphic to
\[
Y_1
\times_Z
Y_2,
\]
where $Y_i$ is a Fano threefold of No.\
$x_i$-$y_iz_i$ appearing in the corresponding row of the table, and $Y_i \to 
Z$ is the blowup along the corresponding centre. 

By Theorem \ref{t parameter rho 2} (2-$y_iz_i$), 
Proposition
\ref{p rho=3 fixed lin sys} (3-21), and Proposition \ref{p unique easy} (3-28), there
exists a parameter space of No.\ $x_i$-$y_iz_i$ for every number
appearing in the third column. Moreover, the automorphism scheme of each
possible threefold
\[
Z
\in
\{
\P^3,
Q,
\P^2\times\P^1
\}
\]
is integral by Remark \ref{r Aut integral}.

It follows from the lists of extremal rays in
\cite[Subsections 7.2, 7.3, and 7.4]{FanoIV} that, for every intermediate
Fano threefold $Y_i$ in the table, there exists a unique extremal ray 
whose contraction has target isomorphic to $Z$. Indeed, the relevant
contractions are the unique contractions
\[
\begin{array}{c|c}
\text{No.}
&
\text{target}
\\
\hline
\text{2-25, 2-27, 2-28, 2-33, 2-35}
&
\P^3
\\
\text{2-29, 2-30, 2-31}
&
Q
\\
\text{3-21, 3-28}
&
\P^2\times\P^1.
\end{array}
\]
Therefore, all the assumptions of Lemma
\ref{l disjoint centre} are satisfied for every row of the first table.
The assertion follows.
\end{proof}

\subsection{Proof of irreducibility ($\rho=3, 4$)}


\begin{prop}\label{p rho=3 fibre blowup}
There exists a parameter space of No.\ 3-4.
\end{prop}
\begin{proof}
We use the parameter space of No.\ $2$-$18$ constructed in  the proof of Proposition
\ref{p irre double cover}. 
More precisely, there exist an irreducible
quasi-projective variety $U$, 
a smooth projective morphism $\alpha\colon
\cY
\longrightarrow
U$ and a finite surjective morphism 
$f\colon
\cY
\to 
\P^2_1\times\P^1\times U$ 
over $U$ of degree $2$ 
such that every Fano threefold of No.\ $2$-$18$ is isomorphic
to a geometric fibre of $\alpha$.

For $V := \P^2_2 \times U \xrightarrow{\pr_2} U$, 
consider the following diagram in which every square is cartesian: 
\[
\begin{tikzcd}
\cY
\arrow[d, "f"']
\arrow[dd, bend right=100, "\alpha"']
&
\cY_V := \cY \times_U V 
\arrow[l]
\arrow[d, "g"]
\arrow[dd, bend left=80, "\beta"]
\\
\P^2_1 \times \P^1 \times U
\arrow[d]
&
\P^2_1 \times \P^2_2 \times \P^1 \times U
\arrow[l]
\arrow[d]
\\
U
&
V=\P^2_2 \times U
\arrow[l, "\pr_2"']
\end{tikzcd}
\]
For the diagonal $\Delta_{\P^2} \subset \P^2_1 \times \P^2_2$, 
let $\cC \subset \cY \times_U V  = \cY_V$ be its inverse image. 
For every closed point $v = (p, u) \in \P^2_2 \times U$, 
we have 
\[
\cC_v =\cC \cap \beta^{-1}(v) 
= g^{-1}(\Delta_{\P^2} \times \P^1 \times U) \cap g^{-1}(\P^2_1 \times \{p\} \times \P^1 \times \{ u\}) = 
g^{-1}( \{ (p, p)\} \times \P^1 \times \{u\}). 
\]
Therefore, 
every Fano
threefold of No.\ $3$-$4$ is isomorphic to
$\Bl_C Y$ for some closed point $v=(p, u) \in \P^2_2 \times U$ satisfying $Y = (\cY_V)_v$ and $\cC_v =C$  \cite[Subsection 7.3]{FanoIV}. 
Then there exists a parameter space of No.\ 3-4 by 
Lemma \ref{l relative blowup parameter space}. 
\qedhere



%
\end{proof}

\begin{prop}\label{p rho=3 CI centre}
Let
$\text{x-yz}
\in
\{
\text{3-7, 3-9, 3-11, 4-2}
\}.$ 
Then there exists a parameter space of No.\ x-yz.
\end{prop}

\begin{proof}
By Proposition \ref{p relative CI centre parameter space}, 
it is enough to prove that 
every Fano threefold $X$ of 
No. x-yz 
$\in
\{
\text{3-7, 3-9, 3-11, 4-2}\}$ is isomorphic to
$\Bl_C\,Y$, 
where $Y, D_1, D_2$ are  as in the following table and 
$C$ is a smooth curve satisfying $C =D'_1 \cap D'_2$ for some 
$D'_1 \in |D_1|$ and $D'_2 \in |D_2|$. 
\[
\begin{array}{c|c|c|c}
\text{No.}
&
Y
&
D_1
&
D_2
\\
\hline
\text{3-7}
&
W
&
-\frac{1}{2}K_W
&
-\frac{1}{2}K_W
\\
\text{3-11}
&
V_7
&
-\frac{1}{2}K_{V_7}
&
-\frac{1}{2}K_{V_7}
\\
\text{3-9}
&
\P_{\P^2}
\left(
\MO_{\P^2}
\oplus
\MO_{\P^2}(2)
\right)
&
S_0+\pi^*\MO_{\P^2}(2)
&
\pi^*\MO_{\P^2}(4)
\\
\text{4-2}
&
\P_{\P^1\times\P^1}
\left(
\MO_{\P^1\times\P^1}
\oplus
\MO_{\P^1\times\P^1}(1,1)
\right)
&
S_0+\pi^*\MO_{\P^1\times\P^1}(1,1)
&
\pi^*\MO_{\P^1\times\P^1}(2,2).
\end{array}
\]
The cases 3-7 and 3-11 are settled by the blowup descriptions in
\cite[Subsection 7.3]{FanoIV}.

In what follows, we treat the  cases $3$-$9$ and $4$-$2$ simultaneously. 
For
\[
(P,\cL)
:=
\begin{cases}
\left(
\P^2,
\MO_{\P^2}(2)
\right)
&
\text{for No.\ }3\text{-}9,
\\
\left(
\P^1\times\P^1,
\MO_{\P^1\times\P^1}(1,1)
\right)
&
\text{for No.\ }4\text{-}2, 
\end{cases}
\]
we have the induced $\P^1$-bundle structure 
$\pi : Y = 
\P_P
\left(
\MO_P
\oplus
\cL
\right) \to P$. 
Let $S_0\subset Y$ be the section satisfying $\MO_Y(S_0)|_{S_0}
\simeq
\cL^{-1}$. 
By 
\cite[Proposition 4.44(3) and Proposition 5.29(3)]{FanoIV}, 
the blowup centre $C$ is a complete intersection 
of a member $D'_2 \in |D_2|$ and a $\pi$-section $T$ 
satisfying $T \cap S_0 =\emptyset$. 
For such a $\pi$-section $T$, it is enough to prove $D_1 \sim T$. 
Since both $T$ and $S_0$ are sections of $\pi$, we
have 
$T\sim
S_0+\pi^*\cM$ for some $\cM \in\Pic(P)$. 
This, together with $T \cap S_0 = \emptyset$ and 
 $\MO_Y(S_0)|_{S_0}
\simeq
\cL^{-1}$, implies $\cM \simeq \cL$. 
In particular, $T \sim S_0+\pi^*\cL =D_1$, as required. 
\qedhere



%
\end{proof}

\begin{thm}\label{t parameter rho 3}
Let x-yz be a number with $x = 3$, i.e., 
\[
\text{x-yz}
\in
\{\text{3-1, 3-2, ..., 3-31}\}.
\]
Then there exists a parameter space of No.\ x-yz.
\end{thm}

\begin{proof}
The assertion holds as follows: 
\begin{itemize}
\item 3-1:
Proposition \ref{p irre double cover}. 
\item 3-2, 3-3, 3-8, 3-17:
Proposition \ref{p irre CI}. 
\item 3-4:
Proposition \ref{p rho=3 fibre blowup}. 
\item 3-5, 3-13, 3-21:
Proposition \ref{p rho=3 fixed lin sys}. 
\item 3-6, 3-10, 3-12, 3-14, 3-15, 3-19, 3-20:
Proposition \ref{p rho=3 disjoint centre}. 
\item 3-7, 3-9, 3-11:
Proposition \ref{p rho=3 CI centre}. 
\item 3-18:
Proposition \ref{p irre 5-1}. 
\item 3-16, 3-22, 3-23, 3-24, 3-25, 3-26, 3-27, 3-28, 3-29, 3-30, 3-31:
Proposition \ref{p unique easy}. 
\end{itemize}
\end{proof}

\begin{thm}\label{t parameter rho 4}
Let x-yz be a number with $x = 4$, i.e., 
\[
\text{x-yz}
\in
\{\text{4-1, 4-2, ..., 4-13}\}.
\]
Then there exists a parameter space of No.\ x-yz.
\end{thm}

\begin{proof}
The assertion holds as follows: 
\begin{itemize}
\item 4-1:
Proposition \ref{p irre CI}. 
\item 
4-2: Proposition \ref{p rho=3 CI centre}. 
\item 4-3, 4-13:
Proposition \ref{p rho=4 tridegree}. 
\item 4-4:
Proposition \ref{p irre 5-1}. 
\item 4-5:
Proposition \ref{p rho=3 disjoint centre}.
\item 4-6, 4-7, 4-8, 4-9, 4-10, 4-11, 4-12:
Proposition \ref{p unique easy}. 
\end{itemize}



\end{proof}

\begin{thm}\label{t parameter final}
Let x-yz be an arbitrary number. 
Then there exists a parameter space of No.\ x-yz.
\end{thm}
\begin{proof}
The assertion holds as follows: 
\begin{itemize}
\item $\rho=1$: Theorem \ref{t parameter rho 1}.
\item $\rho=2$: Theorem \ref{t parameter rho 2}.
\item $\rho=3$: Theorem \ref{t parameter rho 3}.
\item $\rho=4$: Theorem \ref{t parameter rho 4}.
\item $\rho\geq5$: Theorem \ref{t parameter rho 5}.
\end{itemize}    
\end{proof}










\bibliographystyle{skalpha}
\bibliography{reference.bib}

@article{Amb23,
  author  = {Ambrosi, Emiliano},
  title   = {Specialization of {N}{\'e}ron--{S}everi groups in positive characteristic},
  journal = {Annales Scientifiques de l'\'{E}cole Normale Sup\'{e}rieure},
  volume  = {56},
  number  = {3},
  pages   = {665--711},
  year    = {2023},
  doi     = {10.24033/asens.2542},
  eprint  = {1810.06481},
  archivePrefix = {arXiv}
}

@article {BFT22,
    AUTHOR = {De Biase, Lorenzo and Fatighenti, Enrico and Tanturri, Fabio},
     TITLE = {Fano 3-folds from homogeneous vector bundles over
              {G}rassmannians},
   JOURNAL = {Rev. Mat. Complut.},
  FJOURNAL = {Revista Matem\'atica Complutense},
    VOLUME = {35},
      YEAR = {2022},
    NUMBER = {3},
     PAGES = {649--710},
      ISSN = {1139-1138,1988-2807},
   MRCLASS = {14J45 (14E30 14J30 14M15)},
  MRNUMBER = {4482268},
MRREVIEWER = {Guolei\ Zhong},
       DOI = {10.1007/s13163-021-00401-2},
       URL = {https://doi-org.kyoto-u.idm.oclc.org/10.1007/s13163-021-00401-2},
}

@article {CT20,
    AUTHOR = {Cascini, Paolo and Tanaka, Hiromu},
     TITLE = {Relative semi-ampleness in positive characteristic},
   JOURNAL = {Proc. Lond. Math. Soc. (3)},
  FJOURNAL = {Proceedings of the London Mathematical Society. Third Series},
    VOLUME = {121},
      YEAR = {2020},
    NUMBER = {3},
     PAGES = {617--655},
      ISSN = {0024-6115,1460-244X},
   MRCLASS = {14C20 (14G17)},
  MRNUMBER = {4100119},
MRREVIEWER = {Andreas\ H\"{o}ring},
       DOI = {10.1112/plms.12323},
       URL = {https://doi.org/10.1112/plms.12323},
}

@book{CDL25,
  author    = {Cossec, Fran{\c{c}}ois and Dolgachev, Igor and Liedtke, Christian},
  title     = {Enriques Surfaces I},
  edition   = {2},
  publisher = {Springer Singapore},
  year      = {2025},
  pages     = {xxi+681},
  isbn      = {978-981-96-1214-7},
  doi       = {10.1007/978-981-96-1214-7}
}

@article{Dem77,
  author  = {Demazure, Michel},
  title   = {Automorphismes et d{\'e}formations des vari{\'e}t{\'e}s de {B}orel},
  journal = {Inventiones Mathematicae},
  volume  = {39},
  number  = {2},
  pages   = {179--186},
  year    = {1977},
  doi     = {10.1007/BF01390108}
}

@book{EH24,
  author    = {Eisenbud, David and Harris, Joe},
  title     = {The Practice of Algebraic Curves: A Second Course in Algebraic Geometry},
  series    = {Graduate Studies in Mathematics},
  volume    = {250},
  publisher = {American Mathematical Society},
  address   = {Providence, RI},
  year      = {2024},
  isbn      = {978-1-4704-7637-3},
}

@article {Fuj83semipositive,
    AUTHOR = {Fujita, Takao},
     TITLE = {Semipositive line bundles},
   JOURNAL = {J. Fac. Sci. Univ. Tokyo Sect. IA Math.},
  FJOURNAL = {Journal of the Faculty of Science. University of Tokyo.
              Section IA. Mathematics},
    VOLUME = {30},
      YEAR = {1983},
    NUMBER = {2},
     PAGES = {353--378},
      ISSN = {0040-8980},
   MRCLASS = {32L20 (14F05)},
  MRNUMBER = {722501},
MRREVIEWER = {J.\ Kajiwara},
}

@book {Har77,
    AUTHOR = {Hartshorne, Robin},
     TITLE = {Algebraic geometry},
    SERIES = {Graduate Texts in Mathematics, No. 52},
 PUBLISHER = {Springer-Verlag, New York-Heidelberg},
      YEAR = {1977},
     PAGES = {xvi+496},
      ISBN = {0-387-90244-9},
   MRCLASS = {14-01},
  MRNUMBER = {0463157},
MRREVIEWER = {Robert Speiser},
}

@article {Isk77,
    AUTHOR = {Iskovskih, V. A.},
     TITLE = {Fano threefolds. {I}},
   JOURNAL = {Izv. Akad. Nauk SSSR Ser. Mat.},
  FJOURNAL = {Izvestiya Akademii Nauk SSSR. Seriya Matematicheskaya},
    VOLUME = {41},
      YEAR = {1977},
    NUMBER = {3},
     PAGES = {516--562, 717},
      ISSN = {0373-2436},
   MRCLASS = {14J10 (14M20 14N05)},
  MRNUMBER = {463151},
MRREVIEWER = {Miles Reid},
}

@article {Isk78,
    AUTHOR = {Iskovskih, V. A.},
     TITLE = {Fano threefolds. {II}},
   JOURNAL = {Izv. Akad. Nauk SSSR Ser. Mat.},
  FJOURNAL = {Izvestiya Akademii Nauk SSSR. Seriya Matematicheskaya},
    VOLUME = {42},
      YEAR = {1978},
    NUMBER = {3},
     PAGES = {506--549},
      ISSN = {0373-2436},
   MRCLASS = {14J10 (14J30 14M20 14N05)},
  MRNUMBER = {503430},
MRREVIEWER = {Miles Reid},
}

@incollection {IP99,
    AUTHOR = {Iskovskikh, V. A. and Prokhorov, Yu. G.},
     TITLE = {Fano varieties},
 BOOKTITLE = {Algebraic geometry, {V}},
    SERIES = {Encyclopaedia Math. Sci.},
    VOLUME = {47},
     PAGES = {1--247},
 PUBLISHER = {Springer, Berlin},
      YEAR = {1999},
   MRCLASS = {14J45 (14E07 14F22 14K30)},
  MRNUMBER = {1668579},
MRREVIEWER = {Takao Fujita},
}

@misc{IKTTV5,
  author        = {Ito, Tetsushi and Kanemitsu, Akihiro and
                   Takamatsu, Teppei and Tanaka, Yuuji},
  title         = {Quintic del Pezzo threefolds in positive and mixed characteristic},
  year          = {2025},
  eprint        = {2506.15086},
  archivePrefix = {arXiv},
  primaryClass  = {math.AG},
  doi           = {10.48550/arXiv.2506.15086}
}

@misc{IKTTg12,
  author        = {Tetsushi Ito and Akihiro Kanemitsu and Teppei Takamatsu and Yuuji Tanaka},
  title         = {Fano threefolds of genus 12 with large automorphism group in positive and mixed characteristic},
  year          = {2026},
  eprint        = {2601.10106},
  archivePrefix = {arXiv},
  primaryClass  = {math.AG}
}

@article{dJ96,
  author  = {de Jong, A. J.},
  title   = {Smoothness, semi-stability and alterations},
  journal = {Publications Math{\'e}matiques de l'IH{\'E}S},
  volume  = {83},
  year    = {1996},
  pages   = {51--93},
  doi     = {10.1007/BF02698644}
}

@article{dJ97,
  author  = {de Jong, A. Johan},
  title   = {Families of curves and alterations},
  journal = {Ann. Inst. Fourier (Grenoble)},
  volume  = {47},
  number  = {2},
  year    = {1997},
  pages   = {599--621},
  doi     = {10.5802/aif.1575},
}

@article{KTprime,
  author =        {Kanemitsu, Akihiro and Tanaka, Hiromu},
  journal =       {in preparation},
  title =         {Prime Fano threefolds in positive characteristic},
  year =          {2026},
}

@article{KT25,
  author  = {Kawakami, Tatsuro and Tanaka, Hiromu},
  title   = {Weak quasi-$F$-splitting and del Pezzo varieties},
  journal = {J. Lond. Math. Soc. (2)},
  volume  = {111},
  number  = {2},
  year    = {2025},
  pages   = {Paper No. e70098},
  doi     = {10.1112/jlms.70098},
}

@article{KTLift1,
  author =        {Kawakami, Tatsuro and Tanaka, Hiromu},
  journal =       {preprint available at arXiv:2503.10236v1},
  title =         {Liftability and vanishing theorems for {F}ano threefolds in positive characteristic {I}},
  year =          {2025},
}

@article{KTLift2,
  author =        {Kawakami, Tatsuro and Tanaka, Hiromu},
  journal =       {preprint available at arXiv:2404.04764v2},
  title =         {Liftability and vanishing theorems for {F}ano threefolds in positive characteristic {II}},
  year =          {2025},
}

@article{Kaw85,
  author  = {Kawamata, Yujiro},
  title   = {Pluricanonical systems on minimal algebraic varieties},
  journal = {Invent. Math.},
  volume  = {79},
  number  = {3},
  year    = {1985},
  pages   = {567--588},
  doi     = {10.1007/BF01388524},
}

@article {Kee99,
    AUTHOR = {Keel, Se\'{a}n},
     TITLE = {Basepoint freeness for nef and big line bundles in positive
              characteristic},
   JOURNAL = {Ann. of Math. (2)},
  FJOURNAL = {Annals of Mathematics. Second Series},
    VOLUME = {149},
      YEAR = {1999},
    NUMBER = {1},
     PAGES = {253--286},
      ISSN = {0003-486X},
   MRCLASS = {14C20 (14E30)},
  MRNUMBER = {1680559},
MRREVIEWER = {Yuri G. Prokhorov},
       DOI = {10.2307/121025},
       URL = {https://doi-org.utokyo.idm.oclc.org/10.2307/121025},
}

@article {Kol91,
    AUTHOR = {Koll\'{a}r, J\'{a}nos},
     TITLE = {Extremal rays on smooth threefolds},
   JOURNAL = {Ann. Sci. \'{E}cole Norm. Sup. (4)},
  FJOURNAL = {Annales Scientifiques de l'\'{E}cole Normale Sup\'{e}rieure. Quatri\`eme
              S\'{e}rie},
    VOLUME = {24},
      YEAR = {1991},
    NUMBER = {3},
     PAGES = {339--361},
      ISSN = {0012-9593},
   MRCLASS = {14J30 (14E30 14J10)},
  MRNUMBER = {1100994},
MRREVIEWER = {Eckart Viehweg},
       URL = {http://www.numdam.org/item?id=ASENS_1991_4_24_3_339_0},
}

@book {KM98,
    AUTHOR = {Koll\'{a}r, J\'{a}nos and Mori, Shigefumi},
     TITLE = {Birational geometry of algebraic varieties},
    SERIES = {Cambridge Tracts in Mathematics},
    VOLUME = {134},
      NOTE = {With the collaboration of C. H. Clemens and A. Corti,
              Translated from the 1998 Japanese original},
 PUBLISHER = {Cambridge University Press, Cambridge},
      YEAR = {1998},
     PAGES = {viii+254},
      ISBN = {0-521-63277-3},
   MRCLASS = {14E30},
  MRNUMBER = {1658959},
MRREVIEWER = {Mark Gross},
       DOI = {10.1017/CBO9780511662560},
       URL = {https://doi-org.utokyo.idm.oclc.org/10.1017/CBO9780511662560},
}

@book{Lan83,
  author    = {Lander, Eric S.},
  title     = {Symmetric Designs: An Algebraic Approach},
  publisher = {Cambridge University Press},
  year      = {1983}
}

@article{LL22,
  author  = {Liendo, Alvaro and Lucchini Arteche, Giancarlo},
  title   = {Automorphisms of products of toric varieties},
  journal = {Mathematical Research Letters},
  volume  = {29},
  number  = {2},
  pages   = {529--540},
  year    = {2022},
  doi     = {10.4310/MRL.2022.v29.n2.a9}
}

@article {MM81,
    AUTHOR = {Mori, Shigefumi and Mukai, Shigeru},
     TITLE = {Classification of {F}ano {$3$}-folds with {$B_{2}\geq 2$}},
   JOURNAL = {Manuscripta Math.},
  FJOURNAL = {Manuscripta Mathematica},
    VOLUME = {36},
      YEAR = {1981/82},
    NUMBER = {2},
     PAGES = {147--162},
      ISSN = {0025-2611},
   MRCLASS = {14J30 (14J10)},
  MRNUMBER = {641971},
MRREVIEWER = {Mary Schaps},
       DOI = {10.1007/BF01170131},
       URL = {https://doi.org/10.1007/BF01170131},
}

@incollection {MM83,
    AUTHOR = {Mori, Shigefumi and Mukai, Shigeru},
     TITLE = {On {F}ano {$3$}-folds with {$B_{2}\geq 2$}},
 BOOKTITLE = {Algebraic varieties and analytic varieties ({T}okyo, 1981)},
    SERIES = {Adv. Stud. Pure Math.},
    VOLUME = {1},
     PAGES = {101--129},
 PUBLISHER = {North-Holland, Amsterdam},
      YEAR = {1983},
   MRCLASS = {14J30},
  MRNUMBER = {715648},
MRREVIEWER = {I. Dolgachev},
       DOI = {10.2969/aspm/00110101},
       URL = {https://doi.org/10.2969/aspm/00110101},
}

@incollection{MM85,
  author    = {Mori, Shigefumi and Mukai, Shigeru},
  title     = {Classification of {Fano} 3-folds with {$B_2 \geq 2$}, {I}},
  booktitle = {Algebraic and Topological Theories: To the Memory of Dr. Takehiko Miyata},
  editor    = {Nagata, Masayoshi},
  publisher = {Kinokuniya},
  address   = {Tokyo},
  year      = {1985},
  pages     = {496--545}
}

@article {MM03,
    AUTHOR = {Mori, Shigefumi and Mukai, Shigeru},
     TITLE = {Erratum: ``{C}lassification of {F}ano 3-folds with {$B_2\geq
              2$}'' [{M}anuscripta {M}ath. {\bf 36} (1981/82), no. 2,
              147--162; {MR}0641971 (83f:14032)]},
   JOURNAL = {Manuscripta Math.},
  FJOURNAL = {Manuscripta Mathematica},
    VOLUME = {110},
      YEAR = {2003},
    NUMBER = {3},
     PAGES = {407},
      ISSN = {0025-2611},
   MRCLASS = {14J45 (14E30 14J30)},
  MRNUMBER = {1969009},
       DOI = {10.1007/s00229-002-0336-2},
       URL = {https://doi.org/10.1007/s00229-002-0336-2},
}

@incollection{Mum70,
  author    = {Mumford, David},
  title     = {Varieties Defined by Quadratic Equations},
  booktitle = {Questions on Algebraic Varieties},
  series    = {C.I.M.E. Summer Schools},
  pages     = {29--100},
  publisher = {Edizioni Cremonese},
  address   = {Rome},
  year      = {1970},
}

@article{Poc25,
  author  = {Poczobut, Pawe{\l}},
  title   = {An algebraic variant of the {F}ischer--{G}rauert theorem},
  journal = {Mathematische Zeitschrift},
  volume  = {310},
  number  = {2},
  pages   = {33},
  year    = {2025},
  doi     = {10.1007/s00209-025-03728-4}
}

@article {Sho79a,
    AUTHOR = {\v{S}okurov, V. V.},
     TITLE = {The existence of a line on {F}ano varieties},
   JOURNAL = {Izv. Akad. Nauk SSSR Ser. Mat.},
  FJOURNAL = {Izvestiya Akademii Nauk SSSR. Seriya Matematicheskaya},
    VOLUME = {43},
      YEAR = {1979},
    NUMBER = {4},
     PAGES = {922--964, 968},
      ISSN = {0373-2436},
   MRCLASS = {14J30 (14M20)},
  MRNUMBER = {548510},
MRREVIEWER = {Miles Reid},
}

@article {Sho79b,
    AUTHOR = {\v{S}okurov, V. V.},
     TITLE = {Smoothness of a general anticanonical divisor on a {F}ano
              variety},
   JOURNAL = {Izv. Akad. Nauk SSSR Ser. Mat.},
  FJOURNAL = {Izvestiya Akademii Nauk SSSR. Seriya Matematicheskaya},
    VOLUME = {43},
      YEAR = {1979},
    NUMBER = {2},
     PAGES = {430--441},
      ISSN = {0373-2436},
   MRCLASS = {14J30},
  MRNUMBER = {534602},
MRREVIEWER = {Werner Kleinert},
}

@misc{SP,
  author =        {{The} {Stacks Project Authors}},
  howpublished =  {\url{http://stacks.math.columbia.edu}},
  title =         {\itshape {S}tacks {P}roject},
}

@incollection{Str87,
  author    = {Str{\o}mme, Stein Arild},
  title     = {On parametrized rational curves in Grassmann varieties},
  booktitle = {Space Curves},
  series    = {Lecture Notes in Mathematics},
  volume    = {1266},
  pages     = {251--272},
  publisher = {Springer},
  address   = {Berlin},
  year      = {1987},
  doi       = {10.1007/BFb0078187}
}

@article{FanoI,
  author =        {Tanaka, Hiromu},
  journal =       {arXiv:2308.08121},
  title =         {Fano threefolds in positive characteristic {I}},
  year =          {2023},
}

@article{FanoII,
  author =        {Tanaka, Hiromu},
  journal =       {arXiv:2308.08122},
  title =         {Fano threefolds in positive characteristic {II}},
  year =          {2023},
}

@article{FanoIII,
  author =        {Asai, Masaya and Tanaka, Hiromu},
  journal =       {arXiv:2308.08124},
  title =         {Fano threefolds in positive characteristic {III}},
  year =          {2023},
}

@article{FanoIV,
  author =        {Tanaka, Hiromu},
  journal =       {arXiv:2308.08127},
  title =         {Fano threefolds in positive characteristic {IV}},
  year =          {2023},
}

@article{WW82, 
author = {Watanabe, Keiichi and Watanabe, Masayuki}, 
title = {The Classification of {Fano} $3$-Folds with Torus Embeddings}, 
journal = {Tokyo Journal of Mathematics}, 
volume = {5}, number = {1}, pages = {37--48}, year = {1982}, doi = {10.3836/tjm/1270215033} }

@article{Wis09,
  author  = {Wi{\'s}niewski, Jaros{\l}aw A.},
  title   = {Rigidity of the Mori cone for Fano manifolds},
  journal = {Bull. London Math. Soc.},
  volume  = {41},
  year    = {2009},
  number  = {5},
  pages   = {779--781},
  doi     = {10.1112/blms/bdp025},
}

\end{document}